\documentclass{amsart}
\usepackage[latin1]{inputenc}
\usepackage{geometry}
\usepackage[all]{xy}
\usepackage{enumerate}
\usepackage[bookmarks=true,bookmarksnumbered=true,hypertexnames=false]{hyperref}
\usepackage{amscd}
\usepackage{mathrsfs}
\usepackage{amsmath}
\usepackage{amsthm}
\usepackage{amsfonts}
\usepackage{amssymb}
\usepackage{amscd}

\newtheorem{lemma}{Lemma}[section]
\newtheorem{prop}{Proposition}[section]
\newtheorem{thm}{Theorem}[section]
\newtheorem{cor}{Corollary}[section]

\DeclareMathOperator{\Hom}{Hom}
\DeclareMathOperator{\Ext}{Ext}
\DeclareMathOperator{\Tor}{Tor}
\DeclareMathOperator{\Spec}{Spec}

\DeclareMathOperator{\id}{id}

\DeclareMathOperator{\Gal}{Gal}
\DeclareMathOperator{\crys}{crys}
\DeclareMathOperator{\cris}{cris}

\DeclareMathOperator{\et}{\acute{e}t}

\DeclareMathOperator{\logcrys}{log-crys}
\DeclareMathOperator{\Inf}{inf}
\DeclareMathOperator{\gr}{gr}
\DeclareMathOperator{\gp}{gp}
\DeclareMathOperator{\DP}{DP}

\DeclareMathOperator{\dlog}{dlog}

\begin{document}

\title[Almost \'etale extensions of Fontaine rings]{Almost \'etale extensions of Fontaine rings and log-crystalline cohomology in the semi-stable reduction case}

\author{R\'emi Shankar Lodh}
\email{remi.shankar@gmail.com}

\begin{abstract}
Let $K$ be a field of characteristic zero complete for a discrete valuation, with perfect residue field of characteristic $p>0$, and let $K^+$ be the valuation ring of $K$. We relate the log-crystalline cohomology of the special fibre of certain affine $K^+$-schemes $X=\Spec(R)$ with semi-stable reduction to the Galois cohomology of the fundamental group of the geometric generic fibre $\pi_1(X_{\bar{K}})$ with coefficients in a Fontaine ring constructed from $R$. This is based on Faltings' theory of almost \'etale extensions.\end{abstract}

\maketitle

\tableofcontents

\section*{Introduction}
The purpose of this article is to make precise the relationship between crystalline cohomology and Galois cohomology of certain Fontaine rings occuring in Faltings' approach to $p$-adic Hodge theory (\cite{fa3}, \cite{fa4}). Let us very briefly recall this approach. Let $K^+$ be a complete discrete valuation ring of fraction field $K$ of characteristic zero and perfect residue field $k$ of characteristic $p>0$. Let $X$ be a proper smooth $K^+$-scheme. One constructs a site (the ``Faltings site"), usually denoted $\mathscr{X}$, whose cohomology formalizes the idea of glueing $\pi_1(X_{\bar{K}})$-cohomology locally on $X$. One sheafifies a construction of Fontaine to obtain a sheaf of rings $\mathscr{A}_{\crys,n}$ on $\mathscr{X}$ together with transformations
\[ H^\ast(\mathscr{X},\mathbb{Z}/p^n\mathbb{Z})\otimes_{\mathbb{Z}_p}A_{\cris}\to H^\ast(\mathscr{X},\mathscr{A}_{\crys,n})\leftarrow H^\ast_{\crys}(X_k|W_n(k),\mathscr{O})\otimes_{W(k)}A_{\cris} \]
where $A_{\cris}$ is the ring of $p$-adic periods constructed by Fontaine \cite{font}, and the group on the right denotes the crystalline cohomology of the special fibre $X_k$. Then one uses Faltings' theory of almost \'etale extensions to show that the intermediate cohomology theory $H^\ast(\mathscr{X},\mathscr{A}_{\crys,n})$ almost satisfies Poincar\'e duality and K\"unneth formula, hence by standard arguments is almost isomorphic to crystalline cohomology (here the term `almost' is used in the sense of almost ring theory (\cite{fa4}, \cite{almost})). Since $X$ is smooth, the group on the left is canonically isomorphic to \'etale cohomology of $X_{\bar{K}}$ tensored with $A_{\cris}$ and compatibility with Poincar\'e duality gives a one-sided inverse to the almost defined transformation to crystalline cohomology, up to a power of an element $t\in A_{\cris}$. So after taking the projective limit and inverting $t$ we obtain almost isomorphisms, which are in fact isomorphisms.

In this article we study closely the map
\[ H^\ast_{\crys}(X_k|W_n(k),\mathscr{O})\otimes_{W(k)}A_{\cris}\to H^\ast(\mathscr{X},\mathscr{A}_{\crys,n}) \]
locally on $X$. Our main result is that this map is \emph{locally} an almost isomorphism up to $t^d$-torsion, where $d=\text{dim}(X_K)$ and $t\in A_{\cris}$ is an element which plays a role analogous to that of $2\pi i$ in the transcendental theory of periods. A similar result also holds in the case $X$ has semi-stable reduction and in the paper we usually work in this context because it is the common one in applications. See the overview below for more details.

Finally, let us mention that F. Andreatta and O. Brinon \cite{ab} have independently found similar results in the good reduction case. Their proofs, although technically different, are based on the same idea.

\subsection*{Overview}
\subsubsection*{\S 1:} We begin by reviewing the (log-) crystalline site. This is mainly to fix notation. Afterwards, we review the construction by Fontaine \cite{font} of the final object of the crystalline site of a ring of characteristic $p$ with surjective (absolute) Frobenius. Such final objects are called \emph{Fontaine rings}. We give the proof for the more general log-crystalline site. Then we give some examples of Fontaine rings due to Fontaine and Kato.

\subsubsection*{\S 2:} We first recall the almost ring theory which we will use, the key input being Faltings' Almost Purity Theorem \cite{fa4}. Afterwards we apply this theorem to certain Fontaine rings, constructed as follows. Let $\Spec(R)$ be a integral $K^+$-scheme with semi-stable reduction. Up to localizing on $\Spec(R)$ we may assume that it is \'etale over a ring of the form $K^+[T_1,...,T_{d+1}]/(T_1\cdots T_r-\pi)$, where $\pi\in K^+$ is a uniformizer. In this case, one says that $R$ is \emph{small}. Let $\Sigma_n=W_n(k)[u]\langle u^e\rangle$, where $e$ is the absolute ramification index of $K$ and the angled brackets mean that we have added divided powers of $u^e$. Then there is a surjection $\Sigma_n\to K^+/pK^+$ whose kernel is a DP-ideal. Let $\bar{R}$ be the normalization of $R$ in the maximal profinite connected \'etale covering of $R[1/p]$. Then via the theory of almost \'etale extensions one can show that the ring $\bar{R}/p\bar{R}$ has surjective Frobenius, hence by Fontaine's theorem recalled in \S 1, we may construct the Fontaine ring
\[ A^+_{\log}:=\varprojlim_n H^0_{\logcrys}(\bar{R}/p\bar{R}|\Sigma_n,\mathscr{O}) \]
(log-crystalline cohomology). Also we can construct another Fontaine ring $A^+_{\log,\infty}$ as follows. Let $\bar{K}$ denote the algebraic closure of $K$ and $\bar{K}^+$ its valuation ring, i.e. the normalization of $K^+$ in $\bar{K}$. Let $\tilde{R}=R\otimes_{K^+}\bar{K}^+$ and let $\tilde{R}_{\infty}$ denote the ring obtained from $\tilde{R}$ by adding all $p$-power roots of the the $T_i$. Define
\[ A^+_{\log,\infty}:=\varprojlim_n H^0_{\logcrys}(\tilde{R}_{\infty}/p\tilde{R}_{\infty}|\Sigma_n,\mathscr{O}). \]
Then the theory of almost \'etale extensions applied to these Fontaine rings implies that the canonical homomorphism
\[ A^+_{\log,\infty}/p^nA^+_{\log,\infty}\to A^+_{\log}/p^nA^+_{\log} \]
is the filtering inductive limit of almost Galois coverings and there are canonical almost isomorphisms for each $i$
\[ \begin{CD}
H^i(\Delta_\infty,A^+_{\log,\infty}/p^nA^+_{\log,\infty}) @>{\thickapprox}>> H^i(\Delta,A^+_{\log}/p^nA^+_{\log})
\end{CD} \]
where $\Delta:=\Gal(\bar{R}[1/p]/\tilde{R}[1/p])$ with quotient $\Delta_\infty:=\Gal(\tilde{R}_\infty[1/p]/\tilde{R}[1/p])\cong\mathbb{Z}_p(1)^d$ (see Corollaries \ref{almostgalois}, \ref{invariants0}). This also applies to $p$-adic divided power bases other than $\Sigma:=\varprojlim_n\Sigma_n$.

\subsubsection*{\S 3:} We construct, via the formalism of log-crystalline cohomology, a canonical logarithmic de Rham resolution of the ring $A^+_{\log}$. We then (almost) compute the $\Delta$-cohomology of the components of this resolution, by reducing to the case of $A^+_{\log,\infty}$ and making an explicit computation there. The result is the following (Corollary \ref{final})

\begin{thm}\label{intromain}
Let $\mathcal{R}_n$ be a log-smooth $\Sigma_n$-lift of $R/pR$, and let $\omega^\bullet_{\mathcal{R}_n/\Sigma_n}$ be the logarithmic de Rham complex of $\mathcal{R}_n/\Sigma_n$. There is a canonical morphism of complexes in the derived category
\[ B^+_{\log}\otimes_{\Sigma}\omega^\bullet_{\mathcal{R}_n/\Sigma_n}\to C^{\bullet}(\Delta,A^+_{\log}/p^nA^+_{\log}) \]
which is an almost quasi-isomorphism up to $t^dx$-torsion, where $x\in B^+_{\log}$ is independent of $n$ and $R$.
\end{thm}
Here
\[ B^+_{\log}=\varprojlim_n H^0_{\logcrys}(\bar{K}^+/p\bar{K}^+|\Sigma_n,\mathscr{O}) \]
is a ring of $p$-adic periods constructed by Kato, $t\in A_{\cris}\subset B^+_{\log}$ is the element alluded to above, and we write $C^{\bullet}(\Delta,-)=\Hom_{\text{cont.},\Delta}(\Delta^{\times\bullet},-)$ for the usual functorial complex computing continuous group cohomology. The proofs of the main results of this section are straightforward, although they involve some lengthy computations.

\subsection*{Remarks on notation}
\begin{itemize}
\item For any subset $R\subset\mathbb{R}$ we write $R_+$ (resp. $R_{>0}$) for the set of elements of $R$ which are greater than or equal to zero (resp. greater than zero); $\mathbb{N}:=\mathbb{Z}_+$ is the set of natural numbers.
\item By \emph{ring} we mean a commutative ring with unity. For any ring $A$, we denote by $A^{\ast}$ its group of units. For any integral domain $A$, we denote by $Q(A)$ its fraction field.
\item All monoids considered will be assumed to be with commutative. To a monoid $M$ we write $M^{\gp}$ for the associated group.
\item If $A$ is a ring and $M$ is an $A$-module, then we denote by $\Gamma_A(M)$ the divided power polynomial $A$-algebra defined by $M$ (see \cite{berthelot} or \cite{bertogus} for a construction of this algebra). If $I\subset A$ is an ideal then we denote by $D_A(I)$ the divided power hull of $A$ for the ideal $I$ (\emph{loc.cit.}). If $A$ is a ring and $I\subset A$ is an ideal with a divided power structure $(\gamma_n:I\to A)_{n\in\mathbb{N}}$, then we will often write $x^{[n]}:=\gamma_n(x)$ when it is clear which divided power structure is meant. Finally, if $X_1,...,X_d$ denotes indeterminates, then we write $A\left\langle X_1,...,X_d\right\rangle $ for the divided power polynomial $A$-algebra in the variables $X_1,...,X_d$. It is the divided power hull of  $A[X_1,...,X_d]$ for the ideal generated by $X_1,...,X_d$.
\item If $G$ is a profinite group and $M$ is a discrete $G$-module, then we write $H^i(G,M)$ for the Galois cohomology groups (i.e. continuous group cohomology).
\end{itemize}

\section{Crystalline cohomology of rings with surjective Frobenius}
In this section we recall Fontaine's construction of the final object of the crystalline site of a ring of characteristic $p>0$ with surjective Frobenius endomorphism. This will play the role of substitute for Poincar\'e's lemma in our approach.
\subsection{Reminder on crystalline sites}
\subsubsection{}
Let $Z$ be a scheme and let $Z_0\hookrightarrow Z$ be a closed immersion, such that the ideal defining the image of $Z_0$ in $Z$ is nilpotent and has a divided power structure (we say that $Z_0\hookrightarrow Z$ is a \emph{divided power thickening}). We usually use the abbreviation DP for ``divided power". Suppose that $T\to Z$ (resp. $T'\to Z$) is a morphism and $U\hookrightarrow T$ (resp. $U'\hookrightarrow T'$) is a DP-thickening such that the morphism $T\to Z$ (resp. $T'\to Z$) is a DP-morphism (\cite[1.9.4]{berthelot}) for the given DP-structures. A morphism of $Z$-schemes $T\to T'$ is called a \emph{$Z$-DP-morphism} if it is a DP-morphism (for the given DP-structures) and such that the diagram
\[ \begin{CD}
T @>>> T' \\
@VVV @VVV \\
Z @= Z
\end{CD} \]
is a commutative diagram of DP-morphisms.

\subsubsection{}\label{logstrdef}
All log-structures in this paper will be in the \'etale topology. Let $(X,M)\to (Z_0,N_0)$ be a morphism of log-schemes (see \cite{log}). If $f:X\to Z_0$ is the underlying morphism, then we write $f^\ast N_0$ for the inverse image log-structure, in contrast with the inverse image sheaf $f^{-1}N_0$. If $M$ is a pre-log-structure on $X$ then we denote by $M^a$ the associated log-structure. If $X$ is a log-scheme and $U\to X$ is an \'etale morphism of schemes, then the restriction $M_U$ of the log-structure $M$ of $X$ to $U$ defines a natural log-structure on $U$ and we will always consider $U$ as a log-scheme for this log-structure.

\subsubsection{} Let $(X,M)\to (Z_0,N_0)$ be a morphism of log-schemes and assume that the log-structure $M$ is integral. Let $(Z_0,N_0)\hookrightarrow (Z,N)$ be an exact closed immersion such that $Z_0\hookrightarrow Z$ is a DP-thickening. The \emph{logarithmic crystalline site} of $(X,M)$ over the DP-log-base $(Z,N)$ (see \cite[2.4]{kato}) is the site whose underlying category has for objects exact closed immersions $(U,M_U)\hookrightarrow T$ of log-schemes, where:
\begin{itemize}
\item $U\to X$ is an \'etale morphism of schemes and the log-structure $M_U$ is as in (\ref{logstrdef})
\item $T$ is a log-$(Z,N)$-scheme with integral log-structure such that the morphism of schemes underlying the structure morphism $T\to Z$ is a DP-morphism
\item the morphism of schemes underlying $(U,M_U)\hookrightarrow T$ is a DP-thickening.
\end{itemize}
We usually write $U\hookrightarrow T$ instead of $(U,M_U)\hookrightarrow T$. A morphism of this category is a commutative diagram of log-schemes
\[ \begin{CD}
U' @>>> T' \\
@VVV @VVV \\
U @>>> T
\end{CD} \]
where morphism on the left is a morphism of log-$(X,M)$-schemes, and the morphism on the right is a morphism of log-$(Z,N)$-schemes such that the underlying morphism of schemes is a $Z$-DP-morphism. The pretopology on this category is given by defining covering families to be families of morphisms
\[ (U_\alpha\hookrightarrow T_\alpha)_\alpha\rightarrow (U\hookrightarrow T) \]
such that the morphisms of schemes underlying $(U_\alpha\to U)_\alpha$ and $(T_\alpha\to T)_\alpha$ are coverings for the \'etale topology, and such that the squares
\[ \begin{CD}
U_\alpha @>>> T_\alpha \\
@VVV @VVV \\
U @>>> T
\end{CD} \]
are cartesian. Given such a covering, for any morphism $(U'\hookrightarrow T')\to (U\hookrightarrow T) $, note that the diagram
\[ \begin{CD}
U'\times_U U_\alpha @>>> T'\times_T T_\alpha \\
@VVV @VVV \\
U' @>>> T'
\end{CD} \]
is cartesian, hence $(U'\times_U U_\alpha\hookrightarrow T'\times_T T_\alpha)$ is an object of the category underlying the logarithmic crystalline site. This shows that the \emph{log-crystalline site of $X$ over the DP-base $Z$} is indeed a site denoted $((X,M)|(Z,N))_{\crys}$ or simply by $(X|Z)_{\logcrys}$ when it is clear which log-structures are meant for $X$ and $Z$.

\subsubsection{}
To give a sheaf $\mathscr{F}$ on $(X|Z)_{\logcrys}$ is the same as giving for all $(U\hookrightarrow T)\in\text{ob}(X|Z)_{\logcrys}$ a sheaf $\mathscr{F}_T$ on the \'etale site of $T$, together with a morphism
\[ g^{\ast}_{\mathscr{F}}:g^{-1}\mathscr{F}_T\rightarrow \mathscr{F}_{T'} \]
for any morphism $g:(U'\hookrightarrow T')\to (U\hookrightarrow T)$, such that the natural transitivity condition holds for morphisms $(U''\hookrightarrow T'')\to (U'\hookrightarrow T') \to (U\hookrightarrow T)$ and moreover $g^{\ast}_{\mathscr{F}}$ is an isomorphism if the square defined by $g$ is cartesian.

In this way we see that the presheaf defined
\[ \mathscr{O}(U\hookrightarrow T):=\mathscr{O}_T(T) \]
is in fact a sheaf, called the \emph{structure sheaf} of $(X|Z)_{\logcrys}$. A \emph{quasi-coherent} $\mathscr{O}$-module on $(X|Z)_{\logcrys}$ is an $\mathscr{O}$-module $\mathscr{E}$ such that for each object $U\hookrightarrow T$ of $(X|Z)_{\logcrys}$, the restriction $\mathscr{E}_T$ is a quasi-coherent $\mathscr{O}_T$-module in the usual sense.

\subsubsection{}
If $X=\Spec(R)$ and $Z=\Spec(S)$ are affine schemes, then we will usually write $(R|S)_{\logcrys}$ instead of $(X|Z)_{\logcrys}$.

\subsubsection{}
Assume in the sequel that $Z$ is annihilated by a power of $p$ and that $(Z,N)$ is a fine log-scheme. By \cite[2.4.2]{kato}, if $f:(X,M)\to (X',M')$ is a morphism of log-$(Z_0,N_0)$-schemes with $M$ integral and $M'$ fine, then $f$ induces a morphism of log-crystalline topoi over the DP log-base $(Z,N)$.

\subsubsection{}\label{limits}
Given a projective system of fine log-$(Z_0,N_0)$-schemes $\left\lbrace (X_i,M_i)\right\rbrace_{i\in I} $ with affine transition morphisms and with projective limit $(X,M)$, then by \cite[2.4.3]{kato} we have
\[ H^j_{\logcrys}(X|Z,\mathscr{O})\cong\varinjlim_{i\in I}H^j_{\logcrys}(X_i|Z,\mathscr{O}) \]
for all $j$.

\subsection{Fontaine's theorem}

\subsubsection{}
Assume $S$ is a ring such that $p^mS=0$ for some $m\in\mathbb{N}$, and $J\subset S$ an ideal containing $p$ with a divided power structure $\gamma:J\to S$. Let $R$ be a $S/J$-algebra. Let $N$ be an integral monoid defining a log-structure on $\Spec(S)$ and $M$ an integral monoid defining a log structure on $\Spec(R)$ and a map $N\to M$ such that we have a morphism of log-schemes $(\Spec(R),M^a)\to (\Spec(S),N^a)$. Then we may consider the log-crystalline site $(R|S)_{\logcrys}$.

For any monoid $M$, the endomorphism
\[ M\to M:m\mapsto m^p \]
is called the \emph{Frobenius} of $M$.

\begin{thm}\label{fontaine}
With the above notation and assumptions. Assume that the (absolute) Frobenius is surjective on $R$ and on $M$. If $M$ is a saturated log-structure, then the site $(R|S)_{\logcrys}$ has a final object.
\end{thm}
We give two proofs, the first generalizing that of Fontaine in the classical case and the second following Breuil \cite{breuil}.
\begin{proof}[First proof]
The proof is constructive and proceeds in several steps.

\underline{Step 1:} We construct a candidate for the final object. First note that for any \'etale map $U\to \Spec(R)$, the (absolute) Frobenius is surjective on $\mathscr{O}_U$. Indeed, since the map is \'etale, its relative Frobenius is an isomorphism, so this follows from the factorization of the absolute Frobenius of $U$ as the relative Frobenius followed by the pullback of the absolute Frobenius of $\Spec(R)$. We first define a perfect ring $P(R)$ as being the projective limit of the diagram
\[ \begin{CD}
\cdots @>F>> R @>F>> R @>F>> R
\end{CD} \]
where $F$ denotes the (absolute) Frobenius of $R$. An element of $P(R)$ is given by a sequence $r=(r^{(n)})$ of elements of $R$ indexed by the natural numbers, such that $r^{(n+1)p}=r^{(n)}$ for all $n$. $P(R)$ is a perfect ring of characteristic $p$, so its ring of Witt vectors $W(P(R))$ is a flat $\mathbb{Z}_p$-algebra. We write $(r_0,r_1,r_2,...)\in W(P(R))$ and $r_i=(r_i^{(n)})$ for each $i=0,1,2,...$. Let $\Spec(R')\hookrightarrow\Spec(A)$ be an object of the site $(R|S)_{\logcrys}$. If $r=(r^{(n)})\in P(R)$ and $g:R\to R'$ is structure homomorphism, then we choose arbitrary lifts $\hat{r}^{(n)}\in A$ of $g(r^{(n)})$ for all $n$ and set
\[ \tilde{r}^{(m)}:=\lim_{n\to\infty}\hat{r}^{(m+n)p^n}\qquad\text{for all}\quad m. \]
Since $p$ is nilpotent on $A$ and $I:=\ker(A\to R')$ has a DP-structure, one sees easily that $ \tilde{r}^{(m)}$ is a lift of $g(r^{(m)})$ which is independent of the choices made. Define a map
\[ \theta_A: W(P(R))\rightarrow A \]
by sending $(r_0,r_1,...)$ to $\sum_{i=0}^{\infty}p^i\tilde{r}_i^{(i)}$. Since $p$ is nilpotent on $A$, these are just the usual Witt polynomials, so the map is indeed a homomorphism of rings. In the case $A=R$, the map $\theta_R$ is none other than the projection $(r_0,...)\mapsto r_0^{(0)}$ and in this way we obtain a commutative diagram
\[ \begin{CD}
W(P(R)) @>{\theta_A}>> A \\
@VVV  @VVV \\
R @>{g}>> R'.
\end{CD} \]
We claim that the map $\theta_A$ is unique for the maps $W(P(R))\to A$ making the above diagram commute. Indeed, any map $\alpha:W(P(R))\to A$ is determined by it values on $[r]$, where $[\cdot]$ denotes the Teichm\"uller lift. If $r=(r^{(n)})$, then write $r(m):=(r^{(m+n)})$ as the sequence ``shifted" by $m$. So for all $m$ we have
\[ \alpha([r])=\alpha([r(m)])^{p^m} \]
and $\alpha([r(m)])=\theta_A([r(m)])+a$ for some $a\in\ker(A\to R)$. If $p^mA=0$, then 
\begin{eqnarray*}
\alpha([r])=\alpha([r(m)])^{p^m}=(\theta_A([r(m)])+a)^{p^m}&=&\sum_{i=0}^{p^m}{p^m\choose i}i!a^{[i]}\theta_A([r(m)])^{p^m-i}\\&=&\theta_A([r(m)])^{p^m}=\theta_A([r])
\end{eqnarray*}
thus proving the uniqueness claim.

We may extend this map to a unique homomorphism of $S$-algebras
\[ \theta_{A,S}:W(P(R))\otimes_\mathbb{Z}S\rightarrow A \]
thereby obtaining a commutative diagram
\[ \begin{CD}
W(P(R))\otimes_\mathbb{Z}S @>{\theta_{A,S}}>> A \\
@VVV  @VVV \\
R @>{g}>> R'.
\end{CD} \]
Define an integral monoid $P(M)$ by
\[ P(M):=\left\lbrace (m^{(n)})\in M^{\mathbb{N}}\: :\: m^{(n+1)p}=m^{(n)},\:\forall n\in\mathbb{N}\right\rbrace . \]
The Teichm\"uller lift defines a map $P(M)\to W(P(R))$. By assumption, the natural map $P(M)\oplus N\to M$ induced from the projection $P(M)\to M$ sending $(m^{(n)})$ to $m^{(0)}$, is surjective. Let $Q:=P(M)\oplus N$ and define
\[ L:=\ker(Q^{\gp}\rightarrow M^{\gp}). \]
Define
\[ \left( W(P(R))\otimes_\mathbb{Z}S\right)_{\log}:= W(P(R))\otimes_\mathbb{Z}S\otimes_{\mathbb{Z}}\mathbb{Z}[L] . \]
Let us show how to extend the map $\theta_{A,S}$ to a unique map
\[ \theta_{A,S}^{\log}:\left( W(P(R))\otimes_\mathbb{Z}S\right)_{\log}\rightarrow A. \]
Let $M_A$ (resp. $M_{R'}$) denote the log-structure of $\Spec(A)$ (resp. $\Spec(R')$). By Lemma \ref{nilexact} (ii) we have
\[ M_{R'}(R')=M_A(A)/(1+I) \]
where $I=\ker(A\to R')$. Define a map $P(M)\to M_A(A)$ by sending $(m^{(n)})_{n\in\mathbb{N}}$ to $\hat{m}^{(n)p^n}$, where $\hat{m}^{(n)}$ is a global section of $M_A$ lifting the image of $m^{(n)}$ in $M_{R'}$ and $n$ is any integer such that $p^nA=0$. Because of the divided power structure of $I$ it is easy to check that this section is independent of all choices. The map $P(M)\to M_A(A)$ extends uniquely to a map $\lambda_A:Q=P(M)\oplus N\to M_A(A)$. Now if $l\in L$, consider $\lambda_A(l)\in M_A(A)^{\gp}$. Since the image of $\lambda_A(l)$ in $M_{R'}(R')^{\gp}$ is the identity element, by exactness we deduce that $\lambda_A(l)\in A^\ast$. So $\lambda_A$ extends uniquely to a map $L\to M_A(A)$. This defines the map
\[ \theta_{A,S}^{\log}:\left( W(P(R))\otimes_\mathbb{Z}S\right)_{\log}\rightarrow A. \]
In the special case $R'=R=A$ we see that the map $\theta_{R,S}^{\log}$ is none other that induced by the ``projection" map $\theta_{R,S}$ and the natural map $L\to M:l\mapsto 1$. It follows from the above that the map $\alpha=\theta_{A,S}^{\log}$ is the unique map for which the diagram
\[ \begin{CD}
\left( W(P(R))\otimes_\mathbb{Z}S\right)_{\log} @>{\alpha}>> A \\
@V{\theta_{R,S}^{\log}}VV  @VVV \\
R @>>> R'.
\end{CD} \]
commutes in a way compatible with the commutative diagram of monoids
\[ \begin{CD}
Q=P(M)\oplus N @>{\lambda_A}>> M_A(A) \\
@VVV @VVV \\
M @>>> M_{R'}(R').
\end{CD} \]
By the universal property of divided power hulls, $\theta_{A,S}^{\log}$ factors over a homomorphism
\[ \left( W(P(R))\otimes_\mathbb{Z}S\right)_{\log}^{\text{DP}}\to A \]
where $\left( W(P(R))\otimes_\mathbb{Z}S \right)_{\log}^{\text{DP}}$ is the divided power hull of $\left( W(P(R))\otimes_\mathbb{Z}S\right) _{\log}$ for the ideal $\ker(\theta_{R,S}^{\log})$ defined above, compatible with $\gamma:J\to S$ (for the construction of divided power hulls see \cite[Ch. I, \S2.3]{berthelot} or \cite[3.19 Theorem]{bertogus}).

We can view $W(P(R))$ canonically as a $\mathbb{Z}[P(M)]$-algebra via the Teichm\"uller lift of the map $P(M)\to P(R)$. Hence $W(P(R))\otimes_{\mathbb{Z}}S$ is canonically a $\mathbb{Z}[Q]=\mathbb{Z}[P(M)]\otimes_{\mathbb{Z}}\mathbb{Z}[N]$-algebra. Thus $\left(W(P(R))\otimes_\mathbb{Z}S\right)_{\log}$ is canonically a $\mathbb{Z}[Q\oplus L]=\mathbb{Z}[Q]\otimes_{\mathbb{Z}}\mathbb{Z}[L]$-algebra, and so is $\left( W(P(R))\otimes_\mathbb{Z}S\right)_{\log}^{\text{DP}}$. Let $\overline{\ker(\theta_{R,S}^{\log})}$ be the DP-ideal of $\left( W(P(R))\otimes_\mathbb{Z}S\right)_{\log}^{\text{DP}}$ generated by the divided powers of all elements of $\ker(\theta_{R,S}^{\log})$, and $K\subset\overline{\ker(\theta_{R,S}^{\log})}$ be the sub-DP-ideal generated by the divided powers of all elements of the form $q_l\otimes l-p_l\otimes 1$, where $p_l,q_l\in Q$ and $l=p_l/q_l\in L$. Define
\[ B:=\left( W(P(R))\otimes_\mathbb{Z}S\right)_{\log}^{\text{DP}}/K. \]
Note that by \cite[Ch. I, Prop. 1.6.2]{berthelot} the ideal $\ker(B\to R)$ has a canonical DP-structure, compatible with $\gamma:J\to S$. Moreover, the map $\theta_{A,S}^{\log}$ factors uniquely over $B$, since the ideal $K$ maps to zero in $A$. The closed immersion
\[ \Spec(R)\hookrightarrow \Spec(B) \]
is our candidate for final object.

\underline{Step 2:} We define the log-structure on $\Spec(B)$ as follows. Let $QL$ be the submonoid of $Q^{\gp}$ consisting of elements of the form $ql$ with $q\in Q$, $l\in L$. The map factors over a map $QL\to B$. We define the log-structure on $B$ as that associated to $QL\to B$.

We claim that this log-structure makes the closed immersion
\[ \Spec(R)\hookrightarrow \Spec(B) \]
exact. To verify this claim we use the criterion of Lemma \ref{nilexact} (i). This criterion is local for the \'etale topology, so let $\bar{x}\to\Spec(R)$ be a geometric point, and $R_{\bar{x}}$, $B_{\bar{x}}$  the strict localizations of $R$ and $B$ respectively at $\bar{x}$. In this case the claim will follow from Lemma \ref{app:exactlogstr} if we can show:
\begin{enumerate}[(a)]
\item the image of $L$ in $((QL)^a)^{\gp}$ lies in $B^{\ast}$
\item the map $(QL)^{\gp}\cap B^{\ast}_{\bar{x}}\to M^{\gp}\cap R^{\ast}_{\bar{x}}$ is surjective.
\end{enumerate}
Since (a) is clear by construction, let us show (b). Suppose that $m=\dfrac{m_1}{m_2}\in M^{\gp}\cap R^{\ast}_{\bar{x}}$ with $m_i\in M$ for $i=1,2$. Since the Frobenius is surjective on $M$, for $i=1,2$ we can choose a sequence $\underline{m_i}:=(m_i=m_i^{(0)},m_i^{(1)},m_i^{(2)},...,m_i^{(n)},...)$ of elements of $M$ such that $m_i^{(n+1)p}=m_i^{(n)}$ for all $n\in\mathbb{N}$. If we define $m^{(n)}:=\dfrac{m_1^{(n)}}{m_2^{(n)}}$, then we have $m^{(n)p^n}=m$ for all $n\in\mathbb{N}$. Since the log structure $M^a$ is saturated, this implies that $m^{(n)}\in M^a_{\bar{x}}$ for all $n\in\mathbb{N}$. But then if $\alpha:M^a\to R$ is the structure morphism of the log structure, we have $\alpha(m^{(n)})^{p^n}\in R^{\ast}_{\bar{x}}$, whence $\alpha(m^{(n)})\in R^{\ast}_{\bar{x}}$, so $m^{(n)}\in R^{\ast}_{\bar{x}}\subset M^a_{\bar{x}}$ for all $n\in\mathbb{N}$. So we can naturally identify $\underline{m}:=\underline{m}_1/\underline{m}_2$ with an element of $P(R_{\bar{x}})^{\ast}$ and, via the Teichm\"uller map, with an element of $W(P(R_{\bar{x}}))^{\ast}$. Now let $\theta_{B_{\bar{x}}}:W(P(R_{\bar{x}}))\to B_{\bar{x}}$ be the homomorphism constructed in Step 1. We claim that the map $P(M)\to B_{\bar{x}}$ factors $P(M)\to W(P(R_{\bar{x}}))\to B_{\bar{x}}$. This follows from the commutative diagram
\[ \begin{CD}
W(P(R)) @>{\theta_B}>>  B \\
@VVV @VVV \\
W(P(R_{\bar{x}})) @>{\theta_{B_{\bar{x}}}}>>  B_{\bar{x}}.
\end{CD} \]
So in $\left(QL\right)^a$ we have $\underline{m}_1=\theta_{B_{\bar{x}}}([\underline{m}])\cdot\underline{m}_2$, i.e. $\underline{m_1}/\underline{m_2}\in P(M)^{\gp}\cap B^{\ast}_{\bar{x}}$, which implies (b). Thus, the closed immersion $\Spec(R)\hookrightarrow\Spec(B)$ is exact, hence $\Spec(R)\hookrightarrow\Spec(B)$ is an object of the site $(R|S)_{\logcrys}$.

\underline{Step 3:} We have shown that if $\Spec(R')\hookrightarrow\Spec(A)$ is an $S$-DP-thickening we have unique morphisms of $(R|S)_{\logcrys}$
\[ \begin{CD}
\left( \Spec(R')\hookrightarrow \Spec(A)\right) \\
@VVV \\
\left( \Spec(R)\hookrightarrow \Spec(B)\right)
\end{CD} \]
and hence for any object $U\hookrightarrow T$ and any affine covering $(U_\alpha\hookrightarrow T_\alpha)_\alpha \to (U\hookrightarrow T)$ we may define unique morphisms
\[ (U_\alpha \hookrightarrow T_\alpha)\to \left( \Spec(R)\hookrightarrow \Spec(B)\right)  \]
and since the uniqueness allows us to glue, we are done.
\end{proof}

Now we sketch the second proof, \emph{cf.} \cite[\S 4-5]{breuil}.
\begin{proof}[Second proof (sketch)]
Suppose $p^mS=0$. For any $S$-DP-thickening $\Spec(R)\hookrightarrow \Spec(A)$, define a map
\[ \theta_A:W_m(R)\to A \]
by sending $(r_0,...,r_{m-1})$ to $\sum_{i=0}^{m-1}p^i\hat{r}_i^{p^{m-i}}$, where $\hat{r}_i$ denotes an arbitrary lift of $r_i\in R$ to $A$. Because of the divided power structure of $\ker(A\to R)$, this is a well-defined homomorphism of rings. In this way we obtain a commutative square, and to check the uniqueness of $\theta_A$ for such commutative squares, we first reduce to checking it for Teichm\"uller lifts and then use the fact that the Frobenius is surjective on $R$. Let $Q:=M\oplus N$, $L=\ker(\lambda^{\gp}:Q^{\gp}\to M^{\gp})$, where the map $\lambda:Q\to M$ is induced by the maps $M\to M:m\mapsto m^{p^n}$ and the structure map $N\to M$. We view $W_m(R)\otimes_\mathbb{Z}S$ as a $\mathbb{Z}[Q]$-algebra via the map
\[ M\oplus N\to W_m(R)\otimes_\mathbb{Z}S:(r,n)\mapsto\hat{r}^{p^m}\otimes n. \]
Let $QL$ be the submonoid of $Q^{\gp}$ consisting of elements of the form $ql$ with $q\in Q$, $l\in L$. Then the map $\lambda$ extends uniquely to $\lambda':QL\to M:L\ni l\mapsto 1$. Set
\[ \left( W_m(R)\otimes_\mathbb{Z}S\right)_{\log}:=(W_m(R)\otimes_\mathbb{Z}S)\otimes_{\mathbb{Z}}\mathbb{Z}[L]. \]
Take the divided power hull $\left( W_m(R)\otimes_\mathbb{Z}S\right)_{\log}^{\text{DP}}$ for the kernel of the unique surjection onto $R$ extending the $S$-linearization $\theta_{R,S}$ of $\theta_R$ by $\lambda'$. Define
\[ B=\left( W_m(R)\otimes_\mathbb{Z}S\right)_{\log}^{\text{DP}}/\mathscr{K} \]
where $\mathscr{K}$ is the DP-ideal generated by the divided powers of $q_l\otimes l-p_l$, $l\in L$, where $q_l,p_l\in Q$ and $l=p_l/q_l\in L$. Then there is a map $QL\to B$ and we endowing $B$ with the log-structure associated to this map, thereby obtaining an object
\[ \Spec(R)\hookrightarrow \Spec(B) \]
of the site $(R|S)_{\logcrys}$. Then we claim that this is the final object: one can verify the universal property as in the first proof.
\end{proof}

\subsubsection{Remarks.}\label{remarks}
\begin{enumerate}[(a)]
\item Note that in the construction of the final object, we may replace the group $L$ by any subgroup $L'\subset L$ such that for all $l\in L$ we have $l'u=l$ for some $l'\in L'$ and $u\in P(R)^{\ast}$ in the first proof (resp. $u\in R^{\ast}$ in the second proof), and obtain a ring $B'$. Then exactly the same argument as in Step 2 of the proof of Theorem \ref{fontaine} shows that the closed immersion
\[ \Spec(R)\hookrightarrow \Spec(B') \]
is exact, hence defines an object of $(R|S)_{\logcrys}$, and it is easy to see that it is the final object.
\item The main advantage of the first proof is the uniformity of the construction. If $(S,J)\to (S/p^nS,J\cdot S/p^nS)$ is a DP-homomorphism, then it follows from \cite[3.20, Remark 8]{bertogus} that the final object of the crystalline site $(R|S/p^nS)_{\logcrys}$ is given by the $S/p^nS$-DP-thickening
\[ \Spec(R)\hookrightarrow \Spec(B)\otimes_\mathbb{Z}\mathbb{Z}/p^n\mathbb{Z}. \]
In particular, if $S$ is a $p$-adically complete ring (we say that $S$ is \textit{$p$-adic base}), then as in \cite{bertogus} one may define a crystalline site $(R|S)_{\logcrys}$ whose cohomology automatically computes the derived projective limit of the cohomology of each $(R|S/p^nS)_{\logcrys}$. Then the proof given in the theorem also works for such $S$ and shows that
\[ \Spec(R)\hookrightarrow \text{Spf}(\hat{B}) \]
is the final object of the site $(R|S)_{\logcrys}$, where the hat denotes the $p$-adic completion.
\item The existence of the final object of the site $(R|S)_{\logcrys}$ implies that the cohomology of any sheaf is canonically isomorphic to the cohomology of its restriction to the \'etale site of the final object. In particular, the crystalline cohomology of any quasi-coherent $\mathscr{O}$-module vanishes in non-zero degree and we have
\[ H^0_{\logcrys}(R|S,\mathscr{O})=B. \]
In the case $S$ is a $p$-adic base, using the previous remark we see that
\[ \varprojlim_nH^0_{\logcrys}(R|S/p^{n+1}S,\mathscr{O})=\hat{B}. \]
\item If one compares the 1st and 2nd proofs of Theorem \ref{fontaine}, then it is not difficult to see that the canonical isomorphism
\[ (W_m(P(R))\otimes_{\mathbb{Z}}S)_{\log}^{\DP}/K\cong (W_m(R)\otimes_{\mathbb{Z}}S)_{\log}^{\DP}/\mathscr{K} \]
sends the image of an element $[x]\in W_n(P(R))$, where $x=(x^{(0)},x^{(1)},...)\in P(R)$, to the image of the element $[x^{(m)}]\in W_m(R)$.
\end{enumerate}

\subsubsection{} Let $R$, $M$, and $S$ be as in Theorem \ref{fontaine}. Let $\Spec(A)$ be a fine saturated log-$S/J$-scheme and let $h:\Spec(R)\to\Spec(A)$ be a homomorphism of logarithmic $S/J$-schemes. Denote also by $h$ the associated morphism of log-crystalline sites
\[ h:(R|S)_{\logcrys}\to (A|S)_{\logcrys}. \]
Then for any sheaf of abelian groups $\mathscr{F}$ on $(R|S)_{\logcrys}$, the $i$th direct image sheaf $R^ih_\ast\mathscr{F}$ is the sheaf on $(A|S)_{\logcrys}$ associated to the presheaf
\[ (U\hookrightarrow T)\rightsquigarrow H^i_{\logcrys}(\Spec(R)\times_{\Spec(A)}U|T,\mathscr{F}). \]
If $\mathscr{E}$ is a quasi-coherent sheaf of $\mathscr{O}$-modules on $(R|S)_{\logcrys}$ then it follows immediately from Theorem \ref{fontaine} that for $i\neq 0$
\[ R^ih_\ast\mathscr{E}=0 \]
and hence for all $i$
\[ H^i_{\logcrys}(R|S,\mathscr{E})\cong H^i_{\logcrys}(A|S,h_\ast\mathscr{E}) \]
which is again zero for $i\neq 0$.

\begin{prop}\label{qcohcrystal}
With the above notation and assumptions, $h_\ast\mathscr{O}$ is a quasi-coherent crystal of $\mathscr{O}$-modules on $(A|S)_{\logcrys}$.
\end{prop}

\begin{proof}
Assume $p^nS=0$. Let $\mathcal{P}$ be the presheaf on $(A|S)_{\logcrys}$ defined
\[ \mathcal{P}(U\hookrightarrow T):=H^0_{\logcrys}(\Spec(R)\times_{\Spec(A)}U|T,\mathscr{O}). \]
Consider a morphism $g:(U'\hookrightarrow T')\to (U\hookrightarrow T)$ of $(A|S)_{\logcrys}$. Then $g:T'\to T$ is an open map, and hence
\[ \left( g^{-1}\mathcal{P}|_T\otimes_{g^{-1}\mathscr{O}_T}\mathscr{O}_{T'}\right) (T')=\mathcal{P}(g(T'))\otimes_{\mathscr{O}_T(g(T))}\mathscr{O}_{T'}(T'). \]
We claim that, up to localizing on $U'$, we have
\[ g^{-1}\mathcal{P}|_T\otimes_{g^{-1}\mathscr{O}_T}\mathscr{O}_{T'}\cong\mathcal{P}|_{T'}.  \]
We have a commutative square
\[ \begin{CD}
U' @>>> g(U') \\
@VVV @VVV \\
T' @>>> g(T')
\end{CD} \]
Let $C$ be the unique flat $g(T')$-scheme making the following square cartesian
\[ \begin{CD}
U' @>>> g(U') \\
@VVV @VVV \\
C @>>> g(T').
\end{CD} \]
Then $C\to g(T')$ is \'etale and therefore the morphism $U'\hookrightarrow C$ has a unique structure of a DP thickening which is compatible with that of $g(U')\hookrightarrow g(T')$. Moreover, we have a commutative square
\[ \begin{CD}
U' @>>> C \\
@VVV @VVV \\
T' @>>> g(T').
\end{CD} \]
Since the left vertical arrow is a nilpotent thickening, there exists a unique morphism $T'\to C$, necessarily a DP-morphism, making the resulting diagram commute. Hence we have a commutative diagram
\[ \begin{CD}
U' @= U' @>>> g(U') \\
@VVV @VVV @VVV \\
T' @>>> C @>>> g(T')
\end{CD} \]
and so we reduce to proving the claim in the following two cases
\begin{enumerate}[I.]
\item $g:T'\to g(T')$ is \'etale
\item $g:U'\to g(U')$ is an isomorphism.
\end{enumerate}
In case I, we must show the following: if $(U'\hookrightarrow T')\to (U\hookrightarrow T)$ is an \'etale map with $U$ and $U'$ affine and $U\times_TT'\cong U'$, then the canonical map
\begin{equation}\label{qcoheq1} \mathcal{P}(U\hookrightarrow T)\otimes_{\mathscr{O}(T)}\mathscr{O}(T')\to\mathcal{P}(U'\hookrightarrow T') 
\end{equation}
is an isomorphism. I thank the referee for the following argument (my original argument was a direct verification). To simplify, let $\mathcal{P}=\mathcal{P}(U\hookrightarrow T)$. Note that by Theorem \ref{fontaine}, there is a closed immersion $\left(\Spec(R)\hookrightarrow\Spec(\mathcal{P})\right)$ which is in fact the final object of the site $(\Spec(R)\times_{\Spec(A)}U|T)_{\logcrys}$. It suffices to show that the closed immersion $X:=(\Spec(R)\times_{\Spec(S)}U'\hookrightarrow \Spec(\mathcal{P})\times_TT')$ is the final object of the site $(\Spec(R)\times_{\Spec(S)}U'|T')_{\logcrys}$, where $\Spec(R)\times_{\Spec(S)}U'$ is endowed with the inverse image log-structure under the structure morphism to $\Spec(R)$ and $T'$ the inverse image log-structure under the structure morphism to $T$. We endow $\Spec(\mathcal{P})\times_TT'$ with the inverse image log-structure under the structure morphism to $\Spec(\mathcal{P})$. Note that by flatness of $T'\to T$, the divided power structure of the closed immersion $\Spec(R)\times_{\Spec(A)}T\hookrightarrow\Spec(\mathcal{P})$ extends uniquely to $X$, so that $X$ is indeed an object of the site in question. Since any object $Y$ of this site can be naturally considered as an object of the site $(\Spec(R)\times_{\Spec(A)}U|T)_{\logcrys}$, there is a unique morphism $Y\to \left(\Spec(R)\hookrightarrow\Spec(\mathcal{P})\right)$. This morphism factors uniquely $Y\to X\to \left(\Spec(R)\hookrightarrow\Spec(\mathcal{P})\right)$, thus proving that $X$ is the final object.

In case II, we can assume $U'$ is affine, and hence so are $T'$ and $g(T')$. Let $U'=\Spec(E_0)$, $g(T')=\Spec(E)$ and $T'=\Spec(F)$. Let $M\to\mathscr{O}_{U'}$ be the log-structure of $U'$. Up to localizing further we may assume that, in a neighbourhood of a geometric point $\bar{x}\to U'$, the log-structure of $U'$ is given by the fine saturated monoid $\Lambda:=M_{\bar{x}}/\mathscr{O}_{U',\bar{x}}^{\ast}$. Since $U'\hookrightarrow g(T')$ is an exact closed immersion, it follows from Lemma \ref{nilexact} (iii) that the there is a map $\Lambda\to E$ defining the log-structure of $E$, and moreover the map $\Lambda\to E\to F$ also defines the log-structure of $F$. We will construct $g^{-1}\mathcal{P}|_T\otimes_{g^{-1}\mathscr{O}_T}\mathscr{O}_{T'}$ and $\mathcal{P}|_{T'}$ as in the proof of Theorem \ref{fontaine} and prove that they are canonically isomorphic. Let $Q:=P(M)\oplus\Lambda$. As in the proof of Theorem \ref{fontaine}, there is a canonical map $Q\to M$. Define
\[ L:=\ker(Q^{\gp}\to M^{\gp}). \]
Let $J_F$ (resp. $J_E$) denote the ideal sheaf of $U'$ in $T'$ (resp. $g(U')$ in $g(T')$).  Define the pairs
\begin{eqnarray*}
(B_E,I_E) &=& \left( (W(P(R\otimes_{A}E_0))\otimes_{\mathbb{Z}}E)\otimes_{\mathbb{Z}}\mathbb{Z}[L], J_E\cdot (W(P(R\otimes_{A}E_0))\otimes_\mathbb{Z}E)\otimes_{\mathbb{Z}}\mathbb{Z}[L]\right) \\
(B_F,I_F) &=& \left( (W(P(R\otimes_{A}E_0))\otimes_{\mathbb{Z}}F)\otimes_{\mathbb{Z}}\mathbb{Z}[L],J_F\cdot (W(P(R\otimes_{A}E_0))\otimes_\mathbb{Z}F)\otimes_{\mathbb{Z}}\mathbb{Z}[L]\right)
\end{eqnarray*}
Note that since the ring of Witt vectors of a perfect ring of characteristic $p$ is $\mathbb{Z}$-torsion free, hence flat over $\mathbb{Z}$, it follows that $E\to B_E$ (resp. $F\to B_F$) is flat. Hence the ideal $I_E$ (resp. $I_F$) is a DP-ideal. Since $(E,J_E)\to (F,J_F)$ is a DP-morphism, so isthe canonical map $(B_E,I_E)\to (B_F,I_F)$. Moreover we have
\[ B_E/I_E\cong B_F/I_F. \]
Define
\begin{eqnarray*}
\mathscr{I}_E &=& \ker(B_E\to R\otimes_{A}E_0) \\
\mathscr{I}_F &=& \ker(B_F\to R\otimes_{A}E_0)
\end{eqnarray*}
for the canonical maps. Then $I_E\subset \mathscr{I}_E$ and $I_F\subset \mathscr{I}_F$, and via the map $B_E\to B_F$, $\mathscr{I}_E$ maps to $\mathscr{I}_F$, and by definition we have $B_E/\mathscr{I}_E\cong B_F/\mathscr{I}_F$. So by \cite[Ch. I, Prop. 2.8.2]{berthelot}, we have a canonical isomorphism
\[ D_{B_F}(\mathscr{I}_F)\cong D_{B_E}(\mathscr{I}_E)\otimes_{B_E}B_F=D_{B_E}(\mathscr{I}_E)\otimes_EF\]
where $D_{B_F}(\mathscr{I}_F)$ (resp. $D_{B_E}(\mathscr{I}_E)$) denotes the divided power hull of $B_F$ (resp. $B_E$) for the ideal $\mathscr{I}_F$ (resp. $\mathscr{I}_E$) compatible with the divided power structure on $J_F$ (resp. $J_E$). Let $K\subset D_{B_E}(\mathscr{I}_E)$ be the DP-ideal generated by the divided powers of all elements of the form $q_l\otimes l-p_l\otimes 1$ as in Theorem \ref{fontaine}. Then $\text{Im}(K\otimes_EF)\subset D_{B_F}(\mathscr{I}_F)$ is the ideal generated by the divided powers of all elements of the form $q_l\otimes l-p_l\otimes 1$ (where $l=p_l/q_l\in L\subset Q^{\gp}$), hence in particular is a DP-ideal. Taking the quotient by $K$ we obtain an isomorphism
\[ (D_{B_E}(\mathscr{I}_E)/K)\otimes_EF \cong D_{B_F}(\mathscr{I}_F)/\text{Im}(K\otimes_EF) \]
which is precisely
\[ g^{-1}\mathcal{P}|_T\otimes_{g^{-1}\mathscr{O}_T}\mathscr{O}_{T'}\cong\mathcal{P}|_{T'} \]
and hence we have shown the claim. Sheafifying we see that
\[ h_\ast\mathscr{O}|_{T'}\cong g^\ast\left( h_\ast\mathscr{O}|_{T}\right)  \]
so $h_\ast\mathscr{O}$ is a crystal of $\mathscr{O}$-modules. Finally, the quasi-coherence of $\mathcal{P}$ follows from the fact that the map (\ref{qcoheq1}) is an isomorphism.
\end{proof}

\subsection{Some Fontaine rings}

\subsubsection{}
Let $K^+$ be a complete discrete valuation ring of characteristic zero and of perfect residue field $k$ of characteristic $p>0$. Denote by $K$ its fraction field. Let $\bar{K}$ be an algebraic closure of $K$ and let $\bar{K}^+$ be its valuation ring. Following Fontaine we set
\[ A_{\cris}:=A_ {\cris}(K^+):=\varprojlim_n H^0_{\crys}(\bar{K}^+/p\bar{K}^+|W_{n+1}(k),\mathscr{O}). \]
This ring may be constructed as in the proof of Thm. \ref{fontaine}. Let $P(\bar{K}^+/p\bar{K}^+)$ be as in the proof of \emph{loc. cit.} Consider the element
\[ \underline{1}:=(1,\zeta_p,\zeta_{p^2},...)\in P(\bar{K}^+/p\bar{K}^+) \]
where $\zeta_{p^n}$ denotes a primitive $p^n$th root of unity and $\zeta_{p^{n+1}}^p=\zeta_{p^n}$ for all $n$. Define
\[ A_{\inf}(K^+):=W(P(\bar{K}^+/p\bar{K}^+)) \]
so that $[\underline{1}]\in A_{\inf}(K^+)$. Define an element of $A_{\cris}$ by
\[ t:=\log([\underline{1}])=-\sum_{n>0}(n-1)!(1-[\underline{1}])^{[n]}. \]
This element plays a very important role in $p$-adic Hodge theory. By functoriality of crystalline cohomology, $A_{\cris}$ has a Frobenius endomorphism $\Phi$.

\begin{lemma}\label{flat}
$A_{\cris}(K^+)$ is a flat $W(k)$-algebra.
\end{lemma}
\begin{proof}
Let $P:=P(\bar{K}^+/p\bar{K}^+)$. By \cite[5.1.2]{font} $P$ is a valuation ring, in particular an integral domain. Since $P$ is perfect, it follows that $W(P)$ is a flat and $p$-adically separated $W(k)$-algebra, hence is also an integral domain. As is well-known, $A_{\cris}(K^+)$ is the $p$-adic completion of $W(P)^{\DP}=$ the divided power hull of $W(P)$ for the ideal $(\xi)$, where $\xi=[\underline{p}]-p$ and
\[ \underline{p}=(p,p^{1/p},p^{1/p^2},...)\in P(\bar{K}^+/p\bar{K}^+) \]
is a compatible system of $p$-power roots of $p$ in $\bar{K}^+/p\bar{K}^+$. So it suffices to show that $W(P)^{\DP}$ is $p$-torsion free. Now by construction of divided power hulls, we have an exact sequence
\[ 0\to\mathcal{I}\to\Gamma_{W(P)}((\xi))\to W(P)^{\DP}\to 0 \]
where $\Gamma_{W(P)}((\xi))=W(P)\left\langle X\right\rangle $ is the divided power polynomial algebra on one variable $X$ and $\mathcal{I}$ is the principal ideal generated by $\xi-X$. Since $\Gamma_{W(P)}((\xi))$ is flat over $W(k)$, it suffices to show that the map $\mathcal{I}\otimes_{W(k)}W_n(k)\to\Gamma_{W(P)}((\xi))\otimes_{W(k)}W_n(k)$ is injective for all $n$. Suppose $(\xi-X)f\equiv 0\mod p^n$ for some $f\in \Gamma_{W(P)}((\xi))$. Then $(\xi-X)f=p^ng$ for some $g\in \Gamma_{W(P)}((\xi))$. Write $f=\sum_{i\geq 0}f_iX^{[i]}$ and $g=\sum_{i\geq 0}g_iX^{[i]}$ for some uniquely determined $f_i,g_i\in W(P)$. Then we have $\xi f_i-(i+1)f_{i-1}=p^ng_i$ for all $i$. In particular, for $i=0$ we get $\xi f_0=p^ng_0$. Since $\xi$ generates a prime ideal which does not contain $p$, the identity $\xi f_0=p^ng_0$ implies that $g_0\in (\xi)$, whence $f_0\equiv 0\mod p^n$ since $W(P)$ is an integral domain. For $i=1$ we have $\xi f_1-2p^nf_0'=p^n g_1$, hence again by the same argument $f_1=p^nf'_1$. Continuing in the manner we see that $f=p^nf'$, which implies the claim.
\end{proof}

\subsubsection{}
Because of the lack of a lifting of the Frobenius endomorphism of $K^+/pK^+$ to $K^+/p^nK^+$, we will need to consider a formal divided power lifting $\Sigma$ of $K^+/pK^+$, defined as follows. Firstly, making a choice of uniformizer $\pi$ of $K^+$ determines a presentation $K^+=W(k)[u]/(E(u))$, where $E(u)$ is the minimal equation of $\pi$ over $W(k)$, i.e. $E(u)$ is an Eisenstein equation of degree $e$, where $e$ is the ramification index of $K^+$ over $W(k)$. So $W_n(k)[u]$ is a smooth $W_n(k)$-lift of $K^+/pK^+=k[u]/(u^e)$. Taking the divided power hull for the kernel of the surjection $W_n(k)[u]\to K^+/pK^+$, we obtain $\Sigma_{n}=W_n(k)[u]\left\langle u^e\right\rangle $. It has a lifting of the absolute Frobenius of $K^+/pK^+$ defined as the unique homomorphism sending $u$ to $u^p$ and restricting to the canonical Frobenius on $W_n(k)$. Finally we define
\[ \Sigma:=\lim_n \Sigma_{n+1}. \] 

\subsubsection{}
Let us consider the log-crystalline cohomology of $\bar{K}^+/p\bar{K}^+$ over the logarithmic base $\Sigma_n$, where the log-structures are as follows. We endow $\Sigma_n$ with the \emph{canonical log-structure} on $\Sigma_{n}$, i.e. the fine log-structure associated to the pre-log-structure
\[ \mathcal{L}(u):\mathbb{N}\to \Sigma_{n}:1\mapsto u. \]
Composing this with the canonical map $\Sigma_{n}\to K^+/p^nK^+$ defines a pre-log-structure on the latter, making it into a log-$\Sigma_{n}$-scheme. The associated log-structure is also the inverse image of the canonical log-structure on $K^+$ defined
\[ K^+-\left\lbrace 0\right\rbrace \to K^+.\]
Similarly, the canonical log-structure on $\bar{K}^+$ is given by
\[ \bar{K}^+-\left\lbrace 0\right\rbrace \to \bar{K}^+ \]
and we endow $\bar{K}^+/p\bar{K}^+$ with the inverse image log-structure. If we fix roots of $\pi$, then it is the log-structure associated to the pre-log-structure
\[ M:\mathbb{Q}_+\to \bar{K}^+/p\bar{K}^+:\alpha\mapsto\pi^{\alpha}.  \]
(Note that although this pre-log-structure depends on choices of roots of $\pi$, its associated log-structure does not, as two choices of roots of $\pi$ will differ by a unit of $\bar{K}^+/p\bar{K}^+$.) Let
\[ L:=\ker(P(M)^{\gp}\oplus\mathcal{L}(u)^{\gp}\to M^{\gp}). \]
Since $P(M)$ consists of sequences $(x_n)$ of non-negative rational numbers such that $p\cdot x_{n+1}=x_n$ for all $n$, we see that $P(M)^{\gp}$ consists of sequences $(x_n)$ of rational numbers such that $p\cdot x_{n+1}=x_n$ for all $n$, i.e. $P(M)^{\gp}\cong\mathbb{Q}$. So we see that $L$ is the kernel of the map
\[ \mathbb{Q}\oplus\mathbb{Z}\to\mathbb{Q}:(\alpha,m)\mapsto \alpha+m \]
i.e. $L$ consists of pairs $(m,-m)\in\mathbb{Z}^2$. Note that under the map to $A_{\text{inf}}(K^+)\otimes_\mathbb{Z}\Sigma$, $(m\cdot p^{-n})_n\in P(M)$ maps to $[\underline{\pi}]^{m}\otimes 1$, where $\underline{\pi}=(\pi,\pi^{1/p},\pi^{1/p^2},...)\in P(\bar{K}^+/p\bar{K}^+)$ is the sequence of $p$-power roots of $\pi$ determined by $M$, and $m\in\mathcal{L}(u)$ maps to $1\otimes u^m$. So one should think of $(m,-m)$ as $[\underline{\pi}]^{m}\otimes u^{-m}$.

Define
\[ B^+_{\log}:=\lim_n H^0_{\logcrys}(\bar{K}^+/p\bar{K}^+|\Sigma_{n},\mathscr{O}). \]
By construction (\emph{cf}. proof of Thm. \ref{fontaine}), $B^+_{\log}$ is a quotient of the $p$-adic completion of the divided power hull of $A_{\text{inf}}(K^+)\otimes_\mathbb{Z}\Sigma\otimes_{\mathbb{Z}}\mathbb{Z}[L]$ for the kernel of the canonical surjection onto $\hat{\bar{K}}^+$. In fact, by \cite[Prop. 3.3]{kato} we know that every choice of sequence of $p$-power roots of $\pi$ determines an isomorphism
\[ B^+_{\log}/p^nB^+_{\log}\simeq A_{\cris}/p^nA_{\cris}\left\langle X\right\rangle \]
where $A_{\cris}/p^nA_{\cris}\left\langle X\right\rangle $ is a divided power polynomial ring in one indeterminate $X$ over the ring $A_{\cris}/p^nA_{\cris}$. One should think of $X$ as $[\underline{\pi}]\otimes u^{-1}-1$. $B^+_{\log}$ is also endowed with a Frobenius endomorphism.

Now, let us fix a sequence $(c_0=\pi,c_1,c_2,...)$ of $p$-power roots of $\pi$ in $\bar{K}^+$ such that $c_{n+1}^p=c_n$ for all $n\in\mathbb{N}$. Consider the log-structure associated to
\[ \mathbb{N}[1/p]\to \bar{K}^+/p\bar{K}^+:\dfrac{n}{p^m}\mapsto c_m^n. \]
If we define
\[ L':=\ker\left(P(\mathbb{N}[1/p])^{\gp}\oplus\mathcal{L}(u)^{\gp}\to\mathbb{N}[1/p]^{\gp}\right) \]
then $L'=L$ if the $p$-power roots of $\pi$ which we fixed in order to define $M$ are the same as $(c_0,c_1,...)$. In fact, the log-crystalline cohomology of $\bar{K}^+/p\bar{K}^+$ is the same with this log-structure as with $M$ (this can be checked as in the proof of Prop. \ref{dppolynomialalgebra}). So in the sequel we will use this latter log-structure, rather than the canonical one (\emph{cf.} \S\ref{logstr}).

\subsubsection{}
We now make a computation, due to Fontaine, which will be crucial to us.

\begin{prop}[Fontaine]\label{constants}
Let $\alpha\in\mathbb{Q}$. We have
\begin{enumerate}[(i)]
\item $t^{p-1}\in p\cdot A_{\cris}(K^+)$
\item if $\alpha\in\mathbb{Z}$ then $[\underline{1}]^{\alpha}-1=\alpha\cdot t\cdot u_\alpha$, where $u_\alpha$ is a unit of $A_{\cris}(K^+)$
\item $\frac{t\cdot p^{\max\left\lbrace v_p(\alpha),0\right\rbrace }}{[\underline{1}]^{\alpha}-1}\in A_{\cris}(K^+).$
\end{enumerate}
\end{prop}
\begin{proof}
(i): Clearly $\zeta=\dfrac{1-[\underline{1}]}{1-[\underline{1}]^{1/p}}$ lies in the kernel of $A_{\Inf}(K^+)\to \hat{\bar{K}}^+$. We find $\zeta^p\equiv (1-[\underline{1}])^{p-1}\mod p$, hence $(1-[\underline{1}])^{p-1}\in p\cdot A_{\cris}(K^+)$. Now, we have
\[ t=-\sum_{n=1}^{p}(n-1)!(1-[\underline{1}])^{[n]}-\sum_{n=p+1}^{\infty}(n-1)!(1-[\underline{1}])^{[n]} \]
and the second sum is divisible by $p$, so it suffices to consider
\[ \left( \sum_{n=1}^{p}(n-1)!(1-[\underline{1}])^{[n]}\right)^{p-1}=\sum_{j_1+...+j_p=p-1}\dfrac{(p-1)!(1-[\underline{1}])^{\sum i\cdot j_i}}{j_1!\cdots j_p!\prod_{i=1}^{p}i^{j_i}}. \]
Here each summand has $p$-adic valuation $j_p$ in the denominator and at least $\left[ \frac{\sum i\cdot j_i}{p-1}\right] \geq j_p+1$ in the numerator. This proves (i).

(ii): $[\underline{1}]^{\alpha}-1\in\ker(A_{\cris}(K^+)\to\hat{\bar{K}}^+)$, so the divided power series $\log([\underline{1}]^{\alpha})$ exists and converges to $\alpha\cdot t$. We have
\[ [\underline{1}]^{\alpha}-1=\exp(\alpha\cdot t)-1=\alpha\cdot t\cdot\sum_{n>0}\dfrac{(\alpha\cdot t)^{n-1}}{n!} \]
and
\[ \dfrac{t^{n-1}}{n!}=\dfrac{p^{q_n}}{n!}q_n!(t^{p-1}/p)^{[q_n]}t^{r_n} \]
where $n-1=q_n(p-1)+r_n$, $0\leq r_n<p-1$, and $q_n=\left[ \dfrac{n-1}{p-1}\right]\geq v_p(n!)$. So $u_{\alpha}:=\sum_{n>0}\dfrac{(\alpha\cdot t)^{n-1}}{n!}$ converges to a unit of $A_{\cris}(K^+)$ for all $\alpha\in\mathbb{Z}$. Hence assertion (ii).

(iii): We separate in two cases:
\begin{enumerate}[I.]
\item $v_p(\alpha)\geq 0$
\item $v_p(\alpha)<0$.
\end{enumerate}
In case I, if $\alpha =\frac{x}{y}$ with $x,y$ integers and $v_p(y)=0$, then $\frac{[\underline{1}]^{x}-1}{[\underline{1}]^{\alpha}-1}\in A_{\cris}(K^+)$, so assertion (iii) holds in this case by (ii).

In case II, write $\alpha=\frac{x}{yp^n}$ with $x,y$ coprime integers and $v_p(x/y)=0$. By case I we have $([\underline{1}]^{x}-1)^{-1}\in t^{-1}\cdot A_{\cris}(K^+)$, and $\frac{[\underline{1}]^{x}-1}{[\underline{1}]^{\alpha}-1}\in A_{\cris}(K^+)$, hence assertion (iii) in this case.
\end{proof}

\section{Almost ring theory}
In this section we recall some results in almost ring theory and then apply them to Fontaine rings. 

\subsection{Reminder on almost ring theory}

\subsubsection{}
We refer to \cite{almost} for the fundamentals of almost ring theory, in particular the categories of almost modules and almost rings. We will always refer to morphisms in these almost categories with the adjective ``almost", e.g. ``almost homomorphism", etc. So if we do not use the adjective ``almost" we will be referring to usual morphisms. Finally, we reserve the notation $M\thickapprox N$ for almost isomorphisms.

The almost ring theory we will consider will be relative to a pair $(\Lambda,\mathfrak{m})$ where:
\begin{itemize}
\item $\Lambda$ is a rank 1 valuation ring of characteristic zero
\item $\mathfrak{m}\subset\Lambda$ is its maximal ideal, such that $p\in\mathfrak{m}$
\item $\mathfrak{m}^2=\mathfrak{m}$.
\end{itemize}
Note that since $\Lambda$ is a valuation ring, $\mathfrak{m}$ is the filtered union of principal ideals, the index set for the limit being the set of valuations of elements of $\mathfrak{m}$.

\subsubsection{}\label{relativefrob} 

\begin{lemma}\label{aechilfsatz}
\begin{enumerate}[(i)]
\item Almost \'etale coverings are stable by base change.
\item If $A\to B\to C$ are homomorphisms such that $A\to B$ and the composition $B\to C$ are almost \'etale coverings, then so is $B\to C$.
\item An almost projective module of finite rank which is everywhere of non-zero rank is almost faithfully flat. In particular, an almost \'etale covering which is everywhere of non-zero rank is almost faithfully flat.
\end{enumerate}
\end{lemma}
\begin{proof}
This follows in a straightforward way from the results of \cite{almost}.
\end{proof}

The behaviour of almost \'etale morphisms of $\mathbb{F}_p$-algebras under the Frobenius endomorphism is studied in some detail in \cite{almost}.

\begin{thm}[Gabber-Ramero]\label{almostcocart}
Let $A\to B$ be an almost \'etale homomorphism of $\Lambda/p\Lambda$-algebras. Then the commutative diagram
\[ \begin{CD}
A @>>> B\\
@V{F}VV @V{F}VV\\
A @>>> B
\end{CD} \]
is almost cocartesian, where $F$ denotes the (absolute) Frobenius endomorphism. In other words, the relative Frobenius of $B$ over $A$ is an almost isomorphism.
\end{thm}

This is \cite[Thm. 3.5.13]{almost}. We will need the following corollaries.

In the next corollary, we consider the almost ring theory relative to the pair $(W(\Lambda/p\Lambda),[\mathfrak{m}])$, where the ideal $[\mathfrak{m}]$ is the ideal generated by $[\bar{x}]$ for all $x\in\mathfrak{m}$, where $\bar{x}\in\Lambda/p\Lambda$ denotes the image of $x\in\Lambda$, and $[-]$ is the Teichm\"uller map. We claim that $[\mathfrak{m}]$ is the filtered union of principal ideals and $[\mathfrak{m}]^2=[\mathfrak{m}]$. Since $\Lambda$ is a valuation ring and $[\underline{x}]$ is nilpotent for all $x\in\mathfrak{m}$, one sees easily that every ideal generated by a finite number of generators of $[\mathfrak{m}]$ is principal. From this the first claim follows easily. For the second, suppose that $x=\sum_{i=0}^s[\bar{x}_i]a_i\in [\mathfrak{m}]$. Fix some elements $x_i\in\mathfrak{m}$ such that $\bar{x}_i\equiv x_i\mod p$. Without loss of generality, we may assume that $x_0$ has smallest valuation of all $x_i$. Then $x_i=x_0y_i$ for some $y_i\in\Lambda$, so $x=[\bar{x}_0](a_0+\sum_{i>0}[\bar{y}_i]a_i)$. Now, since $\mathfrak{m}^2=\mathfrak{m}$, it follows that $x_0=x_{-1}x_{-2}$ for some $x_{-1},x_{-2}\in\mathfrak{m}$. So $x=[\bar{x}_{-1}]\cdot [\bar{x}_{-2}](a_0+\sum_{i>0}[\bar{y}_i]a_i)\in[\mathfrak{m}]^2$, which proves the claim.

\begin{cor}\label{wittetale}
If $A\to B$ is an almost \'etale homomorphism of $\Lambda/p\Lambda$-algebras, then the induced homomorphism of Witt vectors
\[ W_n(A)\to W_n(B) \]
is an almost \'etale homomorphism of $W(\Lambda/p\Lambda)$-algebras.
\end{cor}
\begin{proof}
The proof is the same as the classical case \cite[0. Prop. 1.5.8]{illusie}. We reproduce it here for the reader. We prove the assertion by induction on $n$, the case $n=1$ being Theorem \ref{almostcocart}. Suppose that the result is true for $n-1$. Let $V$ and $F$ be the Verschiebung and Frobenius morphisms of the functor $W_n(-)$. For any $\mathbb{F}_p$-algebra $R$, $V^iW_{n-i}(R)$ defines a filtration of $W_n(R)$, called the $V$-adic filtration, whose graded has a natural graded ring structure. We may consider $R$ as a module over itself via the $i$th power $F^i$ of the Frobenius endomorphism $F:R\to R$. We denote it $F^i_{\ast}R$. Then according to \emph{loc. cit.} there is a canonical isomorphism $F^i_{\ast}R\cong\gr^i_VW_n(R):=V^iW_{n-i}(R)/V^{i+1}W_{n-(i+1)}(R)$. Now the natural map $B\otimes_A\gr_V^iW_n(A)\to\gr_V^iW_n(B)$ can be identified with the map $B\otimes_AF^i_{\ast}A\to F^i_{\ast}B$. It follows from Theorem \ref{almostcocart} that it is an almost isomorphism. Hence so is the map $B\otimes_A\gr_V^iW_n(A)\to\gr_V^iW_n(B)$. Since $V^n=0$, it follows that the natural homomorphism $V^{n-1}(A)\otimes_{W_n(A)}W_n(B)\to V^{n-1}(B)$ is an almost isomorphism. Writing $I:=V^{n-1}(A)\subset W_n(A)$, we get that the map $I\otimes_{W_n(A)}W_n(B)\to I\cdot W_n(B)$ is an almost isomorphism and $W_n(B)\otimes_{W_n(A)}W_n(A)/I\thickapprox W_n(B)/V^{n-1}(B)=W_{n-1}(B)$ is almost flat over $W_{n-1}(A)$ by inductive hypothesis. By the next lemma this implies that $W_n(A)\to W_n(B)$ is almost flat.  It follows from  \cite[proof of Thm. 3.2.18 (ii)]{almost} that $W_n(A)\to W_n(B)$ must be the unique almost \'etale lift of $W_{n-1}(A)\to W_{n-1}(B)$.
\end{proof}

In the proof we have used the following almost analogue of Bourbaki's local flatness criterion.

\begin{lemma}\label{locflatcrit}
Suppose that $A$ is a $W_n(\Lambda/p\Lambda)$-algebra, $I\subset A$ a nilpotent ideal, and $M$ an $A$-module. If the natural map $I\otimes_AM\to IM$ is an almost isomorphism and $M/IM$ is almost flat over $A/I$, then $M$ is an almost flat $A$-module.
\end{lemma}
\begin{proof}
We refer to \cite[2.4.10]{almost} for details on $\Tor$-functors and their relation to flatness in the almost setting. It suffices to show that $\Tor_i^A(M,N)\thickapprox 0$, $i>0$, for any $A$-module $N$, and since any module can be written as a quotient of a free module this reduces to showing that $\Tor_1^A(M,N)\thickapprox 0$. Firstly, it follows from the hypothesis that $\Tor_1^A(M,A/I)\thickapprox 0$. Since we may write any $A/I$-module as a quotient of a free module, it follows easily that $\Tor_1^A(M,N)$ for any $A$-module $N$ such that $IN=0$. For general $N$ an easy induction on $m$ such that $I^mN=0$ using the exact sequences
\[ 0\to IN\to N\to N/IN\to 0 \]
shows that $\Tor_1^A(M,N)\thickapprox 0$.
\end{proof}

\begin{cor}\label{frobsurjective0}
Assume that there is a sequence $x_n\in\mathfrak{m}$ such that $x_{n+1}^p=x_n$ for all $n\in\mathbb{N}$. Suppose that $A\to B$ is an almost \'etale homomorphism of $\Lambda$-algebras. Suppose $B$ is a flat $\Lambda$-algebra which is integrally closed in $B\otimes_{\Lambda}Q(\Lambda)$. If the Frobenius is surjective on $A/pA$, then the Frobenius is also surjective on $B/pB$.
\end{cor}
\begin{proof}[Proof (Faltings)]
Let $v$ be the valuation of $\Lambda$ and $\Gamma:=v(\Lambda)$. By Theorem \ref{almostcocart} we know that the Frobenius is almost surjective on $B/pB$. So for all $b\in B$ and all $\alpha\in\Gamma$, $0<\alpha<v(p)$, there is $n$ such that $v(x_0)/p^n\leq\alpha$ and such that we may write $x_{n}b=c^p+pd$. Rewriting this as $c^p=x_{n}(b-\frac{p}{x_n}d)$, we have $c^p=x_{n+1}^p(b-\frac{p}{x_n}d)$. Since $B$ is integrally closed in $B\otimes_{\Lambda}Q(\Lambda)$, we have $c=x_{n+1}e$ for some $e\in B$, so $b=e^p+\frac{p}{x_n}d$. Similarly, $\frac{p}{x_n}d=f^p+pg$, hence $b=e^p+f^p+pg\equiv (e+f)^p\mod p$.
\end{proof}

\subsubsection{} If $R$ is a ring and $X$ is a finite set, then we write $R\times X$ for the product $\prod_{x\in X} R_x$, where $R_x=R$ for all $x\in X$. We say that an almost \'etale homomorphism $A\to B$ of $\Lambda$-algebras is an \emph{almost Galois homomorphism} of group $G$ if $G$ is a finite group acting by $A$-algebra automorphisms on $B$ such that the canonical map
\[ B\otimes_A B\to B\times G \]
induced by the maps $B\to B\times G:b\mapsto (b,b,...,b)$ and $B\to B\times G:b\mapsto (g(b))_{g\in G}$, is an almost isomorphism.

We say that an almost Galois homomorphism $A\to B$ of $\Lambda$-algebras is an \emph{almost Galois covering} if it is almost faithfully flat.\\

Note that an almost Galois homomorphism (resp. covering) is preserved under arbitrary base change.

\begin{prop}\label{galois}
Let $A\to B$ be an almost Galois covering of group $G$. Let $M$ a $B$-module with semi-linear $G$-action. Then
\begin{enumerate}[(i)]
\item $B^G\thickapprox A$
\item $M^G\otimes_A B\thickapprox M$
\item $H^i(G,M)\thickapprox 0$ for all $i\neq 0$.
\end{enumerate}
\end{prop}
\begin{proof}
See \cite{fa4}.
\end{proof}

\subsubsection{}
Let $A$ be a $\Lambda$-algebra and $I\subset A$ an ideal such that $I^2=0$.

\begin{thm}[Faltings]
The category of almost \'etale coverings of $A/I$ is equivalent to the category of almost \'etale coverings of $A$.
\end{thm}

For the proof, see \cite[3. Theorem]{fa4}. More generally, we have (\cite[Thms. 3.2.18, 3.2.28]{almost})

\begin{thm}[Gabber-Ramero]\label{nilpotentlifts}
The category of almost \'etale homomorphisms of $A/I$ is equivalent to the category of almost \'etale homomorphisms of $A$.
\end{thm}

\begin{prop}
The equivalence in Theorem \ref{nilpotentlifts} preserves almost faithful flatness.
\end{prop}
\begin{proof}
Let $A\to B$ be an almost \'etale covering such that $A/I\to B/I\cdot B$ is almost faithfully flat. Let $M$ be an $A$-module such that $M\otimes_A B\approx 0$. Then $(M/IM)\otimes_{A/IA}(B/IB)\approx 0$, hence $M/I\cdot M\approx 0$, whence $M\approx 0$.
\end{proof}

\begin{cor}\label{wittgalois}
If $A\to B$ is an almost \'etale homomorphism and $G$ is a finite group of $A$-automorphisms of $B$, then it is an almost Galois homomorphism (resp. covering) of group $G$ if and only if $A/I\to B/I\cdot B$ is an almost Galois homomorphism (resp. covering) of group $G$.
\end{cor}
\begin{proof}
Since almost faithful flatness is preserved under the equivalence, it suffices to check that almost Galois homomorphisms are preserved. It suffices to prove that the map $B\otimes_A B\to B\times G$ is almost flat for then it must be almost \'etale (see the proof of \cite[Thm. 3.2.18 (ii)]{almost}), hence an almost isomorphism. Note that
\[ \begin{CD}
(B\otimes_AB @>>> B\times G)=\oplus_{g\in G}( B\otimes_AB @>{1\otimes g}>> B\otimes_AB @>{x\otimes y\mapsto xy}>> B).
\end{CD} \]
Since $A\to B$ is almost \'etale, $B$ is an almost projective $B\otimes_AB$-module, hence so is $B\times G$.
\end{proof}

\subsection{Almost purity}
\subsubsection{}\label{small}
Fix $c\in\left\lbrace 1, \pi\right\rbrace \subset K^+$ (where $\pi\in K^+$ was our choice of uniformizer) and define
\[ O(c):=K^+[T_1,...,T_r,T_{r+1}^{\pm 1},...,T_{d+1}^{\pm 1}]/(T_1\cdots T_r -c). \]
Following Faltings, we say that a $K^+$-algebra $R$ is \emph{small} if there is an \'etale map
\[ O(c)\to R. \]
Every smooth (resp. semi-stable) $K^+$-scheme has an \'etale covering by small affines with $c=1$ (resp. $c=\pi$). Note that with these choices of $c$, $R$ is a regular ring; in particular, it is a finite product of integrally closed domains, and we shall assume that $R$ is an integrally closed domain. For any $c$, we endow $R$ with the canonical log-structure (see \S\ref{logstr}). If $c=\pi$, there is a natural chart for this log-structure, namely the map
\[ \mathbb{N}^r\to R:(n_1,...,n_r)\mapsto \prod_{i=1}^r T_i^{n_i} \]
which makes $\Spec(R)\to\Spec(K^+)$ into a a log-smooth morphism, where $\Spec(K^+)$ is given the canonical log-structure (see \S\ref{logstr}).

\subsubsection{}\label{smalln}
Fix $c\in\left\lbrace 1,\pi\right\rbrace $. Fix a sequence $(c=c_0,c_1,c_2,...)$ of elements of $\bar{K}$ satisfying $c_{n+1}^p=c_n$. If $c=1$ we take the $c_n$ to be primitive roots of unity. For any $n\in\mathbb{N}$, let $K_n^+$ be the normalization of $K^+$ in the extension $K_n$ of $K$ generated by $c_n$. Define
\[ O(c)_n:=K^+_n[T_1^{(n)},T_2^{(n)},...,T_r^{(n)},T_{r+1}^{(n)\pm 1},...,T_{d+1}^{(n)\pm 1}]/(T_1^{(n)}\cdots T_r^{(n)}-c_n). \]
Of course $O(c)_0=O(c)$. There is a canonical homomorphism $K_n^+\to K_{n+1}^+$ sending $c_n$ to $c_{n+1}^p$, which extends to a canonical homomorphism
\[ O(c)_n\to O(c)_{n+1}:T_i^{(n)}\mapsto T_i^{(n+1)p}. \]
Also note that there is a canonical isomorphism
\begin{eqnarray*}
O(c)_n[1/p] &\cong & K_n[X_2^{\pm 1},...,X_{d+1}^{\pm 1}] \\
T_1^{(n)} &\mapsto & \dfrac{c_n}{X_2\cdots X_r} \\
T_i^{(n)} &\mapsto & X_i\quad\text{for}\quad i\neq 1
\end{eqnarray*}
and hence $O(c)_n[1/p]$ is finite \'etale over $O(c)[1/p]$. Moreover, since $O(c)_n$ is regular (see Appendix) hence normal, it follows that we may identify $O(c)_n$ with the normalization of $O(c)$ in $K_n[T_2^{\pm p^{-n}},...,T_{d+1}^{\pm p^{-n}}]$ via this map.

If $R$ is a small $K^+$-algebra, then we define
\[ R_n:=R\otimes_{O(c)}O(c)_n. \]
If $K\subset L$ is a finite extension, then write $L_n=L\cdot K_n$ and let $L_n^+$ be the normalization of $K^+$ in $L_n$. Define
\[ O(c)_{n,L}:=O(c)_n\otimes_{K^+_n}L_n^{+}. \]

\subsubsection{}
Define
\[ K_{\infty}^+:=\varinjlim_{n\in\mathbb{N}}K^+_n. \]
We will now consider the almost ring theory of the pair $(K_{\infty}^+,\mathfrak{m}_{\infty})$, where $\mathfrak{m}_{\infty}\subset K^+_{\infty}$ is the maximal ideal. Note that there is a sequence of rational numbers $\varepsilon_n$ occuring as $p$-adic valuations of elements of $K^+_{\infty}$, and tending to zero with $n$. If $c=\pi$ then this is clear, and if $c=1$ then $\varepsilon_n=v_p(\zeta_{p^{n+1}}-1)=\frac{1}{p^n(p-1)}$, where $\zeta_{p^{n+1}}$ denotes a primitive $p^{n+1}$th root of unity. Since $K_{\infty}^+$ is a valuation ring we see that we indeed have $\mathfrak{m}_\infty^2=\mathfrak{m}_\infty$.

\subsubsection{}
Let $S$ be a finite integral $R$-algebra. We say that $S$ is \emph{\'etale in characteristic zero} if $R_K\to S_K$ is \'etale. For all $n$, let $S_\infty$ be the normalization of $R_\infty\otimes_R S$, where
\[ R_\infty:=\varinjlim_{n\in\mathbb{N}}R_n. \]
The key result we will need is Faltings' Almost Purity Theorem:

\begin{thm}[Faltings \cite{fa4}]\label{almostpurity}
If $S$ is a finite integral normal $R$-algebra, flat over $K^+$ and \'etale in characteristic zero, then the canonical homomorphism
\[ R_\infty\to S_\infty \]
is an almost \'etale covering of $K^+_\infty$-algebras.
\end{thm}

We define
\[ \tilde{R}:=R\otimes_{K^+}\bar{K}^+ \]
and
\[ \tilde{R}_\infty:=R_\infty\otimes_{K^+_\infty}\bar{K}^+. \]
If $K^+$ is integrally closed in $R$ and $p$ is not a unit in $R$, then by Proposition \ref{integral} these are integrally closed domains. If $S$ is a finite integral normal $R$-algebra which is \'etale in characteristic zero, then we define $\tilde{S}_\infty$ to be the normalization of $S_\infty\otimes_{K^+_\infty}\bar{K}^+$. Since $R_\infty\to\tilde{S}_\infty$ factors over $\tilde{R}_\infty\to \tilde{S}_\infty$, then by Lemma \ref{aechilfsatz} we deduce that $\tilde{R}_\infty\to \tilde{S}_\infty$ is the filtering inductive limit of almost \'etale coverings, in fact an almost \'etale covering.

\begin{cor}\label{frobsurjective}
\begin{enumerate}[(i)]
\item If $c=\pi$ (resp. $c=1$), then the Frobenius is surjective on $R_\infty/pR_\infty$ (resp. $\tilde{R}_\infty/p\tilde{R}_\infty$).
\item If $c=\pi$ (resp. $c=1$) and $S$ is a finite integral normal $R$-algebra (resp. $\tilde{R}$-algebra), \'etale in characteristic zero and flat over $K^+$, then the Frobenius is surjective on $S_\infty/pS_\infty$ (resp. $\tilde{S}_\infty/p\tilde{S}_\infty$).
\end{enumerate}
\end{cor}
\begin{proof} We do the case $c=\pi$, the case $c=1$ being similar. By the Almost Purity Theorem and Corollary \ref{frobsurjective0} it suffices to show (i). In this case, by \'etaleness over
\[ O(c)_{\infty}:=\varinjlim_{n}O(c)_{n} \]
it suffices to show that the Frobenius is surjective on the latter modulo $p$. First note that the Frobenius is surjective on $K^+_\infty/pK^+_\infty$: we have $K^+/pK^+=k[u]/(u^e)$ where $e$ is the ramification index of $K$ and $K^+_n=K^+[X]/(X^{p^n}-\pi)$, so $K^+_n/pK^+_n=K^+/pK^+[x]$ with $x^{p^n}=u$ and so $K^+_\infty/pK^+_\infty$ has surjective Frobenius. Now every element of $O(c)_{\infty}/pO(c)_{\infty}$ has the form
\[ \sum v_NT_1^{(m_1)n_1}\cdots T_{d+1}^{(m_{d+1})n_{d+1}}=\left( \sum v_N^{1/p}T_1^{(m_1+1)n_1}\cdots T_{d+1}^{(m_{d+1}+1)n_{d+1}}\right)^p \]
for some $v_N\in K^+_\infty/pK^+_\infty$ so the claim in this case is clear.
\end{proof}

\subsubsection{}\label{logstr}
Given a flat $K^+$-algebra $S$, we define the \emph{canonical log-structure} $\mathcal{L}_S$ on $S$ as the sheaf which associates to each \'etale $S$-algebra $S'$ the monoid
\[ (S'[1/p])^{\ast}\cap S'. \]
Note that this is an integral and saturated log-structure if $S$ is normal. For example, if $R$ is a \emph{small} $K^+$-algebra, then $
\mathcal{L}_R$ is the usual fine log-structure of the scheme $\Spec(R)$. Now if $S$ is a finite integral $R$-algebra, \'etale in characteristic zero, let $f:\Spec(S)\to\Spec(R)$ be the structure morphism and let $M_S$ be the log-structure on $\Spec(S)$ associated to the sheaf of monoids
\[ \{ x\in \mathcal{L}_S| \exists n\in\mathbb{N}: x^{p^n}\in f^{\ast}\mathcal{L}_R \}. \]
We call $M_S$ the \emph{$p$-saturation} of $\mathcal{L}_R$ in $\mathcal{L}_S$.

\begin{lemma}\label{logstrmonoid}
With notation as above.
\begin{enumerate}[(i)]
\item If $c=\pi$, then there is a saturated integral monoid with surjective Frobenius inducing a chart for the log-structure $M_{S_{\infty}}:=\varinjlim_n M_{S_n}$ on $\Spec(S_{\infty})$. In fact, we can take the same monoid as for $M_{R_{\infty}}$.
\item If $c=1$, then there is a saturated integral monoid with surjective Frobenius inducing a chart for $M_{\tilde{S}_{\infty}}:=\varinjlim_nM_{\tilde{S}_n}$ on $\Spec(\tilde{S}_{\infty})$, where $M_{\tilde{S}_n}:=\varinjlim_{K_n\subset L}M_{S_n\otimes_{K_n^+}L^+}$, the limit being over all finite extensions $L$ of $K_n$ in $\bar{K}$.
\end{enumerate}
\end{lemma}
\begin{proof}
(i): We first prove the assertion for $R_{\infty}$. Consider the map $\mathbb{N}[1/p]^r\to M_{R_{\infty}}:(\alpha_1,...,\alpha_r)\mapsto\prod_{i=1}^rT_i^{(n_i)m_i}$, where $\alpha_i=\dfrac{m_i}{p^{n_i}}$. We claim that it is a chart for $\mathcal{L}_{R_\infty}$, which will clearly imply the result in this case. To see this, up to replacing $R_{\infty}$ by its strict henselization at a point, we may assume that $R_{\infty}$ is a strictly henselian local ring. Then it is the limit of the strict henselizations of the $R_n$ at the image points, so we can assume that $R$ and $R_n$ are strictly henselian local rings for each $n\in\mathbb{N}$. Since $O(c)_n/(T_i^{(n)})$ is regular, then by \'etaleness it follows that $R_n/(T_i^{(n)})$ is a regular local ring, if $T_i$ is not a unit in $R$. So in this case $T_i^{(n)}$ generates a prime ideal of $R_n$. Moreover, since $\prod_{i=1}^rT_i^{(n)}=c_n$ every prime divisor of $c_n$ in $R_n$ is generated by one of the $T_i^{(n)}$. These are also the prime divisors of $p$ in $R_n$. It follows from this that for each $n\in\mathbb{N}$ there is a canonical isomorphism
\[ \mathcal{L}_{R_n}(R_n)/R_n^{\ast}\cong \dfrac{1}{p^n}\mathbb{N}^s: x\mapsto (v_{T_i}(x)) \]
where $v_{T_i}(-)$ denotes the valuation at the prime ideal generated by $T_i^{(n)}$ normalized by $v_{T_i}(T_i)=1$, assuming $T_i$ is not a unit in $R$ ($0\leq s\leq r$). These maps are compatible as $n$ varies, so taking the limit over $n$ to get an isomorphism
\[ \mathcal{L}_{R_\infty}(R_\infty)/R_\infty^{\ast}\cong \mathbb{N}[1/p]^s. \]
Since the natural map $\mathbb{N}[1/p]^r\to\mathcal{L}_{R_\infty}(R_\infty)/R_\infty^{\ast}\cong \mathbb{N}[1/p]^s$ is none other than the projection, this proves the claim.

We now show that this chart induces the log-structure $M_{S_{\infty}}$. As above for $R_{\infty}$, to see this we may assume $R$, $S_n$, and $S_{\infty}$ are a strictly henselian local rings. Suppose $x\in\mathcal{L}_{S_n}(S_n)$ such that $x^{p^m}\in\mathcal{L}_{R}(R)$, say $x^{p^m}=u\prod_{i=1}^rT_i^{n_i}$ for some $u\in R^{\ast}$. Up to increasing $n$ we may assume that $m\leq n$. Then $y:=\dfrac{x}{\prod_{i=1}^rT_i^{(m)n_i}}$ satisfies $y^{p^m}=u$, whence $y\in S_n^{\ast}$, since $S_n$ is an integrally closed domain (being a normal local ring). So $x$ and $\prod_{i=1}^rT_i^{(m)n_i}$ differ by a unit. Since this holds for all $x$, this implies the assertion.

(ii): Recall that since $c=1$, the log-structure for small $R$ can be given by the map
\[ \mathbb{N}\to R: n\mapsto\pi^n. \]
Let us fix a sequence $\{\pi_n\}$ of elements of $\bar{K}^+$ satisfying $\pi_0=\pi$ and $\pi_{n+1}^p=\pi_n$ for all $n\in\mathbb{N}$. We claim that the map
\[ \mathbb{N}[1/p]\to\tilde{S}_{\infty}: \dfrac{m}{p^n}\mapsto \pi_n^m \]
is a chart for the log-structure $M_{\tilde{S}_{\infty}}$. To see this we may (as above) reduce to the case $R$ and $\tilde{S}_{\infty}$ are strictly henselian local ring (and $\tilde{S}_{\infty}$ is an integrally closed domain, being the limit of normal local rings). If $x\in\mathcal{L}_{\tilde{S}_{\infty}}(\tilde{S}_{\infty})$ is such that $x^{p^m}\in\mathcal{L}_R(R)$, then we may write $x^{p^m}=u\pi^{n}$ for some $u\in R^{\ast}$ and $n\in\mathbb{N}$. Then $y:=\dfrac{x}{\pi_m^n}$ satisfies $y^{p^m}=u$, whence $y\in\tilde{S}_{\infty}^{\ast}$, since $\tilde{S}_{\infty}$ is an integrally closed domain. So $x$ differs from an element of the image of $\mathbb{N}[1/p]$ up to a unit, and this implies (ii).

\end{proof}

\subsection{Almost \'etale coverings of Fontaine rings}\label{aecfontaine}
By Corollary \ref{frobsurjective} we may, following Fontaine, construct the Fontaine rings
\begin{eqnarray*}
A^+_\infty &:=& \varprojlim_n H^0_{\crys}(\tilde{R}_\infty/p\tilde{R}_\infty|\Sigma_{n},\mathscr{O})\\
A^+_\infty(S) &:=& \varprojlim_n H^0_{\crys}(\tilde{S}_\infty/p\tilde{S}_\infty|\Sigma_{n},\mathscr{O}) 
\end{eqnarray*}
where $S$ is a finite integral normal $R$-algebra, \'etale in characteristic zero. Our aim in this section is to show that
\[ A^+_\infty/p^nA^+_\infty\to A^+_\infty(S)/p^nA^+_\infty(S) \]
is an almost \'etale homomorphism and from this deduce some consequences. We first set up the almost ring theory in this context.

\subsubsection{}
We will consider the almost ring theory of the ring $A_{\Inf}(K^+)=W(P(\bar{K}^+/p\bar{K^+}))$ with respect to the ideal $\mathfrak{a}$, union of the principal ideal $[\underline{p}]^\varepsilon$ for all positive rational exponents $\varepsilon>0$. Note that $[\underline{p}]^\varepsilon$ is not a zero divisor: this is true modulo $p$ since $P(\bar{K}^+/p\bar{K^+})$ is a valuation ring (hence an integral domain) and so in general since $A_{\Inf}(K^+)$ is $p$-adically separated and torsion free, being the ring of Witt vectors of a perfect ring of characteristic $p$. We view $W(\bar{K}^+/p\bar{K}^+)$ as a $A_{\Inf}(K^+)$-algebra under the homomorphism induced by the canonical homomorphism $P(\bar{K}^+/p\bar{K^+})\to \bar{K}^+/p\bar{K^+}$. Then the pair $(W(\bar{K}^+/p\bar{K}^+),\mathfrak{a}\cdot W(\bar{K}^+/p\bar{K}^+))$ is an example of the pairs $(W(\Lambda/p\Lambda),[\mathfrak{m}])$ considered in \S\ref{relativefrob}.

\subsubsection{Remark}\label{diffalmoststr}
The ring $A_{\cris}(K^+)/p^nA_{\cris}(K^+)$ can be viewed as an $A_{\Inf}(K^+)$-algebra in 2 different ways, the first is via the canonical map $A_{\Inf}(K^+)\to A_{\cris}(K^+)$ arising from the construction of $A_{\cris}(K^+)$ in the 1st proof of Theorem \ref{fontaine}, and the second is via the map $A_{\Inf}(K^+)\to W(\bar{K}^+/p\bar{K}^+)\to W_n(\bar{K}^+/p\bar{K}^+)\to A_{\cris}(K^+)/p^nA_{\cris}(K^+)$, where the last map is the canonical one arising the construction of $A_{\cris}(K^+)/p^nA_{\cris}(K^+)$ in the 2nd proof of Theorem \ref{fontaine}. We claim that the image of the ideal $\mathfrak{a}$ under these 2 maps is the same. To see this it suffices to note that by Remark (\ref{remarks}) (d) we have $[\underline{p}^{\varepsilon}]=[p^{\varepsilon/p^n}]$ in $A_{\cris}(K^+)/p^nA_{\cris}(K^+)$.

\subsubsection{} Let $R$ be a small integral $K^+$-algebra and let $T$ be a $p$-adically complete logarithmic $\Sigma$-algebra such that $R/pR$ is a $T$-algebra. In practice we will have $T=\Sigma$ or $T$ will be a formal $\Sigma$-lift of $R/pR$. For all $n\geq 1$, write $T_n:=T/p^nT$.

We assume that the log-structure of $T$ is fine and saturated, defined by a fine saturated monoid $\Lambda_T$. Let $S$ be a finite integral normal $R$-algebra \'etale in characteristic zero. We endow it with the log-structure $M_S$ of \S\ref{logstr} and endow $\Spec(S/pS)$ with the inverse image log structure. By taking limits this then defines a log-structure on $\tilde{S}_\infty/p
\tilde{S}_{\infty}$. Define
\begin{eqnarray*}
A^+_{\log,\infty,T} &:=& \varprojlim_n H^0_{\logcrys}(\tilde{R}_\infty/p\tilde{R}_\infty|T_{n+1},\mathscr{O})\\
A^+_{\log,\infty,T}(S) &:=& \varprojlim_n H^0_{\logcrys}(\tilde{S}_\infty/p\tilde{S}_\infty|T_{n+1},\mathscr{O}).
\end{eqnarray*}

\begin{prop}\label{etalefontainering}
The canonical homomorphism
\[ A^+_{\log,\infty,T}/p^nA^+_{\log,\infty,T}\to A^+_{\log,\infty,T}(S)/p^nA^+_{\log,\infty,T}(S) \]
is an almost \'etale homomorphism of $A_{\Inf}(K^+)$-algebras.
\end{prop}
\begin{proof}
Firstly, by Corollary \ref{wittetale} the canonical homomorphism
\[ W_n(\tilde{R}_\infty/p\tilde{R}_\infty)\to W_n(\tilde{S}_\infty/p\tilde{S}_\infty) \]
is almost \'etale of $W_n(\bar{K}^+/p\bar{K}^+)$-algebras. Now, by Lemma \ref{logstrmonoid} there is a saturated integral monoid $M$ with surjective Frobenius giving a chart for the log-structures of both $\tilde{R}_{\infty}$ and $\tilde{S}_{\infty}$, so by Corollary \ref{frobsurjective} we may construct $A^+_{\log,\infty,T}(S)$ as in the second proof of Theorem \ref{fontaine}. Let $M$ be the saturated integral monoid defining the log-structure on $\tilde{R}_{\infty}$, hence also on $\tilde{S}_{\infty}$ (see Lemma \ref{logstrmonoid}), and set $Q:=M\oplus\Lambda_T$. Let $L:=\ker(Q^{\gp}\to M^{\gp})$, where the map $Q\to M$ is induced by the maps $M\to M:m\mapsto m^{p^n}$ and the structure map $\Lambda_T\to M$. By construction, $A^+_{\log,\infty,T}(S)/p^nA^+_{\log,\infty,T}(S)$ is a quotient of a divided power hull of $\left( W_n(\tilde{S}_\infty/p\tilde{S}_\infty)\otimes_{\mathbb{Z}}T_{n}\right) _{\log}$, where by definition (\emph{cf.} the second proof of Theorem \ref{fontaine})
\[ \left( W_n(\tilde{S}_\infty/p\tilde{S}_\infty)\otimes_{\mathbb{Z}}T_{n}\right) _{\log} =
(W_n(\tilde{S}_\infty/p\tilde{S}_\infty)\otimes_{\mathbb{Z}}T_{n})\otimes_{\mathbb{Z}}\mathbb{Z}[L]. \]
Now, the canonical homomorphism
\[ (W_n(\tilde{R}_\infty/p\tilde{R}_\infty)\otimes_{\mathbb{Z}}T_{n})\otimes_{\mathbb{Z}}\mathbb{Z}[L]\to  (W_n(\tilde{S}_\infty/p\tilde{S}_\infty)\otimes_{\mathbb{Z}}T_{n})\otimes_{\mathbb{Z}}\mathbb{Z}[L] \]
is almost \'etale (Cor. \ref{wittetale}). This implies that we have a commutative diagram
\[ \begin{CD}
(W_n(\tilde{R}_\infty/p\tilde{R}_\infty)\otimes_{\mathbb{Z}}T_{n})\otimes_{\mathbb{Z}}\mathbb{Z}[L] @>>>  (W_n(\tilde{S}_\infty/p\tilde{S}_\infty)\otimes_{\mathbb{Z}}T_{n})\otimes_{\mathbb{Z}}\mathbb{Z}[L] \\
@V{\Phi^n\otimes 1\otimes 1}VV @V{\Phi^n\otimes 1\otimes 1}VV \\
(\tilde{R}_\infty/p\tilde{R}_\infty\otimes_{\mathbb{Z}}T_{n})\otimes_{\mathbb{Z}}\mathbb{Z}[L] @>>> (\tilde{S}_\infty/p\tilde{S}_\infty\otimes_{\mathbb{Z}}T_{n})\otimes_{\mathbb{Z}}\mathbb{Z}[L] \\
@V{L\mapsto 1}VV @V{L\mapsto 1}VV \\
\tilde{R}_\infty/p\tilde{R}_\infty\otimes_{\mathbb{Z}}T_{n} @>>> \tilde{S}_\infty/p\tilde{S}_\infty\otimes_{\mathbb{Z}}T_{n} \\
@VVV @VVV \\
\tilde{R}_\infty/p\tilde{R}_\infty @>>> \tilde{S}_\infty/p\tilde{S}_\infty
\end{CD} \]
in which all three squares are almost cofibred, where we denote by $\Phi^n:W_n(A)\to A:(a_0,...,a_{n-1})\mapsto a_0^{p^n}$. In particular, the commutative diagram
\[ \begin{CD}
(W_n(\tilde{R}_\infty/p\tilde{R}_\infty)\otimes_{\mathbb{Z}}T_{n})\otimes_{\mathbb{Z}}\mathbb{Z}[L] @>>>  (W_n(\tilde{S}_\infty/p\tilde{S}_\infty)\otimes_{\mathbb{Z}}T_{n})\otimes_{\mathbb{Z}}\mathbb{Z}[L] \\
@VVV @VVV \\
\tilde{R}_\infty/p\tilde{R}_\infty @>>> \tilde{S}_\infty/p\tilde{S}_\infty
\end{CD} \]
is almost cofibred. Taking divided power hulls for the kernels of the vertical maps and then taking the quotient by the ideal generated by all divided powers of elements of the form $q_l\otimes l-p_l\otimes 1$ (where $q_l,p_l\in Q$ satisfy $l=p_l/q_l\in L\subset Q^{\gp}$), it follows from the next lemma that the canonical map
\[ \begin{CD}
A^+_{\log,\infty,T}/p^nA^+_{\log,\infty,T}\otimes_{(W_n(\tilde{R}_\infty/p\tilde{R}_\infty)\otimes_{\mathbb{Z}}T_{n})\otimes_{\mathbb{Z}}\mathbb{Z}[L]}(W_n(\tilde{S}_\infty/p\tilde{S}_\infty)\otimes_{\mathbb{Z}}T_{n})\otimes_{\mathbb{Z}}\mathbb{Z}[L] \\
@VVV \\
A^+_{\log,\infty,T}(S)/p^nA^+_{\log,\infty,T}(S)
\end{CD} \]
is an almost isomorphism of $W_n(\bar{K}^+/p\bar{K}^+)$-algebras. This proves that the homomorphism of the statement is an almost \'etale morphism of $W_n(\bar{K}^+/p\bar{K}^+)$-algebras. By Remark (\ref{diffalmoststr}), it follows that it is also an almost \'etale morphism of $A_{\Inf}(K^+)$-algebras. \end{proof}

In the proof we have made use of the following almost analogue of a well-known result on divided power hulls.

\begin{lemma}
Suppose there is a homomorphism of rings $A\to B$ making $B$ into an almost flat $A$-algebra. Then for any ideal $I\subset A$ the canonical map
\[ D_A(I)\otimes_A B\to D_B(I\cdot B) \]
is an almost isomorphism, where $D_A(I)$ denotes the divided power hull of $A$ for the ideal $I$.
\end{lemma}
\begin{proof}
For any ring $A$ and any module $M$ let $\Gamma_A(M)$ denote the divided power algebra generated by $M$ (see e.g. \cite[3.9]{bertogus}). Since $B$ is an almost flat $A$-algebra, the canonical map $I\otimes_A B\to I\cdot B$ is an almost isomorphism, hence $\Gamma_B(I\otimes_A B)\thickapprox \Gamma_B(I\cdot B)$ by \cite[Lemma 8.1.13]{almost}. But
\[ \Gamma_B(I\otimes_A B)\cong\Gamma_A(I)\otimes_AB \]
so $\Gamma_B(I\cdot B)$ is an almost flat $\Gamma_A(I)$-algebra. Hence if $J\subset\Gamma_A(I)$ denotes the ideal defining $D_A(I)$, then we have $J\otimes_AB\cong J\otimes_{\Gamma_A(I)}\Gamma_B(I\otimes_AB)\thickapprox J\cdot \Gamma_B(I\cdot B)$, and therefore
\[ D_A(I)\otimes_AB=\dfrac{\Gamma_A(I)\otimes_AB}{J\otimes_A B}\thickapprox \dfrac{\Gamma_B(I\cdot B)}{J\cdot\Gamma_B(I\cdot B)}=D_B(I\cdot B). \]
\end{proof}

\subsubsection{} Assume that $K^+$ is integrally closed in $R$, so that $\tilde{R}_{\infty}$ is an integral domain (Prop. \ref{integral}). Define $\bar{R}$ to be the normalization of $R$ in the maximal profinite connected \'etale covering of $R[1/p]$ (taking for base point an algebraic closure of $Q(\tilde{R}_{\infty})$) . Then we have
\[ \pi_1(\Spec(\tilde{R}[1/p]))=\Gal(\bar{R}[1/p]/\tilde{R}[1/p]) \]
and $\tilde{R}_\infty[1/p]\to\bar{R}[1/p]$ is the inductive limit of Galois coverings of $\tilde{R}_\infty[1/p]$. Define
\[ \Delta:=\Gal(\bar{R}[1/p]/\tilde{R}[1/p]) \]
and
\[ \Delta_\infty:=\Gal(\tilde{R}_\infty[1/p]/\tilde{R}[1/p]). \]
Then $\Delta_\infty$ is a quotient of $\Delta$ and let $\Delta_0:=\ker(\Delta\to\Delta_\infty)$. With the log-structure of \S\ref{logstr} on $\bar{R}$, define
\[ A^+_{\log,T}:=\varprojlim_n H^0_{\logcrys}(\bar{R}/p\bar{R}|T_{n+1},\mathscr{O}).\]

\begin{lemma}
If $K^+$ is integrally closed in $R$, then the homomorphism $\tilde{R}_\infty\to\bar{R}$ is the filtering inductive limit of almost Galois coverings.
\end{lemma}
\begin{proof}
Let us first show that every non-zero almost \'etale covering of $\tilde{R}_\infty$ is almost faithfully flat. Consider the ring $(\tilde{R}_{\infty})_{\ast}:=\Hom_{K^+_{\infty}}(\mathfrak{m}_{\infty},\tilde{R}_{\infty})$ defined in \cite[\S 2.2.9]{almost}. If $\varphi\in\Hom_{K^+_{\infty}}(\mathfrak{m}_{\infty},\tilde{R}_{\infty})$ is an idempotent, then for all $x,y\in\mathfrak{m}_{\infty}$
\[ x\varphi(y)=\varphi(xy)=\varphi(x)\cdot\varphi(y) \]
so $x=\varphi(x)$, whence $\varphi=\id$ in $(\tilde{R}_{\infty})_{\ast}$. So $\tilde{R}_{\infty}$ has no almost idempotents, hence every almost projective $\tilde{R}_{\infty}$-module of finite rank is either almost zero or everywhere of constant non-zero rank, hence almost faithfully flat.

Now, let $S$ be a finite integral normal $R$-algebra, such that $R[1/p]\to S[1/p]$ is a Galois covering. Let $L^+$ be the integral closure of $K^+$ in $S$, and write $L=Q(L^+)$. Then by the almost purity theorem, the homomorphism $R_{\infty}\to S_\infty$ is an almost \'etale covering and factors over the almost \'etale covering $R_\infty\to R_{\infty,L}$. By Lemma \ref{aechilfsatz} it follows that $R_{\infty,L}\to S_\infty$ is an almost \'etale covering. It is easy to see that it must also be an almost Galois covering. Let $G$ be the Galois group of $R_{\infty,L}[1/p]\to S_\infty[1/p]$. The same argument with $S$ replaced by $S\otimes_{L^+}E^+$ for $E^+$ the normalization of $L^+$ in a finite Galois extension $L\subset E$ proves that
\[ R_{\infty,E}\to S_{\infty,E} \]
is an almost Galois covering, where $S_{\infty,E}$ is the normalization of $R_{\infty}\otimes_RS\otimes_{L^+}E^+$. Let $\tilde{S}_{\infty}:=\varinjlim_ES_{\infty,E}$. Taking the limit over $E$ we get a homomorphism
\[ \tilde{R}_\infty\to \tilde{S}_{\infty} \]
which is an almost Galois covering: there is a canonical almost isomorphism
\[ \varinjlim_ES_{\infty,E}\otimes_{R_{\infty,E}}S_{\infty,E}\approx \varinjlim_ES_{\infty,E}\times G.\]
\end{proof}

In particular, $\tilde{R}_\infty\to\bar{R}$ is almost faithfully flat.

\begin{cor}\label{almostgalois}
The canonical homomorphism
\[ A^+_{\log,\infty,T}/p^nA^+_{\log,\infty,T}\to A^+_{\log,T}/p^nA^+_{\log,T} \]
is the filtering inductive limit of almost Galois coverings.
\end{cor}
\begin{proof}
The proof is similar to that of Proposition \ref{etalefontainering}, so we refer there for details and point out the differences in this case. First note that since $\tilde{R}_\infty/p\tilde{R}_\infty\to\bar{R}/p\bar{R}$ is the filtering inductive limit of almost Galois coverings, by Cor. \ref{wittgalois} so is
\[ W_n(\tilde{R}_\infty/p\tilde{R}_\infty)\to W_n(\bar{R}/p\bar{R}). \]
As in the proof of Proposition \ref{etalefontainering} define $Q:=M\oplus\Lambda_T$ and
\[ L:=\ker(Q^{\gp}\to M^{\gp}). \]
Looking at the proof of Proposition \ref{etalefontainering} we note that we have an almost cofibred square
\[ \begin{CD}
(W_n(\tilde{R}_\infty/p\tilde{R}_\infty)\otimes_{\mathbb{Z}}T_{n})\otimes_{\mathbb{Z}}\mathbb{Z}[L] @>>>  (W_n(\bar{R}/p\bar{R})\otimes_{\mathbb{Z}}T_{n})\otimes_{\mathbb{Z}}\mathbb{Z}[L] \\
@VVV @VVV \\
\tilde{R}_\infty/p\tilde{R}_\infty @>>> \bar{R}/p\bar{R}.
\end{CD} \]
Hence we deduce that the canonical map
\[ A^+_{\log,\infty,T}/p^nA^+_{\log,\infty,T}\otimes_{(W_n(\tilde{R}_\infty/p\tilde{R}_\infty)\otimes_{\mathbb{Z}}T_{n})\otimes_{\mathbb{Z}}\mathbb{Z}[L]}(W_n(\bar{R}/p\bar{R})\otimes_{\mathbb{Z}}T_{n})\otimes_{\mathbb{Z}}\mathbb{Z}[L]\to A^+_{\log,T}/p^nA^+_{\log,T} \]
is an almost isomorphism, and since tensor product commutes with inductive limits, we are done.
\end{proof}

\begin{cor}\label{invariants0}
\begin{enumerate}[(i)]
\item The canonical map
\[ A^+_{\log,\infty,T}/p^nA^+_{\log,\infty,T}\to \left( A^+_{\log,T}/p^nA^+_{\log,T}\right)^{\Delta_0} \]
is an almost isomorphism.
\item $A^+_{\log,T}/p^nA^+_{\log,T}$ is a discrete $\Delta$-module and for all $i\neq 0$ we have
\[ H^i(\Delta_0,A^+_{\log,T}/p^nA^+_{\log,T})\thickapprox 0. \]
\end{enumerate}
\end{cor}
\begin{proof}
Part (i) follows from Proposition \ref{galois} (i) by taking the inductive limit. For (ii), note that $A^+_{\log,T}/p^nA^+_{\log,T}$ is a quotient of a divided power hull of $(W_n(\bar{R}/p\bar{R})\otimes_{\mathbb{Z}}T_{n})\otimes_{\mathbb{Z}}\mathbb{Z}[L]$. Now, we may view $M$ as a submonoid of $\bar{R}$. In this way we see that every element of $M^{\gp}$, hence also $Q^{\gp}$, is fixed by an open subgroup of $\Delta$. So every element of $(W_n(\bar{R}/p\bar{R})\otimes_{\mathbb{Z}}T_{n})\otimes_{\mathbb{Z}}\mathbb{Z}[L]$ is also fixed by an open subgroup of $\Delta$. Thus, every element of $A^+_{\log,T}/p^nA^+_{\log,T}$ is fixed by an open subgroup of $\Delta$, which proves the first assertion of (ii). As in (i), the second assertion of (ii) follows from Proposition \ref{galois} (iii) by taking the inductive limit.
\end{proof}

In particular, via the Hochschild-Serre spectral sequence of Galois cohomology we deduce canonical almost isomorphisms for all $i$
\[ H^i(\Delta_\infty,A^+_{\log,\infty,T}/p^nA^+_{\log,\infty,T})\thickapprox H^i(\Delta,A^+_{\log,T}/p^nA^+_{\log,T}). \]
This almost isomorphism will enable us to express the right-hand side in terms of crystalline cohomology.

\section{Geometric Galois cohomology of Fontaine rings}
In this section we construct a natural de Rham resolution of the Fontaine rings $A^+_{\log,\Sigma_{n}}$ and then we compute the (geometric) Galois cohomology of its components. Assume throughout that $R$ is an \'etale localization of a small integral $K^+$-algebra and that $K^+$ is integrally closed in $R$.

\subsection{Canonical de Rham resolutions of Fontaine rings}
\subsubsection{} If $c=1$, then let $v=1$, and if $c=\pi$ then let $v=u$, where $\Sigma_n=W_n(k)[u]\left\langle u^e\right\rangle $. Define
\[ \Theta(c):=\Sigma [T_1,...,T_r,T_{r+1}^{\pm 1},...,T_{d+1}^{\pm 1}]/(T_1\cdots T_r -v). \]
Since $R/pR$ is an \'etale $O(c)/pO(c)$-algebra, for all $n\geq 1$ there exists an \'etale $\Theta(c)/p^n\Theta(c)$-algebra $\mathcal{R}_n$ lifting $R/pR$, unique up to canonical isomorphism. Fix a projective system $\left\{\mathcal{R}_n\right\}_{n\in\mathbb{Z}_{>0}}$ of compatible lifts and define $\mathcal{R}:=\varprojlim_n\mathcal{R}_n$. We endow $\mathcal{R}_n$ with the log-structure associated to
\[ \mathcal{N}:\mathbb{N}^{r}\to\mathcal{R}_n:(n_1,...,n_r)\mapsto\prod_{i=1}^r T_i^{n_i} \]
in the case $c=\pi$, and
\[ \mathcal{N}:\mathbb{N}\to\mathcal{R}_n:n\mapsto u \]
in the case $c=1$.

\subsubsection{}\label{resolution}
Consider the canonical map
\[ h:\Spec(\bar{R}/p\bar{R})\to\Spec(R/pR) \]
and the associated morphism of log-crystalline topoi with respect to the DP-base $\Sigma_n$. By Proposition \ref{qcohcrystal}, $h_\ast\mathscr{O}$ is a quasi-coherent crystal of $\mathscr{O}$-modules on $(R/pR|\Sigma_n)_{\logcrys}$. Since $\Spec(\mathcal{R}_n)$ is a log-smooth lift of $\Spec(R/pR)$, by \cite[Thm. 6.2]{log} there is an integrable quasi-nilpotent logarithmic connection $d$ on $h_\ast\mathscr{O}(\mathcal{R}_n)$ whose associated de Rham complex computes the log-crystalline cohomology of $\bar{R}/p\bar{R}$ over the DP-base $\Sigma_n$. Since this cohomology vanishes in non-zero degree by Theorem \ref{fontaine}, it follows that the augmentation
\[ A^+_{\log,\Sigma}/p^nA^+_{\log,\Sigma}\cong H^0_{\logcrys}(\bar{R}/p\bar{R}|\Sigma_n,\mathscr{O})\to h_\ast\mathscr{O}(\mathcal{R}_n)\otimes_{\mathcal{R}_n}\omega_{\mathcal{R}_n/\Sigma_n}^{\bullet}  \]
is a quasi-isomorphism, where $\omega_{\mathcal{R}_n/\Sigma_n}^{i}:=\wedge^i_{\mathcal{R}_n}\omega_{\mathcal{R}_n/\Sigma_n}^{1}$ and $\omega_{\mathcal{R}_n/\Sigma_n}^{1}$ is the sheaf of logarithmic differentials (\cite[1.7]{log}). This quasi-isomorphism is the analogue of Poincar\'e's lemma which we will use in our comparison strategy. Our aim is to (almost) compute the $\Delta$-cohomology of the components of this resolution.

From now on we write
\[ A^+:=A^+_{\log,\Sigma} \]
and
\[ M^+:=\varprojlim_n H^0_{\logcrys}(\bar{R}/p\bar{R}|\mathcal{R}_n,\mathscr{O}).  \]
Note that $M^+/p^nM^+\cong h_\ast\mathscr{O}(\mathcal{R}_n)$.

\subsubsection{}
Let
\[ \tilde{O}(c)_\infty:=\varinjlim_{n,L}O(c)_{n,L} \]
so that $\tilde{R}_\infty=R\otimes_{O(c)} \tilde{O}(c)_\infty$. The same argument as above can be applied to the morphism
\[ h_\infty:\Spec(\tilde{O}(c)_\infty/p\tilde{O}(c)_\infty)\to\Spec(O(c)/pO(c)). \]
By Proposition \ref{qcohcrystal} the sheaf $h_{\infty,\ast}\mathscr{O}$ is a quasi-coherent crystal of $\mathscr{O}$-modules on the site $(O(c)/pO(c)|\Sigma_n)_{\logcrys}$ and hence we have a canonical isomorphism
\[ h_{\infty,\ast}\mathscr{O}(\mathcal{R}_n)\cong h_{\infty,\ast}\mathscr{O}\left( \Theta(c)/p^n\Theta(c)\right) \otimes_{\Theta(c)/p^n\Theta(c)}\mathcal{R}_n. \]
Note that in the notation of the previous section we have
\[ h_{\infty,\ast}\mathscr{O}(\mathcal{R}_n)=A^+_{\log,\infty,\mathcal{R}}/p^nA^+_{\log,\infty,\mathcal{R}}. \]
Also, there is an integrable quasi-nilpotent logarithmic connection $d$ on $h_{\infty,\ast}\mathscr{O}\left(\mathcal{R}_n\right) $ whose associated de Rham complex is a resolution of
\[ A^+_{\log,\infty,\Sigma}/p^nA^+_{\log,\infty,\Sigma}. \]
Define
\[ A^+_\infty=A^+_\infty(R):=A^+_{\log,\infty,\Sigma} \]
and
\[ M^+_\infty=M^+_\infty(R):=\varprojlim_n h_{\infty,\ast}\mathscr{O}(\mathcal{R}_n). \]

\subsubsection{}
Note that, in the notation of \S \ref{aecfontaine}, we have
\begin{eqnarray*}
M^+_\infty &=& A^+_{\log,\infty,\mathcal{R}}\\
M^+ &=& A^+_{\log,\mathcal{R}}
\end{eqnarray*}
in particular, the canonical map $M^+_\infty/p^nM^+_\infty\to M^+/p^nM^+$ is the inductive limit of almost Galois coverings, $M^+_\infty/p^nM^+_\infty\thickapprox \left( M^+/p^nM^+\right)^{\Delta_0}$, and we have canonical almost isomorphisms for all $i$
\[ H^i(\Delta_\infty, M^+_\infty/p^nM^+_\infty)\thickapprox H^i(\Delta, M^+/p^nM^+). \]

\subsubsection{} Define
\begin{eqnarray*}
A_{\cris,\infty}(R) &:=& \varprojlim_n H^0_{\crys}(\tilde{R}_\infty/p\tilde{R}_\infty|W_{n+1}(k),\mathscr{O}) \\
A_{\cris}(R) &:=& \varprojlim_n H^0_{\crys}(\bar{R}/p\bar{R}|W_{n+1}(k),\mathscr{O})
\end{eqnarray*}
This is just the classical crystalline cohomology, i.e. we ignore the log-structures.

\begin{lemma}\label{xi}
Let $A$ be an integrally closed domain of characteristic zero such that the Frobenius is surjective on $A/pA$. Assume that $A$ contains a sequence $\underline{p}:=(p,p^{1/p},p^{1/p^2},...)$ of $p$-power roots of $p$ satisfying $(p^{1/p^{n+1}})^p=p^{1/p^n}$ for all $n$. Let
\[ \theta:W(P(A/pA))\to\hat{A} \]
denote the canonical map constructed in the proof of Theorem \ref{fontaine}. Then $\ker(\theta)$ is a principal ideal generated by $\xi:=[\underline{p}]-p$, where we identify $\underline{p}$ with the element it defines in $P(A/pA)$.
\end{lemma}
\begin{proof}
We assume $x\in\ker(\theta\mod p)$. Write $x=(x^{(n)})_n$ with $x^{(n+1)p}=x^{(n)}$ for all $n$. Then we have $x^{(1)p}=0$. Let $\hat{x}^{(1)}\in A$ be a lift of $x^{(1)}$. Then $\hat{x}^{(1)p}=py$ for some $y\in A$ so $\frac{\hat{x}^{(1)}}{p^{1/p}}\in A$ because $A$ is integrally closed. Hence $x^{(1)}\in p^{1/p}\cdot A/pA$. Continuing in this manner we find that $x\in\underline{p}\cdot P(A/pA)$. This proves the claim modulo $p$, and since $W(P(A/pA))$ is $p$-adically complete and $\hat{A}$ is $p$-torsion free the lemma follows.
\end{proof}

\begin{prop}\label{flat2}
Let $A$ be a flat $\bar{K}^+$ algebra. Assume that $A$ is an integrally closed domain, and that the Frobenius is surjective on $A/pA$. Let
\[ A_{\cris}(A/pA):=\varprojlim_nH^0_{\crys}(A/pA|W_n(k),\mathscr{O}). \]
Then
\begin{enumerate}[(i)]
\item There is a canonical isomorphism
\[ A_{\cris}(A/pA)/p\cong (A/pA)\otimes_{\bar{K}^+/p\bar{K}^+}A_{\cris}(K^+)/p. \]
\item There is a canonical isomorphism
\[  A_{\cris}(A/pA)/p^n\cong  W_n(P(A/pA))\otimes_{W_n(P(\bar{K}^+/p\bar{K}^+))}A_{\cris}(K^+)/p^n. \]
\item $A_{\cris}(A/pA)/p^n$ is flat over $W_n(k)$.
\end{enumerate}
\end{prop}
\begin{proof}
(i): By construction (\emph{cf.} the second proof of Thm. \ref{fontaine}), $A_{\cris}(A/pA)/p$ is the divided power hull of $A/pA$ for the kernel of the Frobenius, i.e. for the ideal generated by $p^{1/p}$. Since the same is true for $A_{\cris}(K^+)/p$ and $A/pA$ is flat over $\bar{K}^+/p\bar{K}^+$, the result follows from the fact that formation of divided power hulls commutes with flat tensor product.

(ii): Let $P:=P(\bar{K}^+/p\bar{K}^+)$. We first show that $P(A/pA)$ is flat over $P$. Since $P$ is a valuation ring, every finitely generated ideal is principal. Let $x=(x^{(n)})\in P$ and $a=(a^{(n)})\in P(A/pA)$ and consider $a\otimes x\in (a)\otimes_{P}P(A/pA)$. Then $xa=0$ if and only if $x^{(n)}a^{(n)}=0$ for all $n$, if and only if $a^{(n)}\otimes x^{(n)}=0$ in $(a^{(n)})\otimes_{\bar{K}^+/p\bar{K}^+}A/pA$ for all $n$ ($A/pA$ being flat over $\bar{K}^+/p\bar{K}^+$), if and only if $a\otimes x=0$. This proves the flatness claim. Now since $P(A/pA)$ is a perfect ring, it follows that $W(P(A/pA))$ is $p$-torsion free, hence by Bourbaki's flatness criterion we deduce that $W_n(P(A/pA))$ is a flat $W_n(P)$-algebra for all $n$. Now by Lemma \ref{xi}, we know that $A_{\cris}(A/pA)/p^n$ is the divided power hull of $W_n(P(A/pA))$ for the ideal generated by $\xi$, hence since taking divided power hulls commutes with flat base change, (ii) follows. 

(iii): follows easily from these considerations.
\end{proof}

Let $1\leq i\leq d+1$. Note that by construction we have an element
\[ \underline{T_i}=(T_i^{(0)},T_i^{(1)},T_i^{(2)},...)\in P(\tilde{R}_\infty/p\tilde{R}_\infty). \]
If $\alpha=n/p^r$ with $n,r\in\mathbb{N}$, then we define
\[ T_i^{\alpha}:=T_i^{(r)n} \]
and
\[ \underline{T_i}^{\alpha}:=(T_i^{(r)n},T_i^{(r+1)n},...)\in P(\tilde{R}_\infty/p\tilde{R}_\infty). \]
We'll also usually write $[\underline{T_i}]^{\alpha}:=[\underline{T_i}^{\alpha}]$, where $[\cdot ]$ is the Teichm\"uller lift.

\begin{prop}\label{dppolynomialalgebra}
In the case $R=O(c)$, there is a canonical isomorphism
\[ M^+_\infty/p^n M^+_\infty\simeq  A_{\cris,\infty}(R)/p^n A_{\cris,\infty}(R)\left\langle X, X_2,X_3,...,X_{d+1}\right\rangle  \]
where the $X,X_2,...,X_{d+1}$ are indeterminates. In particular, $M^+_\infty/p^nM^+_\infty$ is a flat $W_n(k)$-module.
\end{prop}
\begin{proof}
The flatness claim will follow from Proposition \ref{flat2} once we show that $M^+_\infty/p^nM^+_\infty$ has the desired form. For this we check the universal property. Define
\[ \underline{T_i}:=(T_i,T_i^{(1)},T_i^{(2)},...)\in P(\tilde{R}_\infty/p\tilde{R}_\infty) \]
and $\underline{c}:=(c,c_1,c_2,...)=\prod_{i=1}^r\underline{T_i}$. Consider the ring
\[ C:=W_n(P(\tilde{R}_\infty/p\tilde{R}_\infty))\left[ X,\frac{1}{1+X}, X_2,\frac{1}{1+X_2},...,X_{d+1},\frac{1}{1+X_{d+1}}\right] . \]
Make $C$ into a $W_n[u][T_1,...,T_r,T_{r+1}^{\pm 1},...,T_{d+1}^{\pm 1}]/(\prod_{i=1}^r T_i-v)$-algebra as follows. If $c=\pi$ (resp. $c=1$), then send $u$ to $[\underline{\pi}]\cdot (1+X)^{-1}$ (resp. $u$ to $[\underline{\pi}]-X$), $T_1$ to $[\underline{T_1}]\dfrac{\prod_{i=2}^r(1+X_i)}{1+X}$ (resp. $T_1$ to $[\underline{T_1}]\dfrac{\prod_{i=2}^r(1+X_i)}{[\underline{c}]}$), and $T_i$ to $[\underline{T_i}]\cdot (1+X_i)^{-1}$ for $2\leq i\leq d+1$. For an affine object $(U\hookrightarrow T)$ of the site $(\tilde{R}_\infty/p\tilde{R}_\infty|\Theta(c)/p^n\Theta(c))_{\logcrys}$ define a map
\[ C\to\mathscr{O}_T(T) \]
extending the canonical map $\theta_T:W_n(P(\tilde{R}_\infty/p\tilde{R}_\infty))\to\mathscr{O}_T$ (see the proof of Theorem \ref{fontaine}) by sending $X$ to $\theta_T([\underline{c}])\cdot\pi^{-1}-1$ (resp. $X$ to $\theta_T([\underline{\pi}])-u$) and $X_i$ to $\theta_T([\underline{T_i}])\cdot T_i^{-1}-1$ (note that by Lemma \ref{nilexact} these elements exist in $\mathscr{O}_T(T)$). One checks easily that this is the unique map of $W_n[u][T_1,...,T_r,T_{r+1}^{\pm 1},...,T_{d+1}^{\pm 1}]/(\prod_{i=1}^r T_i-v)$-algebras extending $\theta_T$. It follows from Lemma \ref{xi} that the kernel of the map to $C\to \tilde{R}_\infty/p\tilde{R}_\infty$ is the ideal generated by $p$, $\xi$, and $X,X_2,...,X_{d+1}$. So the divided power hull $C^{\text{DP}}$ of $C$ for this ideal is precisely $A_{\cris,\infty}(R)/p^n A_{\cris,\infty}(R)\left\langle X, X_2,X_3,...,X_{d+1}\right\rangle$. By the uniqueness of $\theta_T$ we see that it suffices now to define a log-structure on $C^{\text{DP}}$ such that $C^{\text{DP}}$ defines an object of $(\tilde{R}_\infty/p\tilde{R}_\infty|\Theta(c)/p^n\Theta(c))_{\logcrys}$, i.e. the surjection $C^{\text{DP}}\to \tilde{R}_\infty/p\tilde{R}_\infty$ is an exact closed immersion. We show this in the case $c=\pi$, the case $c=1$ being similar. Recall (\S\ref{logstr}) that the log-structure $M_{\tilde{R}_\infty}$ is associated to
\begin{eqnarray*}
\mu:\mathbb{N}[1/p]^{r} &\to & \tilde{R}_\infty/p\tilde{R}_\infty \\
(\alpha_1,...,\alpha_r) &\mapsto & \prod_{i=1}^r T_i^{\alpha_i}.
\end{eqnarray*}
Let $Q=P(\mathbb{N}[1/p]^{r})\oplus\mathbb{N}^{r}$ and let $\lambda:Q\to\mathbb{N}[1/p]^r$ the map induced by the natural maps $P(\mathbb{N}[1/p]^{r})\to\mathbb{N}[1/p]^{r}$ and $\mathbb{N}^{r}\subset \mathbb{N}[1/p]^{r}$. Define $L:=\ker(\lambda^{\gp}:Q^{\gp}\to \mathbb{Z}^{r})$ and let $QL\subset Q^{\gp}$ be the submonoid consisting of elements of the form $ql$ with $q\in Q$ and $l\in L$. Since
\[ P(\mathbb{N}[1/p]^{r})\cong \mathbb{N}[1/p]^{r}:(\alpha\cdot p^{-n})_{n\in\mathbb{N}}\mapsto\alpha \]
we have
\[ P\left(\mathbb{N}[1/p]^{r}\right)^{\gp}\cong\mathbb{Z}[1/p]^{r}. \]
So $L$ is the kernel of the map
\begin{eqnarray*}
\mathbb{Z}[1/p]^{r}\oplus\mathbb{Z}^{r} &\to & \mathbb{Z}[1/p]^{r} \\
\left((\alpha_1,...,\alpha_r),(n_1,...,n_r)\right) &\mapsto & \left(\alpha_1+n_1,...,\alpha_r+n_r\right).
\end{eqnarray*}
That is, $L$ consists of the tuples $\left((-n_1,...,-n_r),(n_1,...,n_r)\right)$ with $n_i\in\mathbb{Z}$. Define $1+X_1:=\dfrac{1+X}{\prod_{i=2}^r(1+X_i)}$. Note that
\[ X,X_1,...,X_r\in\ker(C^{\text{DP} }\to \tilde{R}_\infty/p\tilde{R}_\infty) \]
and hence $1+X,1+X_1,...,1+X_r$ are units of $C^{\text{DP}}$. We endow $C^{\DP}$ with the log-structure associated to
\[ QL\to C^{\DP} \]
induced by the maps
\begin{eqnarray*}
P(\mathbb{N}[1/p]^{r+1}) &\to& C^{\DP}: (m^{(n)})_{n\in\mathbb{N}}\to [(\mu(m^{(n)}))_{n\in\mathbb{N}}] \\
\mathbb{N}^{r} &\to& \Theta(c)/p^n\Theta(c)\to C^{\DP} \\
L &\to & C^{\text{DP}}:\left((-n_1,...,-n_r),(n_1,...,n_r)\right) \mapsto \prod_{i=1}^r (1+X_i)^{-n_i}.
\end{eqnarray*}
It is easy to check that these maps are compatible and do indeed define a map $QL\to C^{\DP}$. With the natural maps $QL\to \mathbb{N}[1/p]^{r}$ and $C^{\text{DP}}\to\tilde{R}_\infty/p\tilde{R}_\infty$, we get a closed immersion of log schemes. Now by the same argument as that of Step 2 of the proof of Theorem \ref{fontaine} we see that it is an exact closed immersion, and hence we are done.
\end{proof}

\subsection{Computations in Galois cohomology}

\subsubsection{}\label{koszul} Note that $\Delta_\infty\cong\mathbb{Z}_p(1)^d$. Let $\sigma_2,...,\sigma_{d+1}$ be a choice of topological generators of $\Delta_\infty$. For each $2\leq i\leq d+1$, we have $\sigma_i(T_i^{(n)})=\zeta_{p^n}T_i^{(n)}$ for a $p^n$th root of unity $\zeta_{p^n}$. Define
\[ \underline{1}:=(1,\zeta_{p},\zeta_{p^n},...)\in P(\bar{K}^+/\bar{K}^+) \]
so that $\sigma(\underline{T_i})=\underline{1}\cdot\underline{T_i}$. Furthermore, define
\[ t:=\log([\underline{1}]). \]
Let us make some remarks on the Galois cohomology of the group $\Delta_\infty$. Let $\mathbb{Z}_p[[\Delta_\infty]]$ be the completion of the group ring $\mathbb{Z}_p[\Delta_\infty]$ in the topology induced by the profinite topology of $\Delta_\infty$. Then for any discrete $p$-torsion $\Delta_\infty$-module $N$, we have canonical isomorphisms for all $i$
\[ \Ext^i_{\mathbb{Z}_p[[\Delta_\infty]]}(\mathbb{Z}_p,N)\cong H^i(\Delta_\infty,N) \]
where the $\Ext$-group is taken in the category of topological $\mathbb{Z}_p[[\Delta_\infty]]$-modules. Since $\Delta_\infty\cong\mathbb{Z}_p(1)^d$, we have an isomorphism of rings
\[ \mathbb{Z}_p[[\Delta_\infty]]\simeq \mathbb{Z}_p[[\sigma_2-1,...,\sigma_{d+1}-1]]. \]
This implies that the Koszul complex $L:=\otimes_{\mathbb{Z}_p[[\Delta_\infty]]} L_i$, where $L_i$ is the complex defined
\[ \begin{CD}
0 @>>> \mathbb{Z}_p[[\Delta_\infty]] @>{\sigma_i-1}>> \mathbb{Z}_p[[\Delta_\infty]] @>>> 0,
\end{CD} \]
is a homological resolution of $\mathbb{Z}_p$ by free compact $\mathbb{Z}_p[[\Delta_\infty]]$-modules. Then we have an isomorphism of complexes (up to shifting the degree)
\[ \Hom_{\mathbb{Z}_p[[\Delta_\infty]]}(L,N)\cong L\otimes_{\mathbb{Z}_p[[\Delta_\infty]]}N \]
i.e. the Galois cohomology of $N$ can be computed using the Koszul complex $L\otimes_{\mathbb{Z}_p[[\Delta_\infty]]}N$.

\subsubsection{} 

\begin{lemma}\label{write}
Assume $R=O(c)$. Every element of  $A_{\cris,\infty}(R)/p^n A_{\cris,\infty}(R)$ can be written as a finite sum of the form
\[ \sum_{k\geq 0} x_k\xi^{[k]} \]
with $x_k\in W_n(P(\tilde{R}_\infty/p\tilde{R}_\infty))$ of the form
\[ \sum_m p^mv_m\prod_{i=1}^{d+1}[\underline{T_ i}]^{\alpha_{i,m}}\]
where $\alpha_{i,m}\in\mathbb{N}[1/p]$ for $1\leq i\leq r$ and $\alpha_{i,m}\in\mathbb{Z}[1/p]$ for $r+1\leq i\leq d+1$ and $v_m\in A_{\Inf}(K^+)$.
\end{lemma}
\begin{proof}
The lemma will follow from Lemma \ref{xi} once we show that we can get the $x_k$ in the desired form. Since the ring of Witt vectors $W_n(A)$ of a perfect ring $A$ of characteristic $p$ is equal to its subring of elements of the form $\sum_m p^m[a_m]$, where $a_m\in A$ and $[\cdot ]$ denotes the Teichm\"uller lift, it suffices to prove the claim modulo $p$. By definition, we obtain an element of $P(\tilde{R}_\infty/p\tilde{R}_\infty)$ by taking roots of $r\in\tilde{R}_\infty/p\tilde{R}_\infty$. We may write
\[ r=\sum_N v_N\prod_{i=1}^{d+1}T_i^{n_i}  \]
where $v_N\in\bar{K}^+/p\bar{K}^+$ and $n_i\in\mathbb{N}[1/p]$. Make a choice of $p$-power roots of elements of $\bar{K}^+/p\bar{K}^+$ (this will not affect the statement of the lemma). If $v\in \bar{K}^+/p\bar{K}^+$, then let $\underline{v}:=(v,v^{1/p},v^{1/p^2},...)\in P(\bar{K}^+/p\bar{K}^+)$ for the made choice of $p$-power roots of $v$. Define
\[ \underline{r}:=\sum_N\underline{v_N}\prod_{i=1}^{d+1}\underline{T_i}^{n_i}. \]
Taking $p$th roots of $r$ we get
\begin{eqnarray*}
r^{(1)} &=& \sum_N v_N^{1/p}\prod_{i=1}^{d+1}T_i^{(1)n_i}+p^{a_1/p}r_1 \\
r^{(2)} &=& \sum_N v_N^{1/p^2}\prod_{i=1}^{d+1}T_i^{(2)n_i}+p^{a_1/p^2}r_1^{p/p^2}+p^{a_2p/p^2}r_2\\
&\cdots & \\
(r^{(n)}) &=& \underline{r}+\underline{p}^{a_1}\underline{r_1}+\underline{p}^{a_2p}\underline{r_2}+...
\end{eqnarray*}
where $a_i\in\mathbb{N}$. Now recall that $\underline{p}$ has divided powers in $A_{\cris,\infty}(R)/p A_{\cris,\infty}(R)$ so we get
\[ (r^{(n)}) = \underline{r}+\underline{p}^{a_1}\underline{r_1}. \]
The lemma follows.
\end{proof}

\subsubsection{}
For all $2\leq i\leq d$ we define a sub-$\bar{K}^+$-algebra $\tilde{O}(c)_{\infty}^{(i)}\subset \tilde{O}(c)_{\infty}$ as follows:
\begin{itemize}
\item if $2\leq i\leq r$, then $\tilde{O}(c)_{\infty}^{(i)}$ is the sub-$\bar{K}^+$-algebra of $\tilde{O}(c)_{\infty}$ generated by $T_j^{\beta_j}$, $j\notin\left\lbrace 1,i\right\rbrace$, $2\leq j\leq d+1$, $\beta_j\in\mathbb{N}[1/p]$, and $(T_iT_1)^{\beta}$ with $\beta\in\mathbb{N}[1/p]$
\item if $r+1\leq i\leq d+1$ then $\tilde{O}(c)_{\infty}^{(i)}$ is the sub-$\bar{K}^+$-algebra of $\tilde{O}(c)_{\infty}$ generated by $T_j^{\beta_j}$, $j\neq i$, $2\leq j\leq d+1$, $\beta_j\in\mathbb{N}[1/p]$.
\end{itemize}

\begin{lemma}\label{subflat}
$\tilde{O}(c)_{\infty}$ is faithfully flat over $\tilde{O}(c)_{\infty}^{(i)}$.
\end{lemma}
\begin{proof}
The case $c=1$ being trivial, we only consider the case $c=\pi$. First assume $2\leq i\leq r$. Write $\tilde{O}(c)_{\infty}^{(i)}=\left(\varinjlim_nO(c)_{n}^{(i)}\right)\otimes_{K^+_{\infty}}\bar{K}^+$ where
\[ O(c)_{n}^{(i)}=\dfrac{K^+_n[Z^{(n)},T_j^{(n)}]_{2\leq j\leq d+1,j\neq i}}{(Z^{(n)}\prod_{2\leq j\leq r,j\neq i}T_j^{(n)}-c_n)} \]
where $Z^{(n)}=T_1^{(n)}T_i^{(n)}$. It suffices to check that $O(c)_{n}^{(i)}\to O(c)_{n}$ is faithfully flat. Note that
\[ O(c)_{n}=O(c)_{n}^{(i)}[T_1^{(n)},T_i^{(n)}]/(Z^{(n)}-T_1^{(n)}T_i^{(n)}). \]
Since $O(c)_{n}$ is an integral domain, it suffices to check that $O(c)_{n}^{(i)}/(Z^{(n)})\to O(c)_{n}/(Z^{(n)})$ is flat. But $O(c)_{n}/(Z^{(n)})=\dfrac{k[T_1^{(n)},T_i^{(n)}]}{(T_1^{(n)}T_i^{(n)})}\otimes_kO(c)_{n}^{(i)}/(Z^{(n)})$ and the claim follows. For $r+1\leq i\leq d+1$ the claim is obvious.
\end{proof}

\begin{lemma}
Fix some $i$, $2\leq i\leq d+1$, and let $\mathbb{Z}_p(1)$ be the factor of $\Delta_\infty$ corresponding to $\sigma_i$. There is an isomorphism of $\tilde{O}(c)_{\infty}^{(i)}[\mathbb{Z}_p(1)]$-modules
\[ \tilde{O}(c)_{\infty}=\bigoplus_{\alpha\in\mathbb{Z}[1/p]}E_{\alpha} \]
where $E_{\alpha}$ is the $\sigma_i$-eigenspace of eigenvalue $\zeta_{p^r}^n$, where $\alpha=n/p^r$ with $(n,p)=1$. $E_{\alpha}$ is a free $\tilde{O}(c)_{\infty}^{(i)}$-module of rank 1 with generator $e_{\alpha}$ given by
\begin{itemize}
\item if $2\leq i\leq r$ and $\alpha\geq 0$, then $e_{\alpha}=T_i^{\alpha}$
\item if $2\leq i\leq r$ and $\alpha<0$, then $e_{\alpha}=T_1^{-\alpha}$
\item if $r+1\leq i\leq d+1$, then $e_{\alpha}=T_i^{\alpha}$.
\end{itemize}
\end{lemma}
\begin{proof}
First assume that $2\leq i\leq r$. Clearly, every element $f\in\tilde{O}(c)_{\infty}$ has the form $f=\sum_{\beta,\gamma\in\mathbb{N}[1/p]}x_{\beta,\gamma}T_i^{\beta}T_1^{\gamma}$ with $x_{\beta,\gamma}$ in the sub-$\bar{K}^+$-algebra of $\tilde{O}(c)_{\infty}$ generated by $T_j^{\beta_j}$, $j\notin\left\lbrace 1,i\right\rbrace$, $2\leq j\leq d+1$, $\beta_j\in\mathbb{N}[1/p]$. So we can write
\[ f=\sum_{\alpha\geq 0}\sum_{\beta-\gamma=\alpha}x_{\beta,\gamma}(T_iT_1)^{\gamma}T_i^{\alpha}+ \sum_{\alpha<0}\sum_{\beta-\gamma=\alpha}x_{\beta,\gamma}(T_iT_1)^{\beta}T_1^{-\alpha}. \]
Hence
\[  \tilde{O}(c)_{\infty}=\sum_{\alpha\in\mathbb{Z}[1/p]}E_{\alpha}. \]
Let $X_{\alpha}=\sum_{\beta-\gamma=\alpha}x_{\beta,\gamma}(T_iT_1)^{\gamma}$ if $\alpha\geq 0$ and $X_{\alpha}=\sum_{\beta-\gamma=\alpha}x_{\beta,\gamma}(T_iT_1)^{\beta}$ otherwise. It remains to show that the writing $f=\sum_{\alpha\geq 0}X_{\alpha}T_i^{\alpha}+\sum_{\alpha<0}X_{\alpha}T_1^{-\alpha}$ is unique. Equivalently it suffices to show that if $f=0$ then $X_{\alpha}=0$ for all $\alpha$. It suffices to show this for the image of $f$ in $\tilde{O}(c)_{\infty}[1/p]\cong \varinjlim_n\bar{K}[T_2^{\pm p^{-n}},...,T_{d+1}^{\pm p^{-n}}]$, where it is evident. The case $i>r$ is similar.
\end{proof}

In the sequel we often write $A_{\cris,\infty}:=A_{\cris,\infty}(O(c))$. For $2\leq i\leq d$ define a sub-$A_{\cris}(K^+)$-algebra $A_{\cris,\infty}^{(i)}/p^nA_{\cris,\infty}^{(i)}\subset A_{\cris,\infty}/p^nA_{\cris,\infty}$ as follows:
\begin{itemize}
\item if $2\leq i\leq r$, then $A_{\cris,\infty}^{(i)}/p^nA_{\cris,\infty}^{(i)}$ the sub-$A_{\cris}(K^+)$-algebra of $A_{\cris,\infty}/p^nA_{\cris,\infty}$ generated by $[\underline{T_j}]^{\beta_j}$ with $j\notin\left\lbrace 1,i\right\rbrace$, $2\leq j\leq d+1$, $\beta_j\in\mathbb{N}[1/p]$, and $[\underline{T_iT_1}]^{\beta}$ with $\beta\in\mathbb{N}[1/p]$
\item if $r+1\leq i\leq d+1$, then $A_{\cris,\infty}^{(i)}/p^nA_{\cris,\infty}^{(i)}$ the sub-$A_{\cris}(K^+)$-algebra of $A_{\cris,\infty}/p^nA_{\cris,\infty}$ generated by $[\underline{T_j}]^{\beta_j}$, $j\neq i$, $2\leq j\leq d+1$, $\beta_j\in\mathbb{N}[1/p]$.
\end{itemize}

\begin{prop}\label{direct}
Fix some $i$, $2\leq i\leq d+1$, and let $\mathbb{Z}_p(1)$ be the factor of $\Delta_\infty$ corresponding to $\sigma_i$. There is an isomorphism of $A_{\cris,\infty}^{(i)}/p^nA_{\cris,\infty}^{(i)}[\mathbb{Z}_p(1)]$-modules
\[ A_{\cris,\infty}(O(c))/p^nA_{\cris,\infty}(O(c))\cong \bigoplus_{\alpha\in\mathbb{Z}[1/p]}F_{\alpha} \]
where $F_{\alpha}$ is the $\sigma_i$-eigenspace of eigenvalue $[\underline{1}]^{\alpha}$. $F_{\alpha}$  is a free $A_{\cris,\infty}^{(i)}/p^nA_{\cris,\infty}^{(i)}$-module of rank 1 with generator $e_{\alpha}$ given by
\begin{itemize}
\item if $2\leq i\leq r$ and $\alpha\geq 0$, then $e_{\alpha}=[\underline{T_i}]^{\alpha}$
\item if $2\leq i\leq r$ and $\alpha<0$, then $e_{\alpha}=[\underline{T_1}]^{-\alpha}$
\item if $r+1\leq i\leq d+1$, then $e_{\alpha}=[\underline{T_i}]^{\alpha}$.
\end{itemize}
\end{prop}
\begin{proof}
First assume $2\leq i\leq r$. If $n=1$, then by Proposition \ref{flat2} we have an isomorphism
\[ A_{\cris,\infty}/pA_{\cris,\infty}\cong (\tilde{O}(c)_{\infty}/p\tilde{O}(c)_{\infty})\otimes_{\bar{K}^+/p\bar{K}^+}A_{\cris}(K^+)/pA_{\cris}(K^+) \]
which sends $\underline{T_i}$ to $T_i^{(1)}\otimes 1$. It follows from Lemma \ref{subflat} that under this isomorphism we have
\[ A_{\cris,\infty}^{(i)}/pA_{\cris,\infty}^{(i)}\cong (\tilde{O}(c)_{\infty}^{(i)}/p\tilde{O}(c)_{\infty}^{(i)})\otimes_{\bar{K}^+/p\bar{K}^+}A_{\cris}(K^+)/pA_{\cris}(K^+) \]
hence in this case the claim follows from the last lemma.

Now we show the result by induction on $n$. Assume it true for $n-1$. From Lemma \ref{write} we see that we can write any $a\in A_{\cris,\infty}/p^nA_{\cris,\infty}$ in the form
\[ a=\sum_{\beta,\gamma}x_{\beta,\gamma}[\underline{T_i}]^{\beta}[\underline{T_1}]^{\gamma}=\sum_{\alpha\geq 0}\sum_{\beta-\gamma=\alpha}x_{\beta,\gamma}[\underline{T_iT_1}]^{\gamma}[\underline{T_i}]^{\alpha}+\sum_{\alpha<0}\sum_{\beta-\gamma=\alpha}x_{\beta,\gamma}[\underline{T_iT_1}]^{\beta}[\underline{T_1}]^{-\alpha} \]
with $x_{\beta,\gamma}$ in the sub-$A_{\cris}(K^+)$-algebra of $A_{\cris,\infty}/p^nA_{\cris,\infty}$ generated by $[\underline{T_j}]^{\beta_j}$ with $j\notin\left\lbrace 1,i\right\rbrace$, $2\leq j\leq d+1$, $\beta_j\in\mathbb{N}[1/p]$. Define $X_{\alpha}:=x_{\beta,\gamma}[\underline{T_iT_1}]^{\gamma}$ if $\alpha\geq 0$ and $X_{\alpha}:=x_{\beta,\gamma}[\underline{T_iT_1}]^{\beta}$ otherwise. It remains to show that if $a=0$, then $X_{\alpha}=0$ for all $\alpha$. We have an exact sequence
\[ 0\to p^{n-1}A_{\cris,\infty}/p^nA_{\cris,\infty}\to A_{\cris,\infty}/p^nA_{\cris,\infty}\to A_{\cris,\infty}/p^{n-1}A_{\cris,\infty}\to 0 \]
and if $a=0$, then $X_{\alpha}\equiv 0\mod p^{n-1}$ by induction hypothesis. So $X_{\alpha}=p^{n-1}Y_{\alpha}$ for all $\alpha$. Since by flatness over $W_n(k)$ we have $p^{n-1}A_{\cris,\infty}/p^nA_{\cris,\infty}\cong A_{\cris,\infty}/pA_{\cris,\infty}$  (Prop. \ref{flat2}) we see that $Y_{\alpha}=0$ for all $\alpha$ and this completes the proof for $2\leq i\leq r$. For $i>r$ the proof is similar. \end{proof}

\begin{cor}\label{invariants1}
Fix some $i$, $2\leq i\leq d+1$, and let $\mathbb{Z}_p(1)$ be the factor of $\Delta_\infty$ corresponding to $\sigma_i$. Then
\[ t\cdot (A_{\cris,\infty}(O(c))/p^nA_{\cris,\infty}(O(c)))^{\sigma_i=1}\subset \bigoplus_{\alpha\in\mathbb{Z},v_p(\alpha)>0}p^{\max\left\{0,n-v_p(\alpha)\right\}}F_{\alpha}. \]
\end{cor}
\begin{proof}
If $f\in F_{\alpha}$ is invariant under $\sigma_i$, then $0=([\underline{1}]^{\alpha}-1)f$ and by Proposition \ref{constants} there exists $b_{\alpha}\in A_{\cris}(K^+)$ such that $b_{\alpha}([\underline{1}]^{\alpha}-1)f=p^{v_{\alpha}}tf$, where $v_{\alpha}=\max\left\lbrace 0,v_p(\alpha)\right\rbrace $. So if $v_p(\alpha)\leq 0$, then $tf=0$. If $v_p(\alpha)\geq 0$, then $\alpha\in\mathbb{Z}$ and so $([\underline{1}]^{\alpha}-1)f=\alpha u_{\alpha}tf$ for some unit $u_{\alpha}$ (by \emph{loc. cit.}). So $tf$ is killed by $p^{v_p(\alpha)}$ and hence the result follows from the flatness over $W_n(k)$ (Prop. \ref{flat2}).
\end{proof}

\begin{cor}\label{invariants1'}
The canonical map
\[ t^d\cdot\left( A_{\cris,\infty}(O(c))/p^nA_{\cris,\infty}(O(c))\right) ^{\Delta_\infty}\to \left( M^+_\infty(O(c))/p^nM^+_\infty(O(c))\right) ^{\Delta_\infty}  \]
has image in $B^+_{\log}/p^nB^+_{\log}\otimes_{\Sigma_n}\Theta(c)/p^n\Theta(c)$.
\end{cor}
\begin{proof}
We have $X_i=[\underline{T_i}]\otimes T_i^{-1}-1$. If $\alpha\in\mathbb{N}[1/p]$ with $v_p(\alpha)\geq 0$, then $\alpha\in\mathbb{N}$ and setting $m_{\alpha}:=\max\left\{0,n-v_p(\alpha)\right\}$ we have
\begin{eqnarray*}
 p^{m_{\alpha}}[\underline{T_i}]^\alpha &=& p^{m_{\alpha}}T_i^\alpha ([\underline{T_i}]\otimes T_i^{-1}-1+1)^\alpha\\
&=& p^{m_{\alpha}}T_i^\alpha+T_i^\alpha\sum_{r>0}p^{m_{\alpha}}\alpha (\alpha -1)\cdots (\alpha -r+1)X_i^{[r]}\\
&=& p^{m_{\alpha}}T_i^\alpha
\end{eqnarray*}
for all $1\leq i\leq d+1$, so $p^{m_{\alpha}}[\underline{T_i}]^\alpha\in\Theta(c)/p^n\Theta(c)$.

Now, let $a\in \left( A_{\cris,\infty}(O(c))/p^nA_{\cris,\infty}(O(c))\right) ^{\Delta_\infty}$. Assume that $2\leq i\leq r$. Write $a=\sum_{\alpha\geq 0}X_{\alpha}[\underline{T_i}]^{\alpha}+\sum_{\alpha<0}X_{\alpha}[\underline{T_1}]^{-\alpha}$. Then by the last corollary, we have
\[ ta=\sum_{\alpha\geq 0,v_p(\alpha)>0}X'_{\alpha}p^{m_{\alpha}}[\underline{T_i}]^{\alpha}+\sum_{\alpha<0,v_p(\alpha)>0}X'_{\alpha}p^{m_{\alpha}}[\underline{T_1}]^{-\alpha} \]
where $X'_{\alpha}p^{m_{\alpha}}=tX_{\alpha}$. That is,
\[ ta=\sum_{\alpha\in\mathbb{N},v_p(\alpha)>0}X'_{\alpha}p^{m_{\alpha}}T_i^{\alpha}+\sum_{-\alpha\in\mathbb{N},v_p(\alpha)>0}X'_{\alpha}p^{m_{\alpha}}T_1^{-\alpha} \]
Now we pick some $j\neq i$, $2\leq j\leq r$, and apply a similar reasoning to the $X'_{\alpha}$ to deduce that
\[ tX'_{\alpha}=\sum_{\beta\in\mathbb{N},v_p(\beta)>0}Y_{\beta}p^{m_{\beta}}T_j^{\beta}+\sum_{-\beta\in\mathbb{N},v_p(\beta)>0}Y_{\beta}p^{m_{\beta}}T_1^{-\beta} \]
where the $Y_{\beta}$ lie in the sub-$A_{\cris}(K^+)$-algebra of $A_{\cris,\infty}(O(c))/p^nA_{\cris,\infty}(O(c))$ generated by $[\underline{T_jT_iT_1}]^{\beta}$, $\beta\in\mathbb{N}[1/p]$, and $[\underline{T_k}]^{\beta_k}$, $k\notin\left\lbrace 1,i,j\right\rbrace$, $2\leq k\leq d+1$, $\beta_k\in\mathbb{N}[1/p]$. The case for $r+1\leq j\leq d+1$ is similar. Continuing in this way we clearly obtain the result.
\end{proof}

\begin{thm}\label{invariants2}
We have
\[ t^d[\underline{c}]\cdot\left( M^+_\infty(R)/p^nM^+_\infty(R)\right)^{\Delta_\infty}\subset B^+_{\log}/p^nB^+_{\log}\otimes_{\Sigma_n}\mathcal{R}_n. \]
\end{thm}
\begin{proof}
We will give the proof in several steps.

\underline{Step 0:} Recall that by quasi-coherence of $h_{\infty,\ast}\mathscr{O}$, we have a canonical isomorphism
\[ M^+_\infty(R)/p^nM^+_\infty(R)\cong M^+_\infty(O(c))/p^nM^+_\infty(O(c))\otimes_{\Theta(c)/p^n\Theta(c)}\mathcal{R}_n. \]
Hence we may assume that $R=O(c)$. In this case, by Proposition \ref{dppolynomialalgebra} we have
\[ M^+_\infty/p^nM^+_\infty\cong A^+_\infty/p^nA^+_\infty\left\langle X_2,...,X_{d+1}\right\rangle.   \]
For the proof it will be better for us to replace $X$ resp. $X_i$ by $Y:=\frac{-X}{1+X}$ resp. $Y_i:=\frac{-X_i}{1+X_i}$, $2\leq i\leq d+1$. Then we have a canonical isomorphism
\[ M^+_\infty/p^nM^+_\infty\cong A^+_\infty/p^nA^+_\infty\left\langle Y_2,...,Y_{d+1}\right\rangle \]
and morally $Y=[\underline{\pi}]^{-1}\otimes u-1$ if $c=\pi$ and $Y_i=[\underline{T_i}]^{-1}\otimes T_i-1$. Let $D_i$ denote the subalgebra of $M^+_\infty/p^nM^+_\infty$ consisting of elements of the form $\sum_Na_N\prod_{j,j\neq i}X_j^{[n_j]}$ where $a_N\in A^+_\infty/p^nA^+_\infty$. That is,
\[ D_i=A^+_\infty/p^nA^+_\infty\left\langle Y_2,...,\hat{Y_i},...,Y_{d+1}\right\rangle \]
where $\hat{Y_i}$ means we omit $Y_i$. We also define
\[ D_i^{(i)}:=A_{\cris,\infty}^{(i)}/p^nA_{\cris,\infty}^{(i)}\left\langle Y, Y_2,...,\hat{Y_i},...,Y_{d+1}\right\rangle \]

\underline{Step 1:} Suppose that $m\in M^+_\infty/p^nM^+_\infty$ is $t$-torsion. Write $m=\sum_nm_nY_i^{[n]}$. Then $m_n$ is $t$-torsion for all $n$. Write $m_n=\sum_{\alpha\geq 0}X_{\alpha,n}[\underline{T_i}]^{\alpha}+\sum_{\alpha<0}X_{\alpha,n}[\underline{T_1}]^{-\alpha}$ with $X_{\alpha,n}\in D_i^{(i)}$. Then $tX_{\alpha,n}=0$ for all $\alpha$ (\emph{cf.} Prop. \ref{direct}). We have
\begin{eqnarray*}
(\sigma_i-1)([\underline{T_i}]^{\alpha}Y_i^{[n]}) &=& [\underline{1}]^{\alpha}[\underline{T_i}]^{\alpha}([\underline{1}]^{-1}[\underline{T_i}]^{-1}\otimes T_i-1)^{[n]}-[\underline{T_i}]^{\alpha}Y_i^{[n]} \\
&=& [\underline{1}]^{\alpha-n}[\underline{T_i}]^{\alpha}\left( [\underline{T_i}]^{-1}\otimes T_i-1+(1-[\underline{1}])\right) ^{[n]}-[\underline{T_i}]^{\alpha}Y_i^{[n]} \\
&=& [\underline{1}]^{\alpha -n}[\underline{T_i}]^{\alpha}\sum_{r=0}^nY_i^{[r]}(1-[\underline{1}])^{[n-r]}-[\underline{T_i}]^{\alpha}Y_i^{[n]} \\
&=& ([\underline{1}]^{\alpha -n}-1)[\underline{T_i}]^{\alpha}Y_i^{[n]}+[\underline{1}]^{\alpha -n}[\underline{T_i}]^{\alpha}\sum_{r=0}^{n-1}Y_i^{[r]}(1-[\underline{1}])^{[n-r]}.
\end{eqnarray*}
Similarly
\[ (\sigma_i-1)([\underline{T_1}]^{-\alpha}Y_i^{[n]})= ([\underline{1}]^{-\alpha -n}-1)[\underline{T_1}]^{-\alpha}Y_i^{[n]}+[\underline{1}]^{-\alpha -n}[\underline{T_1}]^{-\alpha}\sum_{r=0}^{n-1}Y_i^{[r]}(1-[\underline{1}])^{[n-r]}. \]
Now recall (Prop. \ref{constants}) that $1-[\underline{1}]=tw$ for some unit $w\in A_{\cris}(K^+)$, hence
\[ (1-[\underline{1}])^{[n-r]}=t\dfrac{w^{n-r}t^{n-r-1}}{(n-r)!} \]
is a multiple of $t$ for $0\leq r<n$. Hence
\[ (\sigma_i-1)(m)=\sum_nY_i^{[n]}\left( \sum_{\alpha\geq 0}X_{\alpha,n}([\underline{1}]^{\alpha -n}-1)[\underline{T_i}]^{\alpha}+\sum_{\alpha<0}X_{\alpha,n}([\underline{1}]^{-\alpha -n}-1)[\underline{T_1}]^{-\alpha}\right) . \]
We'll need this later.

\underline{Step 2:} For $2\leq i\leq r$, consider the derivation $\partial_i$ of $M^+_\infty/p^nM^+_\infty$ defined
\[ \partial_i:=[\underline{T_1}]\dfrac{\partial}{\partial Y_i}. \]
We have $\partial_i=[\underline{T_iT_1}]\dfrac{\partial}{\partial T_i}$, hence $\partial_i$ commutes with $\sigma_i$. Moreover, we claim that the kernel of $\partial_i$ lies in $D_i$ up to $[\underline{T_1}]$-torsion. Indeed, suppose that $\partial_i(m)=0$. Write $m=\sum_{k}a_kY_i^{[k]}$ with $a_k\in D_i$. Then $[\underline{T_1}]a_k=0$ for $k\neq 0$, as claimed.

On $[\underline{T_1}]\cdot M^+_\infty/p^nM^+_\infty$ we define a one-sided inverse $\int_i$ to $\partial_i$ as being the unique $D_i$-linear map such that
\[ \int_i [\underline{T_1}]Y_i^{[k]}=Y_i^{[k+1]} \]
for all $k\in\mathbb{N}$.

We claim that if $c=\pi$ (resp. $c=1$)
\begin{eqnarray*}
(\sigma_i-1)\int_i[\underline{T_1T_i}]T_i^n &=& \dfrac{1-[\underline{1}]^{n+1}}{n+1}[\underline{T_i}]^{n+1} \\
(\sigma_i-1)\int_i T_1^n &=& \dfrac{[\underline{1}]^{1-n}-1}{n-1}\dfrac{[\underline{T_1}]^{n-1}(1+Y)}{\prod_{2\leq j\leq r,j\neq i}(1+Y_j)}\qquad (n>0) \\
\text{resp.}\quad (\sigma_i-1)\int_i T_1^n &=& \dfrac{[\underline{1}]^{1-n}-1}{n-1}\dfrac{[\underline{T_1}]^{n-1}[\underline{c}]}{\prod_{2\leq j\leq r,j\neq i}(1+Y_j)}\qquad (n>0)
\end{eqnarray*}
where we recall that for any integer $n\neq 0$ we have $\dfrac{[\underline{1}]^n-1}{n}\in A_{\cris}(K^+)$ (Prop. \ref{constants}).

Let us show this. Since $M^+_{\infty}$ is flat over $W(k)$ (\emph{cf.} Prop. \ref{dppolynomialalgebra}) it suffices to show this in $M^+_{\infty}[1/p]$. By the binomial theorem we have
\[ T_i^{n}=[\underline{T_i}]^{n}(1+Y_i)^{n}=[\underline{T_i}]^{n}\sum_{r=0}^{n}\dfrac{n!}{r!}Y_i^{[n-r]} \]
hence
\[ \int_i[\underline{T_1T_i}]T_i^n=[\underline{T_i}]^{n+1}\sum_{r=0}^{n}\dfrac{n!}{r!}Y_i^{[n+1-r]} \]
so we have
\[ \int_i[\underline{T_1T_i}]T_i^n=\dfrac{T_i^{n+1}}{n+1}-\dfrac{[\underline{T_i}]^{n+1}}{n+1}. \]
It follows that
\[ (\sigma_i-1)\int_i[\underline{T_1T_i}]T_i^n=\dfrac{1-[\underline{1}]^{n+1}}{n+1}[\underline{T_i}]^{n+1}. \]
Now for the case of $T_1^n$ ($n>0$). It is not hard to see that for $c=\pi$
\[ T_1=\dfrac{[\underline{T_1}](1+Y)}{\prod_{2\leq j\leq r}(1+Y_j)} \]
hence
\[ T_1^n=\dfrac{[\underline{T_1}]^{n-1}(1+Y)^n}{\prod_{2\leq j\leq r,j\neq i}(1+Y_j)^n}\dfrac{[\underline{T_1}]}{(1+Y_i)^n}. \]
Also from the binomial theorem we have
\[ \dfrac{1}{(1+Y_i)^n}=\sum_{m=0}^{\infty}\dfrac{(n+m-1)!}{(n-1)!}(-Y_i)^{[m]} \]
hence
\[ \int_i\dfrac{[\underline{T_1}]}{(1+Y_i)^n}=-\sum_{m=0}^{\infty}\dfrac{(n+m-1)!}{(n-1)!}(-Y_i)^{[m+1]} \]
which is the same as
\[ \int_i\dfrac{[\underline{T_1}]}{(1+Y_i)^n}=\dfrac{1}{n-1}\left( 1-\dfrac{1}{(1+Y_i)^{n-1}}\right). \]
So we have
\[ \int_i T_1^n=\dfrac{[\underline{T_1}]^{n-1}(1+Y)^n}{(n-1)\prod_{2\leq j\leq r,j\neq i}(1+Y_j)^n}-\dfrac{T_1^{n-1}(1+Y)}{(n-1)\prod_{2\leq j\leq r,j\neq i}(1+Y_j)} \]
and hence
\[ (\sigma_i-1)\int_i T_1^n=\dfrac{[\underline{1}]^{1-n}-1}{n-1}\dfrac{[\underline{T_1}]^{n-1}(1+Y)}{\prod_{2\leq j\leq r,j\neq i}(1+Y_j)}. \]
In the case $c=1$ we have
\[ T_1=\dfrac{[\underline{T_1}][\underline{c}]}{\prod_{2\leq j\leq r}(1+Y_j)} \]
and a similar argument gives the desired formula.

\underline{Step 3:} Now consider an element $m\in M^+_\infty/p^nM^+_\infty$. Assume that $\sigma_i(m)=m$. Write $m=\sum_{k=0}^{k_0}a_kY_i^{[k]}$ with $a_k\in D_i$. Let us show by induction on $k_0$ that $t[\underline{T_iT_1}]m\in t[\underline{T_iT_1}]\cdot D_i^{\sigma_i=1}[T_i,T_1]$. Here by $D_i^{\sigma_i=1}[T_i,T_1]$ we mean the $D_i^{\sigma_i=1}$-subalgebra of $M^+_{\infty}/p^nM^+_{\infty}$ generated by $T_i$ and $T_1$ (not the polynomial algebra!). For $k_0=0$ this is trivial. Since $\partial_i(m)=\sum_{k=1}^{k_0}[\underline{T_1}]a_kY_i^{[k-1]}$ we can assume the claim is true for $t[\underline{T_iT_1}]\partial_i(m)$. That is, we can write $t[\underline{T_iT_1}]\partial_i(m)$ as a finite sum
\[ t[\underline{T_iT_1}]\partial_i(m)=t[\underline{T_iT_1}]\sum_{n}b_nT_i^n+t[\underline{T_iT_1}]\sum_{n>0}c_nT_1^n \]
where $b_n,c_n\in D_i^{\sigma_i=1}$, $n\in\mathbb{N}$. So we may write
\[ [\underline{T_iT_1}]\partial_i(m)=[\underline{T_iT_1}]\sum_{n}b_nT_i^n+[\underline{T_iT_1}]\sum_{n>0}c_nT_1^n+\tau \]
where $\tau$ is $t$-torsion. Also, if we define
\[ a:=m-\int_i\partial_i(m) \]
then we have
\[ \partial_i(a)=0 \]
hence $[\underline{T_1}]a\in [\underline{T_1}]\cdot D_i$.

We now assume that $c=\pi$. The following argument applies to the case $c=1$ with the formulas obtained in Step 2.

Substituting the expressions for $(\sigma_i-1)\int_i[\underline{T_iT_1}]T_i^n$ and $(\sigma_i-1)\int_iT_1 ^n$ obtained in Step 2, we find
\begin{equation}\label{eq1}
(1-\sigma_i)([\underline{T_iT_1}]a)=\sum_nb_n\dfrac{1-[\underline{1}]^{n+1}}{n+1}[\underline{T_i}]^{n+1}+ \sum_{n>0}[\underline{T_iT_1}]c_n\dfrac{[\underline{1}]^{1-n}-1}{n-1}\dfrac{[\underline{T_1}]^{n-1}(1+Y)^n}{\prod_{2\leq j\leq r,j\neq i}(1+Y_j)^n}+(\sigma_i-1)\int_i\tau.
\end{equation}
Now, since $[\underline{T_iT_1}]a\in [\underline{T_iT_1}]\cdot D_i$ we may write
\[ [\underline{T_iT_1}]a=\sum_{\alpha\in\mathbb{N}[1/p]}[\underline{T_iT_1}]X_{\alpha}[\underline{T_i}]^{\alpha}+\sum_{\alpha\in\mathbb{Z}[1/p],\alpha<0}[\underline{T_iT_1}]X_{\alpha}[\underline{T_1}]^{-\alpha} \]
with $X_{\alpha}$ in $D_i^{(i)}$. Hence
\begin{equation}\label{eq2}
(1-\sigma_i)([\underline{T_iT_1}]a)=\sum_{\alpha\in\mathbb{N}[1/p]}[\underline{T_iT_1}]X_{\alpha}(1-[\underline{1}]^{\alpha})[\underline{T_i}]^{\alpha}+\sum_{\alpha\in\mathbb{Z}[1/p],\alpha<0}[\underline{T_iT_1}]X_{\alpha}(1-[\underline{1}]^{\alpha})[\underline{T_1}]^{-\alpha}.
\end{equation}
Comparing the equations (\ref{eq1}) and (\ref{eq2}) we claim that we get an equality of coefficients of $[\underline{T_i}]^{\alpha}$ and $[\underline{T_1}]^{-\alpha}$. Indeed, since the $Y,Y_j$, $2\leq j\leq d+1$, $j\neq i$, are independent indeterminates we reduce to the case where the coefficients lie in $A_{\cris,\infty}(O(c))/p^nA_{\cris,\infty}(O(c))$ in which case it follows from Proposition \ref{direct}. Now write
\[ \int_i\tau=\sum_{n>0}Y_i^{[n]}\left( \sum_{\alpha\geq 0}X_{\alpha,n}[\underline{T_i}]^{\alpha}+\sum_{\alpha<0}X_{\alpha,n}[\underline{T_1}]^{-\alpha}\right)  \]
with $X_{\alpha,n}\in D_i^{(i)}$ as in Step 1. Since $\int_i\tau $ is $t$-torsion, by Step 1 we have
\[ (\sigma_i-1)(\int_i\tau)=\sum_{n>0}Y_i^{[n]}\left( \sum_{\alpha\geq 0}X_{\alpha,n}([\underline{1}]^{\alpha -n}-1)[\underline{T_i}]^{\alpha}+\sum_{\alpha<0}X_{\alpha,n}([\underline{1}]^{-\alpha -n}-1)[\underline{T_1}]^{-\alpha}\right) . \]
Note that in this sum we have $n>0$. So comparing with equation (\ref{eq2}) we deduce that
\[ (\sigma_i-1)(\int_i\tau)=0. \]
Hence comparing equations (\ref{eq1}) and (\ref{eq2}) we find
\begin{eqnarray}\label{eq3}
0 &=& t[\underline{T_iT_1}]X_{\alpha} \qquad (v_p(\alpha)<0)\\
b_n\dfrac{1-[\underline{1}]^{n+1}}{n+1} &=& [\underline{T_iT_1}]X_{n+1}(1-[\underline{1}]^{n+1}) \\
\left[ \underline{T_iT_1}\right]  c_n \dfrac{([\underline{1}]^{1-n}-1)(1+Y)^n}{(n-1)\prod_{2\leq j\leq r,j\neq i}(1+Y_j)^n} &=& [\underline{T_iT_1}]X_{1-n}(1-[\underline{1}]^{1-n}).
\end{eqnarray}
In particular, from Prop. \ref{constants} we deduce that
\begin{eqnarray*}
tb_n &=& [\underline{T_iT_1}]X_{n+1}t(n+1) \\
tc_n[\underline{T_iT_1}] &=& (n-1)t[\underline{T_iT_1}]X_{1-n}\dfrac{\prod_{2\leq j\leq r,j\neq i}(1+Y_j)^n}{(1+Y)^n} .
\end{eqnarray*}
Hence we can write
\begin{eqnarray*}
t[\underline{T_iT_1}]m &=& t[\underline{T_iT_1}]a+t[\underline{T_iT_1}]\sum_nX_{n+1}(T_i^{n+1}-[\underline{T_i}]^{n+1})\\
&& \quad +t[\underline{T_iT_1}]\sum_{n>0}X_{1-n}\left( [\underline{T_1}]^{n-1}- \dfrac{T_1^{n-1}\prod_{2\leq j\leq r,j\neq i}(1+Y_j)^{n-1}}{(1+Y)^{n-1}}\right) \\
&=& t[\underline{T_iT_1}]\left\lbrace a'+\sum_nX_{n+1}T_i^{n+1}-\sum_{n>0}X_{1-n}\dfrac{T_1^{n-1}\prod_{2\leq j\leq r,j\neq i}(1+Y_j)^{n-1}}{(1+Y)^{n-1}}\right\rbrace 
\end{eqnarray*}
where
\[ a':=a-\sum_nX_{n+1}[\underline{T_i}]^{n+1}+\sum_{n>0}X_{1-n}[\underline{T_1}]^{n-1} \]
is annihilated by $t[\underline{T_iT_1}]$ by equation (\ref{eq3}). So we can write
\[ t[\underline{T_iT_1}]m=t[\underline{T_iT_1}]\left\lbrace \sum_nX_{n+1}T_i^{n+1}-\sum_{n>0}X_{1-n}\dfrac{T_1^{n-1}\prod_{2\leq j\leq r,j\neq i}(1+Y_j)^{n-1}}{(1+Y)^{n-1}}\right\rbrace  \]
and this completes the induction step.

\underline{Step 4:} Now we can repeat Step 3 for any other index $j\neq i$, $2\leq j\leq r$. In more detail, the same argument as that of Step 3 shows that if $m\in (M^+_{\infty}/p^nM^+_{\infty})^{\sigma_i=1,\sigma_j=1}$, then $t^2[\underline{T_jT_iT_1}]m\in t^2[\underline{T_jT_iT_1}]\cdot D_{i,j}^{\sigma_i=1,\sigma_j=1}[T_j,T_i,T_1]$, where
\[ D_{i,j}:=A^+_{\infty}/p^nA^+_{\infty}\left\langle Y_2,...,\hat{Y_i},...,\hat{Y_j},...,Y_{d+1}\right\rangle. \]
It is now clear that if $m\in (M^+_{\infty}/p^nM^+_{\infty})^{\Delta_{\infty}}$, then $t^r[\underline{c}]m\in t^r[\underline{c}]\cdot D_{2,...,r}^{\sigma_2=1,...,\sigma_r=1}[T_1,T_2,....,T_r]$, where
\[ D_{2,...,r}:=A^+_{\infty}/p^nA^+_{\infty}\left\langle Y_{r+1},...,Y_{d+1}\right\rangle. \]

\underline{Step 5:} For an index $i$, $r+1\leq i\leq d+1$, then we may adapt the previous steps with the following changes. We use $\partial_i:=\dfrac{\partial}{\partial T_i}$ and let $\int_i$ be the unique $D_i$-linear map sending $Y_i^{[n]}$ to $[\underline{T_i}]Y_i^{[n+1]}$. Then one checks easily that we have
\[ (\sigma_i-1)\int_iT_i^n=\dfrac{1-[\underline{1}]^{n+1}}{n+1}[\underline{T_i}]^{n+1}. \]
With these changes a similar but simpler proof as that of Step 3 shows that if $m\in\ker(\sigma_i-1)$, then $tm\in t\cdot D_i^{\sigma_i=1}[T_i]$.

\underline{Step 6:} Combining the previous steps, we deduce that if $m\in (M^+_{\infty}/p^nM^+_{\infty})^{\Delta_{\infty}}$, then $t^d[\underline{c}]m\in t^d[\underline{c}]\cdot (A^+_{\infty}/p^nA^+_{\infty})^{\Delta_{\infty}}[T_1,....,T_{d+1}]$. Hence by Corollary \ref{invariants1'} the proof is complete. 
\end{proof}

\subsubsection{} Recall that $R$ is a small integral $K^+$-algebra with $K^+$ integrally closed in $R$, and $\mathcal{R}_n$ is the \'etale $\Theta(c)/p^n\Theta(c)$-algebra lifting $R/pR$.

\begin{thm}\label{vanishing}
For all $i\neq 0$, the $B^+$-module
\[ H^i(\Delta_\infty,M^+_\infty/p^nM^+_\infty) \]
is annihilated by $t^d$.
\end{thm}
\begin{proof}

\underline{Step 0:} Firstly, by Corollary \ref{invariants0}, $M^+_\infty/p^nM^+_\infty$ is a discrete $p$-torsion $\Delta_\infty$-module, so the Koszul complex $L\otimes_{\mathbb{Z}_p[[\Delta_\infty]]}M^+_\infty/p^nM^+_\infty$ computes the Galois cohomology of $M^+_\infty/p^nM^+_\infty$ (\ref{koszul}). We first show that for all $i$ and all $m\in M^+_\infty/p^nM^+_\infty$, $t\cdot m$ lies in the image of the endomorphism $\sigma_i-1$, and from this we will deduce the statement of the theorem. Since
\[ M^+_\infty(R)/p^nM^+_\infty(R)\cong M^+_\infty(O(c))/p^nM^+_\infty(O(c))\otimes_{\Theta(c)/p^n\Theta(c)}\mathcal{R}_n \]
we can assume that $R=O(c)$. Fix some $2\leq i\leq d+1$. Recall that $X_i=[\underline{T_i}]\otimes T_i^{-1}-1$.
Now every element of $M^+_\infty(O(c))/p^nM^+_\infty(O(c))$ is the sum of monomials of the form
\[ \mu_1=x[\underline{T_1}]^{\alpha}X_i^{[n]},\quad\text{or}\quad \mu=x[\underline{T}_i]^{\alpha}X_i^{[n]} \]
with $x$ invariant under $\sigma_i$ and $\alpha\in\mathbb{N}[1/p]$. For $m=\mu_1,\mu$, we will show that $tm=(\sigma_i-1)f_i(m)$ for some $f_i(m)$ by induction on $n$.

\underline{Step 1:} If $n=0$, then we distinguish three cases: $\alpha=0$, $v_p(\alpha)\geq 0$, and $v_p(\alpha)<0$. If $\alpha=0$ then take
\[ f_i(m)=m\log(X_i+1). \]
Since $\sigma_i\log(X_i+1)=\log([\underline{1}](X_i+1))=t+\log(X_i+1)$ we obtain the claim in this case. If $v_p(\alpha)\geq 0$ then $\alpha\in\mathbb{N}$ and
\begin{eqnarray*}
[\underline{T}_i]^{\alpha} &=& T_i^{\alpha}(1+X_i)^{\alpha}\\
&=& T_i^{\alpha}\left( 1+\sum_{r=1}^\alpha \dfrac{\alpha !}{(\alpha -r)!}X_i^{[r]}\right) 
\end{eqnarray*}
so
\begin{eqnarray*}
([\underline{1}]^{\alpha}-1)[\underline{T}_i]^{\alpha} &=& (\sigma_i-1)[\underline{T}_i]^{\alpha} \\
&=&  T_i^{\alpha}(\sigma_i-1)\left( \sum_{r=1}^\alpha \dfrac{\alpha !}{(\alpha -r)!}X_i^{[r]}\right)  \\
&=& (\sigma_i-1)\left( \alpha\cdot T_i^{\alpha}\left( \sum_{r=1}^\alpha \dfrac{(\alpha -1) !}{(\alpha -r)!}X_i^{[r]}\right) \right) 
\end{eqnarray*}
and hence
\[ t[\underline{T}_i]^{\alpha}=(\sigma_i-1)\left( \dfrac{t\alpha}{[\underline{1}]^\alpha -1}T_i^{\alpha}\left( \sum_{r=1}^\alpha \dfrac{(\alpha -1)!}{(\alpha -r)!}([\underline{\pi}\underline{T}_i^{-1}]\otimes u^{-1}T_i-1)^{[r]}\right)\right)   \]
(recall that $\dfrac{t\alpha}{[\underline{1}]^\alpha -1}\in A_{\cris}(K^+)$ by Proposition \ref{constants}). Similarly we have
\[ t[\underline{T}_1]^{\alpha}=(\sigma_i-1)\left( \dfrac{t\alpha}{[\underline{1}]^\alpha -1}T_1^{\alpha}\left( \sum_{r=1}^\alpha \dfrac{(\alpha -1)!}{(\alpha -r)!}X_1^{[r]}\right)\right) \]
where $X_1=[\underline{T_1}]\otimes T_1^{-1}-1$. If $v_p(\alpha)<0$, then
\[ t[\underline{T_i}]^{\alpha}=(\sigma_i-1)\left( \dfrac{t}{[\underline{1}]^\alpha -1}[\underline{T_i}]^{\alpha}\right). \]
This begins the induction.

\underline{Step 2:} Fix some $2\leq i\leq d+1$. Let $\partial_i:=T_i\frac{\partial}{\partial T_i}$. We claim that if $m\in M^+_\infty$ and $\partial_i(m)=0$, then $tm=(\sigma_i-1)m'$ for some $m'\in M^+_\infty$. Indeed, write $m=\sum_jy_jX_i^{[j]}$ with $\partial_i(y_j)=0$ and $y_j=\sum_{\alpha\geq 0}x_{\alpha}[\underline{T_i}]^{\alpha}+\sum_{\alpha<0}x_{\alpha}[\underline{T_1}]^{-\alpha}$ with $x_{\alpha}$ invariant under $\sigma_i$. Then
\[ \partial_i(m)=\left( \sum_jy_jX_i^{[j-1]}\right)(-1)(1+X_i)=0 \]
and since $1+X_i$ is a unit we deduce that $y_j=0$ for all $j\neq 0$, hence $m=y_0$. So by Step 1 above this proves the claim.

\underline{Step 3:} We claim that $\partial_i$ is surjective. It suffices to show that $X_i^{[j]}$ lies in the image of $\partial_i$ for all $j\in\mathbb{N}$. We have
\[ \partial_i(-X_i^{[j+1]})=X_i^{[j]}+(j+1)X_i^{[j+1]} \]
so by induction we deduce the formula
\[ X_i^{[j]}=\sum_{k=1}^\infty (-1)^{k}\dfrac{(j+k-1)!}{j!}\partial_i(X_i^{[j+k]}) \]
thereby proving the claim.

\underline{Step 4:} We claim that for all $\alpha\in\mathbb{N}[1/p]$ and all $n\in\mathbb{N}$, the elements
\[ nt[\underline{T_1}]^{\alpha}X_i^{[n]},\quad nt[\underline{T_i}]^{\alpha}X_i^{[n]} \]
are in the image of $\sigma_i-1$. Let us show this for $nt[\underline{T_i}]^{\alpha}X_i^{[n]}$, the argument for $nt[\underline{T_1}]^{\alpha}X_i^{[n]}$ being the same. First of all, if $v_p(\alpha)\geq 0$ then
\[ [\underline{T_i}]^{\alpha}=T_i^{\alpha}([\underline{T_i}]\otimes T_i^{-1}-1+1)^{\alpha}=T_i^{\alpha}\sum_{r\geq 0}\dfrac{\alpha !}{(\alpha -r)!}X_i^{[r]} \]
so by considering each summand individually we may assume that either $\alpha=0$ or $v_p(\alpha)<0$. Now, we have
\begin{eqnarray*}
(\sigma_i-1)([\underline{T_i}]^{\alpha}X_i^{[n]}) &=& [\underline{1}]^{\alpha}[\underline{T_i}]^{\alpha}([\underline{1}][\underline{T_i}]\otimes T_i^{-1}-1)^{[n]}-[\underline{T_i}]^{\alpha}X_i^{[n]} \\
&=& [\underline{1}]^{\alpha +n}[\underline{T_i}]^{\alpha}\left( [\underline{T_i}]\otimes T_i^{-1}-1+(1-[\underline{1}]^{-1})\right) ^{[n]}-[\underline{T_i}]^{\alpha}X_i^{[n]} \\
&=& [\underline{1}]^{\alpha +n}[\underline{T_i}]^{\alpha}\sum_{r=0}^nX_i^{[r]}(1-[\underline{1}]^{-1})^{[n-r]}-[\underline{T_i}]^{\alpha}X_i^{[n]} \\
&=& ([\underline{1}]^{\alpha +n}-1)[\underline{T_i}]^{\alpha}X_i^{[n]}+[\underline{1}]^{\alpha +n}[\underline{T_i}]^{\alpha}\sum_{r=0}^{n-1}X_i^{[r]}(1-[\underline{1}]^{-1})^{[n-r]}.
\end{eqnarray*}
Now recall (Prop. \ref{constants}) that $1-[\underline{1}]^{-1}=tw$ for some unit $w\in A_{\cris}(K^+)$, hence
\[ (1-[\underline{1}]^{-1})^{[n-r]}=t\dfrac{w^{n-r}t^{n-r-1}}{(n-r)!} \]
is a multiple of $t$ for $0\leq r<n$. So by induction on $n$ we may assume that
\[ [\underline{1}]^{\alpha +n}[\underline{T_i}]^{\alpha}\sum_{r=0}^{n-1}X_i^{[r]}(1-[\underline{1}]^{-1})^{[n-r]} \]
lies in the image of $\sigma_i-1$. Hence
\[ ([\underline{1}]^{\alpha +n}-1)[\underline{T_i}]^{\alpha}X_i^{[n]} \]
lies in the image of $\sigma_i-1$. Recall (Prop. \ref{constants}) that we have
\[ tp^{\max(v_p(\alpha+n),0)}=b([\underline{1}]^{\alpha +n}-1) \]
for some $b\in A_{\cris}(K^+)$. If $\alpha=0$, then this immediately proves the claim. If $v_p(\alpha)<0$, then $v_p(\alpha +n)=v_p(\alpha)<0$, and so the claim also follows.

\underline{Step 5:} Consider an element of the form
\[ t[\underline{T_1}]^{\alpha}X_i^{[n]},\quad t[\underline{T_i}]^{\alpha}X_i^{[n]}. \]
We have
\[ \partial_i(-t[\underline{T_i}]^{\alpha}X_i^{[n]})=t[\underline{T_i}]^{\alpha}X_i^{[n-1]}+nt[\underline{T_i}]^{\alpha}X_i^{[n]}. \]
So by induction hypothesis and by Step 4 we know that
\[  \partial_i(t[\underline{T_i}]^{\alpha}X_i^{[n]})=(\sigma_i-1)m \]
for some $m\in M^+_\infty$. By Step 3 we know that $m=\partial_i(m_1)$ for some $m_1\in M^+_\infty$, hence
\[ \partial_i(t[\underline{T_i}]^{\alpha}X_i^{[n]})=\partial_i\left( (\sigma_i-1)(m_1)\right). \]
So we deduce that
\[ t[\underline{T_i}]^{\alpha}X_i^{[n]}=(\sigma_i-1)(m_1)+m_2 \]
with $\partial_i(m_2)=0$. We claim that $m_2=(\sigma_i-1)(m_3)$. Since $n>0$ and  $\partial_i(m_2)=0$, it suffices to show that the coefficient of $X_i^{[0]}$ in $(\sigma_i-1)(m_1)$ lies in the image of $\sigma_i-1$. To see this, write $m_1=\sum_{\alpha}c_\alpha[\underline{T_i}]^{\alpha}X_i^{[n_\alpha]}$ for some $c_\alpha\in (M^+_\infty/p^n)^{\partial_i=0,\sigma_i=1}$. Then by a computation above we have
\[ (\sigma_i-1)([\underline{T_i}]^{\alpha}X_i^{[n_\alpha]})=([\underline{1}]^{\alpha +n_\alpha}-1)[\underline{T_i}]^{\alpha}X_i^{[n_\alpha]}+[\underline{1}]^{\alpha +n_\alpha}[\underline{T_i}]^{\alpha}\sum_{r=0}^{n_\alpha-1}X_i^{[r]}(1-[\underline{1}]^{-1})^{[n_\alpha-r]} \]
and in this sum the coefficient of $X_i^{[0]}$ is a multiple of $t$ if $n_\alpha\neq 0$ and otherwise it is $([\underline{1}]^{\alpha}-1)[\underline{T_i}]^{\alpha}$ which is in the image of $\sigma_i-1$. Hence the coefficient of $X_i^{[0]}$ in $(\sigma_i-1)(m_1)$ lies in the image of $\sigma_i-1$. Hence $m_2$ lies in the image of $\sigma_i-1$. 

Since the same argument also works for $t[\underline{T_1}]^{\alpha}X_i^{[n]}$ this completes the induction, and proves that for all $i$
\[ H_0(L_i\otimes_{\mathbb{Z}_p[[\Delta_\infty]]}M^+_\infty/p^nM^+_\infty) \]
is annihilated by $t$, where $L_i$ is the complex defined in (\ref{koszul}).

\underline{Step 6:} Now, consider the Koszul complex $L=\otimes_{\mathbb{Z}_p[[\Delta_\infty]]} L_i$. For any complex $K$ of $\mathbb{Z}_p[[\Delta_\infty]]$-modules we have short exact sequences (\cite[Ch. IV, Prop. 1]{localg})
\[ \minCDarrowwidth5pt\begin{CD}
0 @>>> H_0(L_i\otimes_{\mathbb{Z}_p[[\Delta_\infty]]}H_p(K)) @>>> H_p(L_i\otimes_{\mathbb{Z}_p[[\Delta_\infty]]}K) @>>> H_1(L_i\otimes_{\mathbb{Z}_p[[\Delta_\infty]]}H_{p-1}(K)) @>>> 0.
\end{CD} \]
Applying this inductively to $K=L_{\leq e}\otimes_{\mathbb{Z}_p[[\Delta_\infty]]}M^+_\infty/p^nM^+_\infty$, with $L_{\leq e}:=\otimes_{i\leq e}L_i$ the theorem follows.
\end{proof}

\subsubsection{}
We can now tie everything together. Denote $C^{\bullet}(\Delta,-)=\Hom_{\text{cont.},\Delta}(\Delta^{\times\bullet},-)$ the usual functorial complex computing continuous group cohomology of $\Delta$.

\begin{cor}\label{final}
Suppose that $R$ is a small $K^+$-algebra. Then there is a canonical morphism of the derived category
\[ B^+_{\log}\otimes_{\Sigma}\omega^{\bullet}_{\mathcal{R}_n/\Sigma_n}\to C^{\bullet}(\Delta,A^+_{\log,\Sigma}/p^nA^+_{\log,\Sigma}) \]
which is an almost quasi-isomorphism up to $t^d[\underline{c}]$-torsion.
\end{cor}
\begin{proof}
The morphism in the derived category comes from the morphisms of complexes
\[ \minCDarrowwidth15pt \begin{CD}
C^\ast(\Delta,A^+/p^nA^+) @>{\sim}>> C^\ast(\Delta,M^+/p^nM^+\otimes_{\mathcal{R}_n}\omega^\bullet_{\mathcal{R}_n/\Sigma_n}) @<<< B^+_{\log}/p^nB^+_{\log}\otimes_{\Sigma_n}\omega^\bullet_{\mathcal{R}_n/\Sigma_n}
\end{CD} \]
where the morphism of the left is a quasi-isomorphism by (\ref{resolution}). Now by Corollary \ref{almostgalois}, the result follows from Theorems \ref{invariants2} and \ref{vanishing}.
\end{proof}

\subsection{Kummer sequence compatibility}
Let $R$ be a small integral $K^+$-algebra and let $f\in R^{\ast}$. Let $\mathcal{R}_n$ be the ind-\'etale $\Theta(c)/p^n\Theta(c)$-algebra lifting $R\otimes\mathbb{F}_p$ and let $\hat{f}$ be a lift of $f\mod p$ to $\mathcal{R}_n$. The element $f$ defines, via the Kummer sequence, a 1-cocycle
\[ \left(\delta\mapsto \dfrac{\delta(f^{p^{-n}})}{f^{p^{-n}}}\right)\in C^1(\Delta,\mathbb{Z}/p^n\mathbb{Z}(1)). \]
Under the logarithm $\log:\mathbb{Z}/p^n\mathbb{Z}(1)\to A^+/p^nA^+$ it maps to
\[ \log\left( f^{p^{-n}}\right):=\left(\delta\mapsto\log\left( \dfrac{\delta(f^{p^{-n}})}{f^{p^{-n}}}\right)\right) \in C^1(\Delta,A^+/p^nA^+). \]

\begin{prop}
We have $\dlog(\hat{f})=-\log\left(f^{p^{-n}}\right)$ in $H^1(\Delta,A^+/p^nA^+)$.
\end{prop}
\begin{proof}
We have morphisms of complexes
\[ \minCDarrowwidth15pt \begin{CD}
C^\ast(\Delta,A^+/p^nA^+) @>{\sim}>> C^\ast(\Delta,M^+/p^nM^+\otimes_{\mathcal{R}_n}\omega^\bullet_{\mathcal{R}_n/\Sigma_n}) @<<< B^+_{\log}/p^nB^+_{\log}\otimes_{\Sigma_n}\omega^\bullet_{\mathcal{R}_n/\Sigma_n}
\end{CD} \]
and $\dlog(\hat{f})$ is the image under $d$ of $\log\left( \left[f^{p^{-n}}\right]^{-1}\otimes\hat{f}\right)\in C^0(\Delta,M^+/p^nM^+) $ and the latter has image $\log\left( f^{-p^{-n}}\right)=-\log\left( f^{p^{-n}}\right)\in C^1(\Delta,M^+/p^nM^+)$.
\end{proof}

\section{Appendix: results from commutative algebra}

\subsection{Complement on log structures}

\subsubsection{}
Part (i) of the next lemma is implicit in \cite{log}.

\begin{lemma}\label{nilexact}
Let $i:(X_0,M_0)\hookrightarrow (X,M)$ be a nilpotent closed immersion of log-schemes of ideal $\mathscr{I}$. Then
\begin{enumerate}[(i)]
\item $i^\ast M=i^{-1}M/(1+\mathscr{I})$; in particular, $i$ is exact if and only if the map $i^{-1}M\to M_0$ is surjective and locally for all sections $m,m'$ of $i^{-1}M$ with same image in $M_0$, there exists $u\in 1+\mathscr{I}$ such that $m=um'$
\item if $X$ is affine, $i$ exact, and $M$ integral, then $\Gamma(X,M)/\Gamma(X,1+\mathscr{I})=\Gamma(X_0,M_0)$.
\item If $i$ is exact, $M_0$ fine and saturated, and $M$ integral, then $M$ is fine and saturated and is given in a neighbourhood of a geometric point $\bar{x}\to X_0$ by the fine saturated monoid $P:=M_{0,\bar{x}}/\mathscr{O}_{X_0,\bar{x}}^{\ast}$.
\end{enumerate}
\end{lemma}
\begin{proof}
(i): One first shows easily that $\mathscr{O}_{X_0}^\ast\cong i^{-1}\mathscr{O}_X^\ast/(1+\mathscr{I})$. If $L\to\mathscr{O}_{X_0}$ is a log-structure on $X_0$ together with a morphism of pre-log-structures
\[ i^{-1}M\to L \]
then $1+\mathscr{I}\subset i^{-1}M$ maps to $1\in L$, so the map factors (necessarily uniquely)
\[ i^{-1}M/(1+\mathscr{I})\to L. \]
So the claim will follow if we can show that $i^{-1}M/(1+\mathscr{I})$ is a log-structure. If $\alpha:i^{-1}M\to i^{-1}\mathscr{O}_X$ is the inverse image by $i$ of the map defining $M$ as a log-structure on $X$, then $\alpha$ induces an isomorphism $\alpha^{-1}i^{-1}\mathscr{O}_X^\ast\cong i^{-1}\mathscr{O}_X^\ast$, whence an isomorphism
\[ \alpha^{-1}i^{-1}\mathscr{O}_X^\ast/(1+\mathscr{I})\cong i^{-1}\mathscr{O}_X^\ast/(1+\mathscr{I}) \]
i.e. $i^{-1}M/(1+\mathscr{I})\to i^{-1}\mathscr{O}_X/\mathscr{I}\cong \mathscr{O}_{X_0}$ is a log-structure on $X_0$.

(ii): I thank the referee for this argument. Let $s\in \Gamma(X_0,i^{\ast}M)$. Choose an \'etale covering $\mathscr{U}=\{ U_j\to X_0\}_{j\in J}$ such that $s|_{U_j}$ is the image of $t_j\in i^{-1}(M)(U_j)$. For all $j,k\in J$, let $U_{jk}:=U_j\times_{X_0}U_k$. Then since $M$ is integral, $\{t_j/t_k\in (1+\mathscr{I})(U_{jk})\}_{(j,k)\in J^2}$ is a one-cocycle with values in $1+\mathscr{I}$. We claim that, up to refining $\mathscr{U}$, it is a coboundary. It suffices to show that $H^1_{\et}(X,1+\mathscr{I})=0$. To see this, note we have a finite filtration of $1+\mathscr{I}$ by subgroups of the form $1+\mathscr{I}^n$, $n>1$, whose graded is isomorphic to $\mathscr{I}^n/\mathscr{I}^{n+1}$. Since the latter are quasi-coherent sheaves, their cohomology vanishes in degree at least 1, and the claim follows from this. So, after refining $\mathscr{U}$, we may assume that $t_j/t_k=u_j/u_k$ for some $u_j\in (1+\mathscr{I})(U_j)$, $j\in J$. Then $\{u_j^{-1}t_j\in i^{-1}(M)(U_j)\}_{j\in J}$ glues to give an element of $\Gamma(X_0,i^{-1}M)$ which maps to $s$. So the map $\Gamma(X,M)\to\Gamma(X_0,i^{\ast}M)$ is surjective. Now the result follows from (i).

(iii): Since $M_0$ is fine and saturated, $P^{\gp}$ is a finite free $\mathbb{Z}$-module, so there is a section $P^{\gp}\to M_{0,\bar{x}}^{\gp}$. This extends to a map $P\to M_{0,\bar{x}}\to\mathscr{O}_{X_0,\bar{x}}$ which is a chart of the log-structure $M_0$ in a neighbourhood of $\bar{x}$. By (i) it follows that $P=M_{\bar{x}}/\mathscr{O}_{X,\bar{x}}^{\ast}$, so similarly we get a map $s:P\to M_{\bar{x}}\to\mathscr{O}_{X,\bar{x}}$ which we claim defines a chart. Let $\mathcal{P}\to\mathscr{O}_X$ be the log-structure associated to $s$ in a neighbourhood of $\bar{x}$. It comes equipped with a morphism of log-structures $\phi:\mathcal{P}\to M$ and by construction $i^{\ast}\mathcal{P}\cong M_0=i^{\ast}M$. So if $m\in M$ is a section, there is $p\in P$ and $u\in\mathscr{O}_X^{\ast}$ such that $\phi(pu)$ and $m$ have same image in $M_0$. By (i) there is $v\in 1+\mathscr{I}$ such that $\phi(pu)=mv$, hence $m=\phi(puv^{-1})$ so $\phi$ is surjective. Similarly, if $\phi(p_1u_1)=\phi(p_2u_2)$, then there is $w\in 1+\mathscr{I}$ such that $p_1u_1=wp_2u_2$ in $\mathcal{P}$. So $\phi(p_1u_1)=w\phi(p_2u_2)$ in $M$, whence $w=1$ since $M$ is integral. So $p_1u_1=p_2u_2$, and $\phi$ is injective.
\end{proof}

\subsubsection{}
Let $R$ be a strictly henselian ring, $Q$ and integral monoid together with a map $\alpha_Q:Q\to R$. Let $Q^a\to R$ be the log structure on $\Spec(R)$ associated to $\alpha_Q$, i.e. $Q^a=Q\oplus R^{\ast}/\sim$, where $(q_1,u_1)\sim (q_2,u_2)$ ($(q_i,u_i)\in Q\oplus R^{\ast}$ for $i=1,2$) if and only if there are $h_1,h_2\in\alpha_Q^{-1}(R^{\ast})$ such that $h_1q_1=h_2q_2$ and $h_2u_1=h_1u_2$.

\begin{lemma}\label{surjapplogstr}
\begin{enumerate}[(i)]
\item The natural map $Q^{\gp}\to (Q^a)^{\gp}/R^{\ast}$ is surjective.
\item The natural map $Q^{\gp}\to (Q^a)^{\gp}$ is injective.
\end{enumerate}
\end{lemma}
\begin{proof}
(i): We first note that if $f:M\to N$ is a surjective map of integral monoids, then $M^{\gp}\to N^{\gp}$ is surjective. Indeed, if $n_1/n_2\in N^{\gp}$ with $n_1,n_2\in N$, then we can find $m_i\in M$ such that $f(m_i)=n_i$ for $i=1,2$, and then $n_1/n_2$ is the image of $m_1/m_2$.

Consider the monoid $Q^a/R^{\ast}$. Since $Q^a$ is an integral monoid (since $Q$ is) it follows easily that $Q^a/R^{\ast}$ is also integral, where $Q^a/R^{\ast}:=Q^a/\sim$, where $x\sim y$ if and only if there is $u\in R^{\ast}$ such that $xu=y$. Write $\bar{Q}:=Q^a/R^{\ast}$. Then the natural map $Q\to\bar{Q}$ is surjective: if $q\in\bar{Q}$ then it is the image of $(q',u)\in Q^a$ for $q'\in Q$ and $u\in R^{\ast}$; since $(q',u)=(q',1)(1,u)$ it follows the image of $(q',1)$ in $\bar{Q}$ is $q$, and so $q'\in Q$ maps to $q\in\bar{Q}$. So the map $Q^{\gp}\to (\bar{Q})^{\gp}$ is surjective.

Also, the map $(Q^a)^{\gp}\to (\bar{Q})^{\gp}$ is surjective. If $x\in (Q^a)^{\gp}$ lies in its kernel, then write $x=q_1/q_2$ for $q_i\in Q^a$ for $i=1,2$. Then $q_1$ and $q_2$ have the same image in $\bar{Q}$, so $q_1=uq_2$ for some $u\in R^{\ast}$. Hence $(Q^a)^{\gp}/R^{\ast}\cong (\bar{Q})^{\gp}$, and this proves (i).

(ii): Let $x\in\ker(Q^{\gp}\to (Q^a)^{\gp})$. Write $x=m_1/m_2$ with $m_i\in Q$ for $i=1,2$. Then $m_1$ and $m_2$ have same image in $Q^a$, i.e. $(m_1,1)\sim (m_2,1)$ for the equivalence relation described before the statement of the lemma. So there are $h_1,h_2\in\alpha_Q^{-1}(R^{\ast})$ such that $h_1m_1=h_2m_2$ and $h_2\cdot 1=h_1\cdot 1$. So $h_1=h_2$ and $h_1m_1=h_1m_2$. Since $Q$ is integral we must have $m_1=m_2$.
\end{proof}

Now let $I\subset R$ a nilpotent ideal, and $\bar{R}:=R/I$. Suppose that $Q\to M$ is a surjection of integral monoids which fits in a commutative diagram of (multiplicative) monoids
\[ \begin{CD}
Q @>>> M \\
@V{\alpha_Q}VV @V{\alpha_M}VV \\
R @>>> \bar{R}
\end{CD} \]
where the map $R\to R/I$ is the quotient map. We write $Q^a$ for the log structure on $\Spec(R)$ associated to the map $\alpha_Q$ and write ${M^a}$ for the log structure on $\Spec(\bar{R})$ associated to $\alpha_M$.

\begin{lemma}\label{app:exactlogstr}
With notation and hypothesis as above. Suppose that the group $L:=\ker(Q^{\gp}\to M^{\gp})$ consists of 1-units, i.e. the image of $L$ in $(Q^a)^{\gp}$ lies in the subgroup $1+I\subset R^{\ast}\subset Q^a\subset (Q^a)^{\gp}$ (this makes sense because $Q^a$ is a log structure). If the canonical map
\[ Q^{\gp}\cap R^{\ast} \to M^{\gp}\cap\bar{R}^{\ast} \]
is surjective, then we have an exact sequence
\[ 1\to 1+I\to (Q^a)^{\gp}\to ({M^a})^{\gp} \to 1. \]
\end{lemma}
\begin{proof}
Let $\alpha_{Q^a}:Q^a\to R$ (resp. $\alpha_{{M^a}}:{M^a}\to\bar{R}$) be the map deduced from $\alpha_Q$ (resp. $\alpha_M$). Then the map $\alpha_{Q^a}^{-1}(R^{\ast})\to R^{\ast}$ is an isomorphism and we simply write $R^{\ast}$ for $\alpha_{Q^a}^{-1}(R^{\ast})$ when this doesn't lead to confusion. Then we have a commutative diagram with exact rows
\[ \begin{CD}
1 @>>> 1+I @>>> R^{\ast} @>>> \bar{R}^{\ast} @>>> 1\\
@. @VVV @VVV @VVV @. \\
1 @>>> K @>>> (Q^a)^{\gp} @>>> ({M^a})^{\gp} @>>> 1
\end{CD} \]
in which the vertical maps are injective, so by the snake lemma we get an exact sequence
\[ 1\to K/(1+I)\to (Q^a)^{\gp}/R^{\ast}\to ({M^a})^{\gp}/\bar{R}^{\ast}\to 1. \]
It fits in a commutative diagram
\[ \begin{CD}
1 @>>> L @>>> Q^{\gp} @>>> M^{\gp} @>>> 1\\
@. @VVV @VVV @VVV @. \\
1 @>>> K/(1+I) @>>> (Q^a)^{\gp}/R^{\ast}@>>> ({M^a})^{\gp}/\bar{R}^{\ast} @>>> 1
\end{CD} \]
where the maps $Q^{\gp}\to (Q^a)^{\gp}/R^{\ast}$, $M^{\gp}\to ({M^a})^{\gp}/\bar{R}^{\ast}$ are surjective by Lemma \ref{surjapplogstr} (i). By Lemma \ref{surjapplogstr} (ii) we have $\ker(Q^{\gp}\to (Q^a)^{\gp}/R^{\ast})=Q^{\gp}\cap R^{\ast}$ (intersection taken in $(Q^a)^{\gp}$) and similarly $M^{\gp}\cap\bar{R}^{\ast}=\ker(M^{\gp}\to ({M^a})^{\gp}/\bar{R}^{\ast})$. Moreover, $\ker(L\to K/(1+I))=L$ since $L\subset 1+I$ by assumption. So by the snake lemma we get an exact sequence
\[ 1\to L\to Q^{\gp}\cap R^{\ast} \to M^{\gp}\cap\bar{R}^{\ast}\to K/(1+I)\to 1. \]
Since the map $Q^{\gp}\cap R^{\ast} \to M^{\gp}\cap\bar{R}^{\ast}$ is surjective by assumption, we have $K/(1+I)=1$, and this completes the proof.
\end{proof}

\subsection{Integrality results}
Notation as in (\ref{smalln}).
\begin{lemma}
\begin{enumerate}[(i)]
\item $O(c)_{n}$ is a regular integral domain.
\item $O(c)_{n,L}$ is an integrally closed domain.
\end{enumerate}
\end{lemma}

\begin{proof}
(i): This is clear since $O(c)_n$ has either good ($c=1$) or semi-stable ($c=\pi$) reduction.

(ii): Since $O(c)_{n,L}$ is an integral domain, it suffices to show that it is normal. Since $L_n$ is a finite extension of $K_n$ and normality is stable by \'etale localization, up to making a finite unramified extension of $K_n$ we may assume that $L_n$ is a totally ramified extension of $K_n$, in particular $L_n^{+}\cong K^+_n[X]/(f)$, where $f$ is an Eisenstein polynomial (\cite[Ch. I, \S 6, Prop. 18]{locaux}). Since $O(c)_n$ is a Krull ring, by the conjunction of \cite[Ch. 7, \S 4, no. 2, Thm. 2]{algcom} and \cite[Ch. 5, \S 1, no. 2, Prop. 8]{algcom}, we may localize at height 1 prime ideals to reduce to the case $O(c)_n$ is a discrete valuation ring. If $O(c)_n$ is of equal characteristic zero, then $O(c)_{n,L}$ is \'etale over $O(c)_n$, hence normal. If $O(c)_n$ is of mixed characteristic and $c=\pi$, then $c_n$ is a uniformizer for $O(c)_n$ and $O(c)_{n,L}\cong O(c)_n[X]/(f)$. So $O(c)_{n,L}$ is a discrete valuation ring (\cite[Ch. I, \S 6, Prop. 17]{locaux}), in particular a normal ring. If $c=1$, then it is easy to see that the special fibre of $\Spec(O(c)_n)\to\Spec(K^+_n)$ is integral, hence a uniformizer of $K^+_n$ is a uniformizer of $O(c)_n$ and $f$ is an Eisenstein polynomial, so we can conclude.
\end{proof}

Note that the normalization of $K^+_n$ in $R_n$ is unramified over $K^+_n$ because $R$ is small. Define $R_{n,L}:=R\otimes_{O(c)}O(c)_{n,L}$.

\begin{prop}\label{integral}
Assume $K\cap R=K^+$. If $R$ is an integral domain and $K^+$ is integrally closed in $R$, then $R_{n,L}$ is an integrally closed domain for all $n$ and $L$.
\end{prop}
\begin{proof}[Proof (Ramero)] Since $R_{n,L}$ is \'etale over $O(c)_{n,L}$ it follows that $R_{n,L}$ is a normal ring. So it suffices to show that $R_{n,L}$ is an integral domain.

We first show that $S:=R\otimes_{K^+}K^+_n$ is an integral domain. Note that since $K^+$ is integrally closed in $R$ we have $Q(R)\cap K_n=K$. Let $Q(R)\cdot K_n$ be the composite of $Q(R)$ and $K_n$. Now, $S\subset Q(R)\otimes_{K^+} {K_n^+}$ and the latter is a field: we have a surjective map
\[ Q(R)\otimes_{K^+} {K_n^+}\to Q(R)\cdot K_n \]
and $Q(R)\otimes_{K^+} K_n^+$ is a $Q(R)$ vector space of dimension $[K_n:K]$; since $Q(R)\cap K_n=K$, $\dim_{Q(R)}Q(R)\cdot K_n=[K_n:K]$, hence the map is an isomorphism. This proves that $S$ is an integral domain.

We now show that $R_n=R_{n,K}$ is an integral domain. It suffices to show that its spectrum is connected. Since $R_n$ is finite flat over $S$ the image of a connected component of $R_n$ under the morphism $f:\Spec(R_n)\to\Spec(S)$ is both open and closed, hence equal to $S$ because $f$ is generically finite \'etale (in particular every connected component of $R_n$ dominates $S$). So it suffices to show that $f$ has a single connected fibre. Let $\mathfrak{q}$ be a generic point of $\Spec(S/\pi S)$, considered as a prime ideal of $S$, and let $\mathfrak{p}=\mathfrak{q}\cap O$, where $O:=O(c)\otimes_{K^+}K_n^+$. Let $S_{\mathfrak{q}}^h$ resp. $O_{\mathfrak{p}}^h$ denote the henselization of $S$ at $\mathfrak{q}$ resp. $O$ at $\mathfrak{p}$. Since the prime ideals $\mathfrak{p}$ and $\mathfrak{q}$ have height one, these are discrete valuation rings. Let $O_n:=O(c)_n$. We claim that
\[ O^h_{\mathfrak{p}}\otimes_{O}O_n \]
is a noetherian henselian local ring. It suffices to show that
\[ k(\mathfrak{p})\otimes_OO_n \]
is an integral domain. In the case $c=\pi$ we have $\mathfrak{p}=(c_n,T_i)$ for some $1\leq i\leq r$, and otherwise we have $\mathfrak{p}=(\pi_n)$ where $\pi_n$ denotes a uniformizer of $K^+_n$. First consider the case $c=\pi$. Then $k(\mathfrak{p})$ is the fraction field of $k[T_1^{\pm 1},...,\widehat{T_i^{\pm 1}},...,T_{d+1}^{\pm 1}]$ and $k(\mathfrak{p})\otimes_OO_n$ is a localization of
\[ k[T_1^{\pm 1},...,\widehat{T_i^{\pm 1}},...,T_{d+1}^{\pm 1}]\otimes_OO_n\cong k\left[ T_1^{(n)},\dfrac{1}{T_1^{(n)}},...,\widehat{T_i^{(n)}},\widehat{\dfrac{1}{T_i^{(n)}}},...,T_{d+1}^{(n)},\dfrac{1}{T_{d+1}^{(n)}}\right]  \]
hence is an integral domain. As usual, here the hat over a symbol means that we omit it. This proves the claim in the case $c=\pi$ and the case $c=1$ is simpler. Note also that $O^h_{\mathfrak{p}}\otimes_{O}O_n\cong \left( O_{\mathfrak{p}}\otimes_{O}O_n\right)^h$ is normal since $O_{\mathfrak{p}}\otimes_{O}O_n$ is (\cite[IV$_{4}$, Thm. 18.6.9]{ega}). Hence it is a discrete valuation ring.

Now, the extension
\[ O_{\mathfrak{p}}^h\to S_{\mathfrak{q}}^h \]
is finite \'etale of degree $f$ (say) and the extension
\[ O_{\mathfrak{p}}^h\to O^h_{\mathfrak{p}}\otimes_{O}O_n \]
is totally ramified of degree $e$ (say). So the composite $S_{\mathfrak{q}}^h\cdot O_n$ is an extension of $O_{\mathfrak{p}}^h$ of degree $d=e\cdot f$. Hence the canonical map
\[ S_{\mathfrak{q}}^h\otimes_{O^h_{\mathfrak{p}}}O^h_{\mathfrak{p}}\otimes_OO_n\to S_{\mathfrak{q}}^h\cdot O_n\]
is a surjective map of free $O_{\mathfrak{p}}^h$-modules of same rank, so it is an isomorphism. Thus, $S_{\mathfrak{q}}^h\otimes_{O}O_n\cong S_{\mathfrak{q}}^h\cdot O_n$ is connected, and so (since $S_{\mathfrak{q}}^h$ is a henselian noetherian local ring) has connected special fibre
\[ (S_{\mathfrak{q}}^h\otimes_{O}O_n)/\mathfrak{q}\cdot (S_{\mathfrak{q}}^h\otimes_{O}O_n)\cong (S_{\mathfrak{q}}/\mathfrak{q})\otimes_O O_n. \]
This implies that the fibre of $R_n=S\otimes_O O_n$ over $\mathfrak{q}$ is connected, hence so is $R_n$.

Finally, $R_{n,L}$ is a domain, by the same dimension counting argument as for $S$. \end{proof}

\section*{Acknowledgements}
This article contains the main result of my Ph.D. thesis (Bonn 2007). I am very grateful to G. Faltings for suggesting this problem and for his patient help. I am indebted to L. Ramero for numerous comments and suggestions which led to significant improvement. I would like to thank the Max-Planck-Institut f\"ur Mathematik Bonn and the Hebrew University of Jerusalem for hospitality. Finally, I would like to express my sincere thanks to the referee for a very careful reading of the manuscript and making many critical comments and suggestions. In particular, the proof of Theorem \ref{fontaine} was corrected according to the referee's comments.

\bibliographystyle{plain}
\bibliography{bibthesis}
\end{document}